\documentclass[10pt, a4paper, reqno]{amsart}
\usepackage{amscd}
\usepackage{dsfont, pifont}

\usepackage{tikz-cd}

\usepackage{amsmath, upgreek}
\usepackage{amstext}
\usepackage{amsthm}
\usepackage{amssymb}
\usepackage{amsopn}
\usepackage{amssymb}
\usepackage{amsxtra}
\usepackage{amsfonts}
\usepackage{mathtools}
\usepackage{fancyhdr}
\usepackage{paralist, epsfig}
\usepackage[mathcal]{euscript}
\usepackage[english]{babel}
\usepackage[latin1]{inputenc}
\usepackage{bbold}
\usepackage{tkz-graph}

\makeatletter
\g@addto@macro\normalsize{%
  \setlength\abovedisplayskip{10pt}
  \setlength\belowdisplayskip{10pt}
  \setlength\abovedisplayshortskip{5pt}
  \setlength\belowdisplayshortskip{8pt}
} 

\allowdisplaybreaks

\newtheoremstyle{normal}% normale Schrift
{5pt}% hSpace abovei
{5pt}% hSpace belowi
{\normalfont}% hBody fonti
{}% hIndent amounti1
{\bfseries}% hTheorem head fonti
{}% Punctuation after theorem headi
{0.4em}% hSpace after theorem headi2
{\bfseries{\thmname{#1}\thmnumber{ #2}.\thmnote{ \hspace{0.5em}(#3)\newline}}}% hTheorem head spec (can be left empty, meaning `normal')

\newtheoremstyle{kursiv}% kursive Schrift
{5pt}% hSpace abovei
{5pt}% hSpace belowi
{\itshape}% hBody fonti
{}% hIndent amounti1
{\bfseries}% hTheorem head fonti
{}% Punctuation after theorem headi
{0.4em}% hSpace after theorem headi2
{\bfseries{\thmname{#1}\thmnumber{ #2}.\thmnote{ \hspace{0.5em}(#3)\newline}}}% hTheorem head spec (can be left empty, meaning `normal')

\theoremstyle{kursiv}

\theoremstyle{normal}

\newtheorem{thm}{Theorem}[section]
\newtheorem{ex}[thm]{Example}
\newtheorem{rem}[thm]{Remark}
\newtheorem{cor}[thm]{Corollary}
\newtheorem{lem}[thm]{Lemma}
\newtheorem{prop}[thm]{Proposition}
\newtheorem{dfn}[thm]{Definition}
\newtheorem{ass}[thm]{Assumption}

\renewcommand{\epsilon}{\varepsilon}

\newcommand{\id}{\operatorname{id}\nolimits}

\renewcommand{\Re}{\operatorname{Re}\nolimits}
\renewcommand{\theta}{\uptheta}

\newcommand{\pr}{\operatorname{pr}\nolimits}
\newcommand{\gr}{\operatorname{Graph}\nolimits}
\newcommand{\diam}{\operatorname{diam}}

\newcommand{\ran}{\operatorname{Range}\nolimits}

\newcommand{\Cnull}{\operatorname{C_0}\nolimits}
\newcommand{\diag}{\operatorname{diag}\nolimits}

\usepackage{setspace}
\usepackage{color}
\newcommand{\red}[1]{\textcolor{black}{#1}}

\definecolor{grey}{gray}{.3}

\newcommand{\Ls}{\operatorname{L}}

\newcommand{\dd}{\mathrm{d}}

\makeatletter
\providecommand*{\cupdot}{%
  \mathbin{%
    \mathpalette\@cupdot{}%
  }%
}
\newcommand*{\@cupdot}[2]{%
  \ooalign{%
    $\m@th#1\cup$\cr
    \sbox0{$#1\cup$}%
    \dimen@=\ht0 %
    \sbox0{$\m@th#1\cdot$}%
    \advance\dimen@ by -\ht0 %
    \dimen@=.5\dimen@
    \hidewidth\raise\dimen@\box0\hidewidth
  }%
}

\providecommand*{\bigcupdot}{%
  \mathop{%
    \vphantom{\bigcup}%
    \mathpalette\@bigcupdot{}%
  }%
}
\newcommand*{\@bigcupdot}[2]{%
  \ooalign{%
    $\m@th#1\bigcup$\cr
    \sbox0{$#1\bigcup$}%
    \dimen@=\ht0 %
    \advance\dimen@ by -\dp0 %
    \sbox0{\scalebox{2}{$\m@th#1\cdot$}}%
    \advance\dimen@ by -\ht0 %
    \dimen@=.5\dimen@
    \hidewidth\raise\dimen@\box0\hidewidth
  }%
}
\makeatother

\usepackage[indentafter]{titlesec}
\titleformat{name=\section}{}{\thetitle.}{0.8em}{\centering\scshape}
\titleformat{name=\subsection}{}{\thetitle.}{0.8em}{\centering\scshape}

\makeatletter
\renewcommand{\tocsection}[3]{%
  \indentlabel{\@ifnotempty{#2}{\bfseries\ignorespaces#1 #2\quad}}\bfseries#3}
\renewcommand{\tocsubsection}[3]{%
  \indentlabel{\@ifnotempty{#2}{\ignorespaces#1 #2\quad}}#3}
\newcommand\@dotsep{4.5}
\def\@tocline#1#2#3#4#5#6#7{\relax
  \ifnum #1>\c@tocdepth % then omit
  \else
    \par \addpenalty\@secpenalty\addvspace{#2}%
    \begingroup \hyphenpenalty\@M
    \@ifempty{#4}{%
      \@tempdima\csname r@tocindent\number#1\endcsname\relax
    }{%
      \@tempdima#4\relax
    }%
    \parindent\z@ \leftskip#3\relax \advance\leftskip\@tempdima\relax
    \rightskip\@pnumwidth plus1em \parfillskip-\@pnumwidth
    #5\leavevmode\hskip-\@tempdima{#6}\nobreak
    \leaders\hbox{$\m@th\mkern \@dotsep mu\hbox{.}\mkern \@dotsep mu$}\hfill
    \nobreak
    \hbox to\@pnumwidth{\@tocpagenum{\ifnum#1=1\fi#7}}\par% <-- \bfseries for \section page
    \nobreak
    \endgroup
  \fi}
\AtBeginDocument{%
\expandafter\renewcommand\csname r@tocindent0\endcsname{0pt}
}
\def\l@section{\@tocline{2}{0pt}{0pc}{1.5pc}{}}
\def\l@subsection{\@tocline{2}{0pt}{2.0pc}{2.0pc}{}}
\makeatother

\begin{document}

\allowdisplaybreaks

\title{Dissipative extensions  %of operators with unequal deficiency indices 
and port-hamiltonian\\operators on networks}

\author{Marcus Waurick\hspace{0.5pt}\MakeLowercase{$^{\text{1}}$} and Sven-Ake Wegner\hspace{0.5pt}\MakeLowercase{$^{\text{2}}$}}

\renewcommand{\thefootnote}{}
\hspace{-1000pt}\footnote{\hspace{5.5pt}2010 \emph{Mathematics Subject Classification}: Primary 93D15; Secondary 47D06.\vspace{1.6pt}}

\hspace{-1000pt}\footnote{\hspace{5.4pt}\emph{Key words and phrases}: $\Cnull$-semigroup, contraction semigroup, hyperpolic pde, port-Hamiltonian, well-posedness, dif-\linebreak\phantom{x}\hspace{1.2pt}ferential equations on networks, boundary system, boundary triplet, boundary triple, quantum graph. \vspace{1.6pt}}

\hspace{-1000pt}\footnote{\hspace{0pt}$^{1}$\,Corresponding author: University of Strathclyde, Department of Mathematics and Statistics, Livingstone Tower, 26\linebreak\phantom{x}\hspace{1.2pt}Richmond Street, Glasgow G1\:1XH, Scotland, Phone: +44\hspace{1.2pt}(0)\hspace{1.2pt}141\hspace{1.2pt}/\hspace{1.2pt}548\hspace{1.2pt}-\hspace{1.2pt}3817, E-Mail: marcus.waurick@strath.ac.uk.\vspace{1.6pt}}

\hspace{-1000pt}\footnote{\hspace{0pt}$^{2}$\,Teesside University, School of Science, Engineering and Design, Southfield Road, Middlesbrough, TS1\;3BX, United\linebreak\phantom{x}\hspace{1.2pt}Kingdom, phone: +44\,(0)\,1642\:73\:82\:00, e-mail: s.wegner@tees.ac.uk.\vspace{1.6pt}}

%\linebreak\phantom{x}\hspace{1.2pt}

\begin{abstract}
In this article we study port-Hamiltonian partial differential equations on certain one-dimensional manifolds. We classify those boundary conditions that give rise to contraction semigroups. As an application we study port-Hamiltonian operators on networks whose edges can have finite or infinite length. In particular, we discuss possibly infinite networks in which the edge lengths can accumulate zero and port-Hamiltonian operators with Hamiltonians that neither are bounded nor bounded away from zero. We achieve this, by first providing a new description for maximal dissipative extensions of skew-symmetric operators. The main technical tool used for this is the notion of boundary systems. The latter generalizes the classical notion of boundary triple(t)s and allows to treat skew-symmetric operators with unequal deficiency indices. In order to deal with fairly general variable coefficients, we develop a theory of possibly unbounded, non-negative, injective weights on an abstract Hilbert space.
\end{abstract}

\maketitle

\tableofcontents %Introduce Table Of Contents; too messy without.

%\newpage 
%%%%%%%%%%%%%%%%%%%%%%%%%%%%%%%%%%%%%%%%%%%%%%%%%
%%%%%%%%%%%%%%%%%%%%%%%%%%%%%%%%%%%%%%%%%%%%%%%%%
%%                                             %%
%% 1 Introduction                              %%
%%                                             %%
%%%%%%%%%%%%%%%%%%%%%%%%%%%%%%%%%%%%%%%%%%%%%%%%%
%%%%%%%%%%%%%%%%%%%%%%%%%%%%%%%%%%%%%%%%%%%%%%%%%

\section{Introduction}\label{SEC:1}\smallskip

The subject of differential operators on one-dimensional manifolds and their boundary conditions is a very active area of research. The main question is to relate boundary conditions on well-understood boundary data spaces to properties of the differential operator defined on a suitable orthogonal sum of $\Ls^2$-spaces. For this a great deal of research has been devoted to the second derivative operator, or variants thereof, and boundary conditions leading to self-adjoint realizations of these operators. 

\smallskip

Self-adjoint realizations and results on the spectrum, see e.g.\ Exner, Kostenko, Malamud, Neidhardt \cite{Exner2018} and the references therein, for spectral properties of second derivative operators, allow to conclude dynamic properties of evolution equations on the one-dimensional manifold. Due to their applications in mathematical physics a pair consisting of a manifold and a differential operator is also said to be a \textit{quantum graph}. In this context also the first derivative operator together with boundary conditions leading to skew-self-adjoint realizations is of interest---for instance to understand one-dimensional model cases of the Dirac equation, see e.g.~Carlone, Malamud, Posilicano \cite{Carlone2013}. 

\smallskip

The class of port-Hamiltonian equations forms a general framework that covers as special cases for instance the transport equation, the wave equation or the Euler-Bernoulli beam equation. More precisely, a port-Hamiltonian differential equation is of the form
\begin{equation}\label{EQ-PHS-INTRO}
\frac{\partial x}{\partial t}(\xi,t)=P_1\frac{\partial}{\partial \xi}(\mathcal{H}(\xi)x(\xi,t))+P_0\hspace{0.5pt}\mathcal{H}(\xi)x(\xi,t)
\end{equation}
where $P_1\in\mathbb{C}^{d\times d}$ is Hermitian and invertible, $P_0\in\mathbb{C}^{d\times d}$ is arbitrary and $\mathcal{H}\colon[0,\infty)\rightarrow\mathbb{C}^{d\times d}$, the Hamiltonian or Hamiltonian density matrix, is measurable and Hermitian almost everywhere---plus regularity and boundedness assumptions that vary throughout the literature. These operators together with their boundary conditions have been studied mainly in the case of finite or infinite intervals in that the question if the port-Hamiltonian operator
\begin{equation}\label{EQ-PHS-INTRO-OP}
Ax:=P_1(\mathcal{H}x)'+P_0\mathcal{H}x,
\end{equation}
defined on a possible weighted $\Ls^2$-space and endowed with boundary conditions encoded in its domain,  generates a $\Cnull$-semigroup. This question has been addressed, e.g., by Augner \cite{A, Augner2018}, Augner, Jacob \cite{AJ2014}, Le Gorrec, Zwart, Maschke \cite{GZM2005}, Jacob, Kaiser \cite{JK}, Jacob, Morris, Zwart \cite{JMZ2015}, Jacob, Wegner \cite{JW2018}, Jacob, Zwart \cite{JZ}, Villegas \cite{V2007}. Zwart, Le Gorrec, Maschke, Villegas \cite{ZGMV2010} to mention only a sample. For more historical information---and in particular for application and methods related to port-Hamiltonian systems in systems theory---we refer to the book \cite{JZ} and the references therein. 

\smallskip

The next challenge is to consider port-Hamiltonian operators on networks. Also here, results are available. In \cite[Example 3.3]{JMZ2015} a generation results is applied to the transport equation on a finite network. In \cite[Section 5]{JK} this is extended to an infinite line graph and to an infinite binary tree, both with edges being unit intervals. In \cite[Example 6.2]{JW2018} again a finite network is treated but with all edges being semi-axis'. A result for more general graphs and port-Hamiltonian equations other than the case of transport seems not to be available so far. The main aim of the present article is to provide the latter and thus to characterize all port-Hamiltonian operators that generate a $\Cnull$-semigroup of contractions on networks as general as possible and with assumptions on the Hamiltonian as general as possible.

\smallskip

We aim to characterize the aforementioned operators by means of boundary conditions. The main technical tool for the latter will be the notion of boundary systems, see Schubert, Seifert, Voigt, Waurick \cite{BS}. Boundary systems allow to describe extensions of (skew-)symmetric operators defined on a `large' Hilbert space of functions via a `small' Hilbert space of boundary values. In \cite{BTBS} the authors have shown that  boundary systems can be used in particular for operators with unequal deficiency indices. In this they supersede the concept of boundary triple(t)s$^{\,3}$\footnote{$^3$\,There are two concurrent variants of spelling in the literature: \emph{boundary triplet} and \emph{boundary triple}. Mathematically\linebreak\phantom{x}\hspace{1.2pt}both notions coincide. Although the first variant seems to be much more popular we will stick in this article to the\linebreak\phantom{x}\hspace{1.2pt}second. The latter seems grammatically to be more convincing, since a boundary triple indeed consists of three objects,\linebreak\phantom{x}\hspace{1.2pt}a space and two maps.}. Notice that we \emph{do not} enlarge the Hilbert space on which the (skew-)symmetric operator was given initially and thus provide an `intrinsic' extension theory for (skew-)symmetric operators. In this article we are interested in generators of contraction semigroups. By the Lumer--Phillips theorem, see Phillips \cite{Phillips1959}, this amounts to finding maximal dissipative extensions of given operators. Using the elementary fact that an operator is skew-self-adjoint if and only if both the operator as well as its negative are maximal dissipative, see e.g.\ \cite[Proposition 4.5]{Waurick2017}, we can automatically characterize whether there exist skew-self-adjoint extensions of a given skew-symmetric operator and how the respective boundary conditions can be described.

\smallskip

As we mentioned above, we will apply our abstract findings to networks with unbounded coefficient operator and edge lengths that are not necessarily uniformly bounded away from zero. Unbounded coefficients and/or edge lengths having zero as a accumulation point form the most challenging issues in the description of differential operators on graphs. We refer to the concluding section in Lenz, Schubert, Veseli\'{c} \cite{LSV2014} for the particular issues and problems arising with arbitrarily small edge lengths for the Laplacian on networks. See also Gernand, Trunk \cite{Gernandt2018} and the references therein. In the case of Dirac operators, we refer to Carlone, Malamud, Posilicano \cite{Carlone2013} and the references therein. In the latter article, the authors employ the machinery of boundary triples---which we deliberately want to avoid in order to conveniently accommodate for edges with infinite length. For scalar Dirac operators with potentially unequal deficiency indices we refer to Schubert, Seifert, Voigt, Waurick \cite{BS}.

\smallskip

In the classical case of edge lengths being greater than a strictly positive number and bounded coefficients, we will in this contribution establish a boundary system for the port-Hamiltonian operator, i.e., for a matrix-valued first order differential operator with a zero-th order perturbation. Thus, we complement available results for both the scalar Dirac operator as well as the Laplacian. Substituting $v=u'$ in the equation $-u''=f$ and thus writing
\[
    \mathrm{i}\begin{bmatrix} 0 & -1 \\ 1 & 0 \end{bmatrix} \begin{bmatrix} u' \\ v' \end{bmatrix} +  \begin{bmatrix} 0 & 0\\ 0& \mathrm{i}\end{bmatrix} \begin{bmatrix} u\\ v \end{bmatrix} = \mathrm{i}\begin{bmatrix} f\\ 0\end{bmatrix} 
\]
we obtain a formally equivalent first order system. The operator induced by the expression
\[
  \begin{bmatrix} u \\ v \end{bmatrix} \mapsto   \mathrm{i}\begin{bmatrix} 0 & -1 \\ 1 & 0 \end{bmatrix} \begin{bmatrix} u' \\ v' \end{bmatrix} +  \begin{bmatrix} 0 & 0\\ 0& \mathrm{i}\end{bmatrix} \begin{bmatrix} u\\ v \end{bmatrix} 
\]
is a permitted choice of the port-Hamiltonian operators discussed here. Hence, the results also entail information on the Laplacian on graphs. In the following, we will, however, focus on the general form of port-Hamiltonian operators specified in \eqref{EQ-PHS-INTRO} and \eqref{EQ-PHS-INTRO-OP}.

\smallskip

We outline the plan of this paper. In the following section, we present and prove the new abstract characterization result for maximal dissipative extensions of skew-symmetric operators. This is contained in Theorem \ref{CLASS-THM}. The technique has been used in many variants since the 1990s, but always in the context of boundary triples, see Wegner \cite{BT} for a streamlined exposition and historical information. Section \ref{sec:Minty} provides a criterion for elements of an underlying Hilbert space to be contained in the domain of a maximal dissipative operator. This criterion, Proposition \ref{prop:mdr}, is well-known and in fact straightforward in the context of maximal monotone \emph{relations}. One possible way of proving Proposition \ref{prop:mdr} is to show that linear densely defined maximal monotone operators $H$ lead to maximal monotone relations $-H$. For this, one step would be to show that $1-H$ is onto and then to apply Minty's  celebrated theorem \cite{Minty}. The aim of Section \ref{sec:Minty} is to provide an independent proof of Proposition \ref{prop:mdr} as the latter is interesting in its own already for linear operators. To complete Section \ref{sec:Minty} we provide an independent proof, based on our previous work, of Minty's theorem for linear operators, see Theorem \ref{thm:Minty}.

\smallskip

Section \ref{sec:mdo} is concerned with weighted Hilbert spaces and how computing the adjoint of an operator $H$ in an unweighted space relates to the adjoint of $H\mathcal{H}$ in a Hilbert space with scalar product induced by $\langle \cdot,\mathcal{H}\cdot\rangle$. We shall also look into dissipative and maximal dissipative operators in the weighted and unweighted situation. The reason for looking at $H\mathcal{H}$ in a weighted space is that $H$ and $\mathcal{H}$ do not commute. To set the stage, we shall recall well-known results and techniques in Subsection \ref{SEC:4:1}. We mention in passing that these weighted scalar products for strictly positive definite and bounded $\mathcal{H}$ have been applied to equations in mathematical physics in order to deal with variable coefficients in a convenient manner, see e.g.~Picard, McGhee \cite{Picard2011}. These applications to other equations motivated us to provide a small theory of operators in weighted spaces and the adjoints thereof. In fact, we hope that the rationale developed in Subsection \ref{sec:mdo2} will turn out useful to find the proper functional analytic setting for divergence form equations with highly singular variable coefficients in the future. Two of  the main results of this section are Theorem \ref{thm:lfw} and Corollary \ref{cor:essX} concerning maximal dissipative and skew-self-adjoint operators comparing the case of weighted and unweighted Hilbert spaces. The third major result is Theorem \ref{thm:adj}, where the adjoint in the weighted space is compared to the adjoint in an unweighted space. Concerning the latter only its trivial consequence Corollary \ref{thm:adj-easy-cor}---to the best of the authors' knowledge---has been known already.

\smallskip

In Section \ref{SEC:3} we gather the basic definitions and results needed for our operator-theoretic approach to port-Hamiltonian systems. We particularly refer to a weighted analogue of Barb{\u a}lat's lemma, see Farkas, Wegner \cite[Theorem 5]{FW}, which proves to be important for port-Hamiltonian operators on the semi-axis. It is then the semi-axis, that we shall treat as our first example and indeed the port-Hamiltonian operator defined on a semi-axis is a prototype example for the application of the notion of boundary systems. In Subsection \ref{SEC:4} we establish a respective boundary system and thus provide an explicit description of all port-Hamiltonian operators being generators of a contraction semigroup. It turns out that the boundary conditions at zero decide on whether or not the port-Hamiltonian operator does generate a contraction semigroup. The subsequent Subsection \ref{SEC:positive-edge} is devoted to establish a boundary system for a port-Hamiltonian operator for networks with edge lengths not tending to zero and bounded Hamiltonian.

\smallskip

We enter the technically more involved issues of unbounded coefficients and/or arbitrarily small edge lengths in Section \ref{sec:phsII}. The reason, why almost all contributions so far restricted themselves to networks in which the edge lengths have a positive lower bound, is the following: In this case the point evaluation becomes a bounded operator from the Hilbert space that describes the port-Hamiltonian operator over the network into the Hilbert space that describes the boundary data. Of course the former Hilbert space is understood to be endowed with the graph norm. If this condition is violated, the point evaluation is necessarily unbounded. In fact, the boundary system that has been established in the section before cannot be used anymore precisely for this reason. If the edge lengths accumulate zero, the trace map, i.e., evaluation at the boundary, is not bounded operator anymore, and this is needed to get a boundary system. We refer also to Gernandt, Trunk \cite{Gernandt2018}, where boundary triples are used and the so-called $M$-function is an unbounded operator. In Subsection \ref{SEC:case-by-case} we present a first workaround for the aforementioned situation. Indeed, it is possible to use the boundary system from Subsection \ref{SEC:positive-edge} on any subgraph which has uniformly positive edge lengths. Under certain conditions, see Theorem \ref{super-thm2}, the results on these subgraphs can be put together and yield a sufficient conditions for maximal dissipative extensions on the inititial graph. In Subsection \ref{SEC:arbitrary-edge}, we recall the construction of a \emph{canonical} boundary system for any given skew-symmetric operator \cite{BTBS}. The advantage using the canonical boundary system is that point evaluation is not needed and, thus, unbounded boundary operators can be avoided as the canonical boundary system uses volume sources to encode the boundary conditions. Note that using volume sources instead of evaluations at the boundary is very useful to detour regularity issues at the boundary for certain partial differential equations, see Picard, Trostdorff, Waurick \cite{PTW16_IM} and Trostdorff \cite{Trostorff2011}. One therefore might interpret the methodology discussed here as a way to avoid regularity problems with the trace map.  In Theorem \ref{super-thm-2} we present the characterization of all maximal dissiptative extensions of a given skew-symmetric port-Hamiltonian operator on \emph{any} network---with coefficients that are allowed to be both unbounded and not uniformly bounded away from zero. Section \ref{sec:phsII} is concluded with Subsection \ref{SEC:model-network}, where we consider port-Hamiltonian operators without zeroth order term and with constant coefficients but with arbitrarily small edges. We employ Theorem \ref{super-thm-2} to associate with the boundary system for the operator on a network with arbitrary edge lengths a boundary system for an operator on a network with all edges being of length one. With this transformation, we can then again interpret any boundary condition on the complicated network with unbounded trace map via a simpler network that allows for a bounded trace map.

\smallskip

We conclude this article  in Section \ref{SEC:Examples} with a couple of explicit examples for port-Hamiltonian operators on concrete infinite graphs to illustrate our abstract findings.

%%%%%%%%%%%%%%%%%%%%%%%%%%%%%%%%%%%%%%%%%%%%%%%%%
%%%%%%%%%%%%%%%%%%%%%%%%%%%%%%%%%%%%%%%%%%%%%%%%%
%%                                             %%
%% 2 Dissipative extensions                    %%
%%                                             %%
%%%%%%%%%%%%%%%%%%%%%%%%%%%%%%%%%%%%%%%%%%%%%%%%%
%%%%%%%%%%%%%%%%%%%%%%%%%%%%%%%%%%%%%%%%%%%%%%%%%

\section{Dissipative extensions via boundary systems}\label{SEC:2}

\smallskip

Let $X$ be a Hilbert space and let $H_0\colon D(H_0)\subseteq{}X\rightarrow{}X$ be a densely defined, closed, skew-symmetric operator. In particular we have $H_0\subseteq -H_0^\star$. A boundary system, cf.~Schubert et al.~\cite{BS}, for $H_0$ is a quintuplet $(\Omega, \mathcal{G}_1,\mathcal{G}_2, F, \omega)$, consisting of two Hilbert spaces $\mathcal{G}_1$, $\mathcal{G}_2$, two sequilinear forms $\Omega\colon \mathcal{H}\oplus\mathcal{H}\times\mathcal{H}\oplus\mathcal{H} \rightarrow\mathbb{C}$, $\omega\colon\mathcal{G}_1\oplus \mathcal{G}_2\times\mathcal{G}_1\oplus \mathcal{G}_2\rightarrow\mathbb{C}$ and a linear and surjective map $F\colon \gr(H_0^\star)\to \mathcal{G}_1\oplus \mathcal{G}_2$ such that
$$
\Omega((x,H_0^\star x),(y,H_0^\star y))=\omega(F(x,H_0^\star x),F(y,H_0^\star y))
$$
holds for all $x$, $y\in D(H_0^\star)$. In \cite{BTBS} the authors established that for \emph{any} skew-symmetric operator there exists a canonical boundary system, where $\Omega$ is the standard symmetric and $\omega$ the standard unitary form, that is,
$$
\Omega((x,y),(u,v))=\langle{}x,v\rangle{}_X+\langle{}y,u\rangle{}_X \;\text{ and }\; \omega((x,y),(u,v))=\langle{}x,u\rangle{}_{\mathcal{G}_1}-\langle{}y,v\rangle{}_{\mathcal{G}_2}
$$
holds for $x,y,u,v\in\mathcal{H}$, resp.~for $x,u\in\mathcal{G}_1$, $y,v\in\mathcal{G}_2$. We denote by $F_i\colon D(H_0^{\star})\rightarrow \mathcal{G}_i$ the maps given by $F_ix=\pr_i(F(x,H_0^{\star}x))$ for $i\in\{1,2\}$, where $\pr_i\colon\mathcal{G}_1\oplus\mathcal{G}_2\rightarrow\mathcal{G}_i$ denotes the canonical projection.

\medskip

The following theorem classifies all maximal dissipative extensions of $-H_0$ using the boundary system. We recall that $H\colon D(H)\subseteq X\to X$ is called maximal dissipative in a Hilbert space $X$, if for all dissipative operators $K$ extending $H$ we have that $K=H$. An operator $H$ is dissipative, if $\Re\langle Hx,x\rangle\leqslant 0$ for all $x\in D(H)$. We recall moreover that a dissipative extension of $-H_0$ is automatically a restriction of $H_0^{\star}$, cf.~Wegner \cite[Proposition 2.8]{BT}. The method of proof of the following result is akin to the proof of \cite[Theorem 4.2]{BT}.

\medskip

\begin{thm}\label{CLASS-THM} Let $H_0$ be as above and let $(\Omega,\mathcal{G}_1,\mathcal{G}_2,F,\omega)$ be a boundary system for $H_0$ in which $\Omega$ is the standard symmetric form and $\omega$ is the standard unitary form. Then $H$ is a maximal dissipative extension of $-H_0$ and restriction of $H_0^\star$, if, and only if, there exists a contraction $T\colon \mathcal{G}_2\to \mathcal{G}_1$ such that\[
D(H)= \{ x\in D(H_0^\star)\:;\:F_1(x)=TF_2(x)\}.
\]
\end{thm}
\begin{proof} \textquotedblleft{}$\Rightarrow$\textquotedblright{} Let $H$ be a maximal dissipative extension of $-H_0$. By \cite[Proposition 2.8]{BT} it follows $H\subseteq H_0^{\star}$. We thus can compute
\begin{eqnarray*}
0 &\geqslant & 2\Re \langle x,Hx\rangle_{\mathcal{H}}\\
  & = & \Omega((x,H_0^{\star}x),(x,H_0^{\star}x))\\
  & = & \omega(F(x,H_0^{\star}x),F(x,H_0^{\star}x))\\
  & = & \langle{}F_1(x),F_1(x)\rangle{}_{\mathcal{G}_1}-\langle{}F_2(x),F_2(x)\rangle{}_{\mathcal{G}_2}
\end{eqnarray*}
for all $x\in D(H)$. Hence, for all $x\in D(H)$ we have
\begin{equation}\label{eq:T}
\|F_1(x)\|_{\mathcal{G}_1}\leqslant \|F_2(x)\|_{\mathcal{G}_2}.
\end{equation}
We define $T\colon F_2[D(H)]\subseteq\mathcal{G}_2\to \mathcal{G}_1$ via $F_2(x)\mapsto F_1(x)$ for all $x\in D(H)$, which is well-defined by \eqref{eq:T}. Since $H$ and $F$ are linear we get that $T$ is linear. From \eqref{eq:T} it follows that $T$ is a contraction. Since $T$ is in particular continuous and $\mathcal{G}_1$ is complete, we can extend $T$ from $F_2[D(H)]$ to $\overline{F_2[D(H)]}$. Finally, we put $Tx:= 0$ for all $x\in F_2[D(H)]^{\bot_{\mathcal{G}_2}}$. Thus, $T$ is a contraction and defined on the whole of $\mathcal{G}_2$. Note that
$$
D(H)\subseteq \{ x\in D(H_0^\star)\:;\:F_1(x)=TF_2(x)\} =: \tilde{D}
$$
holds and that $\tilde{H}:= H_0^\star|_{\tilde{D}}$ is dissipative. Thus, $H\subseteq\tilde{H}$ is a dissipative extension, which implies $\tilde H=H$ by the maximality of $H$. 

\medskip

\textquotedblleft{}$\Leftarrow$\textquotedblright{} For the other direction, let $T\colon \mathcal{G}_{2}\to \mathcal{G}_{1}$ be a contraction. Then $H\colon D(H)\subseteq X\rightarrow X$ given by
$$
Hx=H_0^{\star}x \;\text{ for }\; x\in D(H)= \{ y\in D(H_0^\star) \:;\: F_1(y)=TF_2(y)\}
$$
is dissipative since
\begin{eqnarray*}
2\Re\langle{}x,Hx\rangle{}_X & = & \langle{}x,H_0^{\star}x\rangle{}_X+ \langle{}H_0^{\star}x,x\rangle{}_X\\
& = & \langle{}F_1(x),F_1(x)\rangle{}_{\mathcal{G}_1}  - \langle{}F_2(x),F_2(x)\rangle{}_{\mathcal{G}_2}\\
& = & \|TF_2(x)\|_{\mathcal{G}_1}^{2} - \|F_1(x)\|_{\mathcal{G}_2}^{2} \leqslant 0
\end{eqnarray*}
holds for each $x\in D(H)$. It remains to show that $H$ is maximal dissipative. Let $H\subset\tilde{H}\subseteq H_0^\star$ be a proper extension of $H$. Let $x \in D(\tilde{H})\setminus D(H)$. By the surjectivity of $F$ we find $y\in D(H_0^{\star})$ such that $F(y)=(TF_2(x),F_2(x))$. This means $F_1(y)=TF_2(x)$ and $F_2(y)=F_2(x)$. It follows $F_1(y)=TF_2(x)=TF_2(y)$ and therefore $y$ belongs to $D(H)$. Since $x\notin D(H)$, we obtain $F_1(x)\neq F_1(y)$. Since $\tilde{H}$ and $F$ are linear, we infer $0\neq x-y\in D(\tilde{H})$ and $F_2(x-y)=0$. We compute
\begin{eqnarray*}
\Re \langle(x-y),H_0^\star(x-y)\rangle_{\mathcal{H}} & = & \Omega\left(((x-y),H_0^\star(x-y)),((x-y),H_0^\star(x-y))\right)\\
& = & \omega(F(((x-y),H_0^\star(x-y)),F((x-y),H_0^\star(x-y)))\\
& = & \left\langle{}F_1(x-y),F_1(x-y)\right\rangle{}_{\mathcal{G}_1}-\left\langle{}F_2(x-y),F_2(x-y)\right\rangle{}_{\mathcal{G}_2}\\
& = & \left\langle{}F_1(x-y),F_1(x-y\rangle{}\right\rangle{}_{\mathcal{G}_1}>0
\end{eqnarray*}
which shows that $\tilde{H}$ is not dissipative. Hence, $H$ is maximal dissipative.
\end{proof}

\bigskip

%%%%%%%%%%%%%%%%%%%%%%%%%%%%%%%%%%%%%%%%%%%%%%%%%%%%%%%%%%%%%%%%%%%%%%%%%%%%%%%%%%%%%%%%%%%%%
%%%%%%%%%%%%%%%%%%%%%%%%%%%%%%%%%%%%%%%%%%%%%%%%%%%%%%%%%%%%%%%%%%%%%%%%%%%%%%%%%%%%%%%%%%%%%
%%                                                                                         %%
%% 3 Minty                                                                                 %%                                                                                          %%                                                                                         %%
%%%%%%%%%%%%%%%%%%%%%%%%%%%%%%%%%%%%%%%%%%%%%%%%%%%%%%%%%%%%%%%%%%%%%%%%%%%%%%%%%%%%%%%%%%%%%
%%%%%%%%%%%%%%%%%%%%%%%%%%%%%%%%%%%%%%%%%%%%%%%%%%%%%%%%%%%%%%%%%%%%%%%%%%%%%%%%%%%%%%%%%%%%%

\section{Minty's Theorem}\label{sec:Minty}

In the next chapter we will study maximal dissipative operators on weighted Hilbert spaces. For this we need Propostion \ref{prop:mdr}. We emphasize that we bypass in our proofs Minty's \cite{Minty} celebrated characterization of maximal monotone relations in Hilbert spaces, see also Trostorff \cite[Theorem 2.3]{Trostorff2014}. This is because we want to avoid relations in this article. We shall, however, conclude this section with a proof of Minty's theorem in the present context, see Theorem \ref{thm:Minty}. 

\smallskip

We point out that Proposition \ref{prop:mdr} is of its own interest due to the following fact: In a nutshell, it says that densely defined maximal dissipative \emph{operators} are maximal monotone \emph{relations}.

\smallskip

\begin{prop}\label{prop:mdr} Let $X$ be a Hilbert space and $H\colon D(H)\subseteq X\to X$ be a linear, densely defined and maximal dissipative operator. Let $(x,y)\in X\times X$ be given. Assume that $\Re \langle x-u,y-Hu\rangle \leqslant 0$ holds for all $u\in D(H)$. Then $x\in D(H)$ and we have $y=Hx$.
\end{prop}

\red{\begin{rem} Note that in Proposition \ref{prop:mdr}, the condition that $H$ is densely defined cannot be dropped. Indeed, by \cite[p 201]{Phillips1959}, there exists a linear, maximal dissipative operator, which is not closed. Moreover, the condition
\[
\forall (x,y)\in X\times X\colon \bigl[\,\forall\; u\in D(H): \Re \langle x-u,y-Hu\rangle \leqslant 0\;\Longrightarrow  x\in D(H) \text{ and } y=Hx\,\bigr]
\]
for some linear, dissipative operator $H\colon D(H)\subseteq X\to X$ in some Hilbert space $X$ implies that ${H}$ is closed. This however contradicts the existence of the counterexample in \cite[p.\ 201]{Phillips1959}.
\end{rem}}

\smallskip

The proof of Proposition~\ref{prop:mdr} requires the following two lemmas. Lemma \ref{lem:mdr1} was formulated explicitly by Beyer \cite[Theorem 4.2.5]{Beyer2007} in a slightly different notation. For the convenience of the reader we provide both proofs.

\smallskip

\begin{lem} \label{lem:mdr1} Let $X$ be a Hilbert space and $H\colon D(H)\subseteq X\to X$ be a linear and densely defined operator. If $H$ is dissipative, then $H$ is closable. If $H$ is maximal dissipative, then $H$ is closed.
\end{lem}
\begin{proof} For the first part we assume that $H$ is not closable. Then there exists a sequence $(\phi_n)_{n\in\mathbb{N}}$ in $D(H)$ tending to $0$ with $((1-H)\phi_n)_{n\in\mathbb{N}}$ converging to some non-zero $\psi\in X$. W.l.o.g.~we may assume that $\|\psi\|=1$ holds. Since $D(H)$ is dense in $X$, we find $\zeta\in D(H)$ such that $\|\psi-\zeta\|<1/2$. Thus, $\|\zeta\|>1/2$. For $\beta>0$ and $n\in \mathbb{N}$ we compute
\[
  (\beta+1)\| \zeta - \frac{1}\beta \phi_n\| \leqslant \| \zeta - \frac1\beta \phi_n +\beta (1-H)(\zeta -\frac1\beta\phi_n) \| = \| \zeta -\frac1\beta\phi_n + \beta (1-H) \zeta - (1-H)\phi_n\|.
\]
where we used that $H$ is dissipative. Letting $n\to\infty$ and then $\beta\to 0$ yields
\[
      \frac{1}{2}<\|\zeta\|\leqslant \| \zeta -\psi \| < \frac12
\]
and therefore a contradiction. It follows that $H$ is closable. For the second part we firstly observe that the closure $\overline{H}$ of a dissipative operator $H$ is also dissipative. Indeed, for $x\in D(\overline{H})$ let $(x_n)_{n\in\mathbb{N}}$ in $D(H)$ be such that $x_n\rightarrow x$ and $Hx_n\rightarrow\overline{H}x$ hold in $X$. Then it follows
\[
   \Re \langle x, \overline{H}x \rangle = \lim_{n\to\infty}\Re \langle x_n,Hx_n\rangle \leqslant 0.
\]
We therefore have $H\subseteq\overline{H}$ and if $H$ is maximal dissipative we obtain $H=\overline{H}$ and $H$ is closed.
\end{proof}

\smallskip

\begin{lem}\label{lem:mdr2} Let $X$ be a Hilbert space and $H\colon D(H)\subseteq X \to X$  be a linear, densely defined and maximal dissipative operator. Then $0 \in \rho(1-H)$.
\end{lem}\red{
\begin{proof} Let $y\in X$ satisfy
\[
  \langle y, (1-H)u\rangle =0\quad  \mbox{for all } u \in D(H).
\]
If $y\in D(H)$, then  $y=0$ by choosing $u=y$ and using the dissipativity of $H$. So assume that $y \not\in D(H)$. Then  define 
\[
  H_e(\alpha y + u) = H(u)-\alpha y, \quad u \in D(H), \alpha \in {\mathbb C}.
\]
Now
\begin{align*}
  \Re\langle \alpha y + u,  H_e(\alpha y + u)\rangle =&\Re \langle \alpha y + u, H(u)-\alpha y \rangle\\
  =& \Re(   \langle \alpha y, H(u)- u + u -\alpha y \rangle + \langle u, H(u) -\alpha y \rangle )\\
  =&\ 0 + \Re( \langle \alpha y, u -\alpha y \rangle + \langle u, H(u) \rangle +  \langle u, -\alpha y \rangle )\\
  \leqslant &\ \Re(  \langle \alpha  y, u -\alpha y \rangle - \langle u, \alpha y \rangle )\leqslant 0,
\end{align*}
So $H_e$ is a proper dissipative extension of $H$ contradicting the maximality of $H$.
Hence, $1-H$ has dense range. Moreover, since $H$ is closed by Lemma \ref{lem:mdr1}, it follows from the dissipativity of $H$, that the range of $1-H$ is also closed. Thus, $1-H$ is onto; hence $0\in \rho(1-H)$.
\end{proof}}

\smallskip

We can now conclude the proof of Proposition~\ref{prop:mdr}.

\begin{proof}[Proof of Proposition~\ref{prop:mdr}] By Lemma~\ref{lem:mdr2}, we may define $w\coloneqq (1-H)^{-1}(x-y)\in D(H)$. Then 
\[
   w-Hw = x-y
\]
and by our assumption
\[
    \|y-Hw\|^2 = \Re \langle y-Hw,y-Hw\rangle = \Re \langle x-w,y-Hw \rangle \leqslant 0
\]
follws. We conclude $y=Hw$ and therefore 
\[
w=w-Hw+Hw=x-y+Hw = x
\]
which completes the proof.
\end{proof}

\smallskip

For the sake of completeness let us give the following version of Minty's theorem and a short proof based on our work above.

\smallskip

\begin{thm}\label{thm:Minty} \textbf{(Minty)} Let $X$ be a Hilbert space, $H\colon D(H)\subseteq X\to X$ linear, densely defined and dissipative. Then the following conditions are equivalent.
\begin{enumerate}
\item[(i)] $H$ is maximal dissipative.\vspace{2pt}
\item[(ii)] $1-H$ is onto.
\end{enumerate}
\end{thm}
\begin{proof}(i)$\Rightarrow$(ii) This is Lemma \ref{lem:mdr2}.

\smallskip

(ii)$\Rightarrow$(i) Let $K\supseteq H$ be a dissipative operator. By Lemma \ref{lem:mdr1} we obtain that $K$ is closable. Moreover, $1-H\subseteq 1-\overline{K}$. Since the left-hand side operator is onto and due to dissipativity of $\overline{K}$, the right-hand side operator is one-to-one, we deduce $1-H = 1-\overline{K}$. Hence, $\overline{K}=H \subseteq K$ implying $K=H$. This shows (i).
\end{proof}

%%%%%%%%%%%%%%%%%%%%%%%%%%%%%%%%%%%%%%%%%%%%%%%%%%%%%%%%%%%%%%%%%%%%%%%%%%%%%%%%%%%%%%%%%%%%%
%%%%%%%%%%%%%%%%%%%%%%%%%%%%%%%%%%%%%%%%%%%%%%%%%%%%%%%%%%%%%%%%%%%%%%%%%%%%%%%%%%%%%%%%%%%%%
%%                                                                                         %%
%% 3 Maximal dissipative operators on weighted spaces - bounded, strictly positive weights %%                                                                                          %%                                                                                         %%
%%%%%%%%%%%%%%%%%%%%%%%%%%%%%%%%%%%%%%%%%%%%%%%%%%%%%%%%%%%%%%%%%%%%%%%%%%%%%%%%%%%%%%%%%%%%%
%%%%%%%%%%%%%%%%%%%%%%%%%%%%%%%%%%%%%%%%%%%%%%%%%%%%%%%%%%%%%%%%%%%%%%%%%%%%%%%%%%%%%%%%%%%%%

\section{Maximal dissipative operators on weighted spaces}\label{sec:mdo}

Having discussed possible maximal dissipative extensions of skew-symmetric operators in Section \ref{SEC:2}, we will now discuss (essentially) maximal dissipative operators on weighted Hilbert spaces. We start with a \emph{uniformly finite weight} $\mathcal{H}\colon X\rightarrow X$, i.e., $\mathcal{H}$ is bounded and bounded away from zero. Later we will relax these assumptions. \red{In the sequel, $L(X)$ denotes the space of bounded linear operators from $X$ into itself.}

\smallskip

\subsection{Uniformly bounded weights}\label{SEC:4:1}

\smallskip

\red{To start with, we rephrase \cite[7.2.3]{JZ} in our notation. We mention that this result can also be found in \cite[Lemma 5.1]{Trostorff2014}.}

\begin{thm}\label{thm:mdw} Let $X$ be a Hilbert space and $\mathcal{H}\in L(X)$ be self-adjoint with $\mathcal{H}\geqslant c>0$ for some $c>0$ in the sense of positive definiteness. Let $H\colon D(H)\subseteq X\to X$ be linear. Then the following conditions are equivalent.\vspace{3pt}
\begin{compactitem}
  \item[(i)] $H$ is densely defined and maximal dissipative in $X$.\vspace{3pt}
  \item[(ii)] $H\mathcal{H}$ is densely defined and maximal dissipative in $X_\mathcal{H} \coloneqq (X,\langle \cdot, \mathcal{H}\cdot\rangle_X)$.
\end{compactitem}
\end{thm}

\smallskip

\begin{proof} The operator $\mathcal{H}^{-1}$ is bounded, self-adjoint and $\mathcal{H}^{-1}\geqslant d>0$ holds for some constant $d$. The latter follows from elementary computations. Therefore, it suffices to prove (i)\,$\Rightarrow$\,(ii). For this let us assume that $H$ defines a densely defined, maximal dissipative operator in $X$. Let $x\in D(H\mathcal{H})^\bot$ where the scalar product is computed in $X_\mathcal{H}$. Then for all $y\in D(H\mathcal{H})$ we have
\[
   0 = \langle y,x\rangle_\mathcal{H}= \langle y,\mathcal{H}x\rangle=\langle \mathcal{H} y,x\rangle.
\]
For $z\in D(H)$ we put $y\coloneqq \mathcal{H}^{-1} z$. Then $y\in D(H\mathcal{H})$ holds and we get $\langle z,x\rangle=0$ from the above. Since $D(H)$ is dense in $X$, we deduce $x=0$ and consequently $D(H\mathcal{H})$ is dense in $X_{\mathcal{H}}$. Next, let $x\in D(H\mathcal{H})$ be given. Then
\[
   \Re \langle x,H\mathcal{H} x\rangle_{\mathcal{H}} =    \Re \langle x,\mathcal{H}H\mathcal{H} x\rangle= \Re \langle \mathcal{H} x,H\mathcal{H} x\rangle\leqslant 0
\]
as $H$ is dissipative. It remains to prove the maximality of $H\mathcal{H}$. For this let $H_0\supseteq H\mathcal{H}$ be a dissipative extension. Let $x\in D(H_0)$. Then for all $y\in D(H\mathcal{H})$ we deduce
\[
    0\geqslant \Re \langle x-y,H_0x - H\mathcal{H}y\rangle_\mathcal{H}= \Re \langle \mathcal{H}x-\mathcal{H}y,H_0x - H\mathcal{H}y\rangle.
\]
Note that $\{\mathcal{H}y\;;\;y\in D(H\mathcal{H})\}=D(H)$ since $\mathcal{H}$ is continuously invertible and that $H$ is maximal dissipative by assumption. Therefore we can apply Proposition~\ref{prop:mdr} with $(\mathcal{H}x,H_0x)\in X\times X$ to obtain that $\mathcal{H}x\in D(H)$ holds. It follows $x\in D(H\mathcal{H})$ and we get $H\mathcal{H}x=H_0x$. This shows that $H_0=H\mathcal{H}$ holds and the statement is proved.
\end{proof}

\medskip

%%%%%%%%%%%%%%%%%%%%%%%%%%%%%%%%%%%%%%%%%%%%%%%%%%%%%%%%%%%%%%%%%%%%%%%%%%%%%%%%%%
%%%%%%%%%%%%%%%%%%%%%%%%%%%%%%%%%%%%%%%%%%%%%%%%%%%%%%%%%%%%%%%%%%%%%%%%%%%%%%%%%%
%%                                                                              %%
%% 4 Maximal dissipative operators on weighted spaces -- locally finite weights %%
%%                                                                              %%
%%%%%%%%%%%%%%%%%%%%%%%%%%%%%%%%%%%%%%%%%%%%%%%%%%%%%%%%%%%%%%%%%%%%%%%%%%%%%%%%%%
%%%%%%%%%%%%%%%%%%%%%%%%%%%%%%%%%%%%%%%%%%%%%%%%%%%%%%%%%%%%%%%%%%%%%%%%%%%%%%%%%%

\subsection{Locally finite weights}\label{sec:mdo2}

\smallskip

In Theorem~\ref{thm:mdw} a crucial ingredient is that $\mathcal{H}$ is an isomorphism. In this subsection we will relax this. We use the following terminology.

\smallskip

\begin{dfn}\label{DEF} Let $X$ be a Hilbert space. A self-adjoint operator $\mathcal{H}$ in $X$ is called a \emph{locally finite weight}, if there exists an increasing sequence $(X_n)_{n\in\mathbb{N}}$ of Hilbert subspaces $X_n\subseteq X$ such that\vspace{3pt}

\begin{compactitem}

\item[(i)] $\mathcal{H}_n\colon X_n\rightarrow X_n$, $x\mapsto \mathcal{H}x$ is everywhere defined, bounded and $\mathcal{H}_n\geqslant c_n>0$ holds with suitable constants $c_n$ for every $n\in \mathbb{N}$,\vspace{3pt}

\item[(ii)] $\mathcal{H}|_{\cup_{n\in \mathbb{N}}X_n}$ is essentially selfadjoint in $X$.\vspace{3pt}

\end{compactitem}
\end{dfn}

\smallskip

\begin{rem}\label{DEF-REM} From the assumptions in Definition \ref{DEF} the following facts follow almost immediately. \vspace{3pt}
\begin{compactitem}
\item[(i)] $X_n\subseteq D(\mathcal{H})$ holds for every $n\in\mathbb{N}$ and $X_n$ is invariant under $\mathcal{H}$.\vspace{3pt}
\item[(ii)] $\mathcal{H}_n\colon X_n\rightarrow X_n$ is an isomorphism for every $n\in\mathbb{N}$.\vspace{3pt}
\item[(iii)] $P_n\mathcal{H}\subseteq\mathcal{H}P_n$ holds for every $n\in\mathbb{N}$ where $P_n\in L(X)$ denotes the projection on $X_n$.\vspace{3pt}
\item[(iv)] $\cup_{n\in\mathbb{N}}X_n\subseteq D(\mathcal{H})$ is dense with respect to the graph norm $\|\cdot\|_{D(\mathcal{H})}=(\|\cdot\|_X^2+\|\mathcal{H}\cdot\|_X^2)^{1/2}$.\vspace{3pt}
\item[(v)] $\cup_{n\in\mathbb{N}}X_n\subseteq X$ is dense with respect to $\|\cdot\|_X$.
\end{compactitem}
\end{rem}

\smallskip

\begin{lem}\label{LEM-INJ} Let $X$ be a Hilbert space and $\mathcal{H}$ be a locally finite weight. Then $\mathcal{H}$ and $\mathcal{H}^{1/2}$ are injective.
\end{lem}
\begin{proof}
 Let $x\in \ker(\mathcal{H})$. We need to show that $x=0$. For $n\in\mathbb{N}$ let $P_n \in L(X)$ be the orthogonal projection onto $X_n$. We find for each $n\in\mathbb{N}$ a constant $c_n>0$ such that 
 \[
      c_n \|y\|^2 \leqslant \langle y, \mathcal{H}y\rangle
 \]
holds for all $y\in X_n$. Thus we deduce
 \[
    c_n\| P_n x\|^2 \leqslant \langle P_nx,\mathcal{H} P_n x\rangle = \langle P_nx, P_n\mathcal{H} x\rangle =0
 \]
for every $n\in\mathbb{N}$ by using the invariance of $X_n$ under $\mathcal{H}$. From this we infer $0=P_n x\to x$ as $n\to\infty$, which shows that $\mathcal{H}$ is injective. The injectivity of $\mathcal{H}^{1/2}$ follows.
\end{proof}

\smallskip

Our goal is now to associate with a locally finite weight $\mathcal{H}$ on $X$ again a weighted space $X_{\mathcal{H}}$. Since $\mathcal{H}$ is not defined on the whole space we can do this as in Section \ref{sec:mdo2} only on each of the corresponding subspaces. Indeed, for every $n\in\mathbb{N}$ we get a Hilbert space $(X_n,\langle{}\cdot,\mathcal{H}\cdot\rangle{})$ to which the results of Section \ref{sec:mdo2} apply. We thus define
\begin{equation}\label{DFN-XH}
     X_{\mathcal{H}} \coloneqq \big(\mathop{{\textstyle\bigcup}}_{\scriptstyle n\in \mathbb{N}} X_n, \langle \cdot, \mathcal{H}\cdot\rangle\big)^{\sim}
\end{equation}
\red{where $(\cdot)^{\sim}$ denotes the completion,} and show in Proposition \ref{upshot-lemma} that we can embed $D(\mathcal{H})$ into $X_{\mathcal{H}}$ and that the space $X_{\mathcal{H}}$ is in fact independent of the choice of the sequence $(X_n)_{n\in\mathbb{N}}$ as long as the latter fulfills the conditions in Definition \ref{DEF}.

\red{\begin{rem}Note that one cannot drop the completion in the definition of $X_{\mathcal{H}}$. Indeed, it is easy to see that $X=L^2(0,1)$,  $\mathcal{H}\colon L^2(0,1)\to L^2(0,1),f\mapsto (x\mapsto xf(x))$, with $P_n$ being the operator of multiplication by the characteristic function of $[1/n,1]$, satisfies the above assumptions. The space $\big(\mathop{{\textstyle\bigcup}}_{\scriptstyle n\in \mathbb{N}} X_n, \langle \cdot, \mathcal{H}\cdot\rangle\big)$ is however not complete.\end{rem}}

\smallskip

\begin{prop}\label{upshot-lemma} Let $X$ be a Hilbert space, let $\mathcal{H}$ be a locally finite weight with $(X_n)_{n\in\mathbb{N}}$ being a corresponding sequence of Hilbert subspaces. Let $X_{\mathcal{H}}$ be defined as in \eqref{DFN-XH}.\vspace{3pt}
\begin{compactitem}

\item[(i)] The map $P\colon D(\mathcal{H})\rightarrow X_{\mathcal{H}}$, $x\mapsto (P_nx)_{n\in\mathbb{N}}$ is well-defined and injective.

\vspace{3pt}

\item[(ii)] If $(K_n)_{n\in\mathbb{N}}$ is another sequence of subspaces corresponding to $\mathcal{H}$, $Q_n\in L(X)$ is the orthogonal projection on $K_n$, $K_{\mathcal{H}}$ is defined analogously to \eqref{DFN-XH}, then there is an isometric isomorphism $J\colon K_{\mathcal{H}}\rightarrow X_{\mathcal{H}}$ that makes the diagram
\begin{equation*}
\begin{tikzcd}
K_{\mathcal{H}}\arrow[rightarrow, "J"]{r}&X_{\mathcal{H}}\\
D(\mathcal{H})\arrow[hookrightarrow, "Q"]{u}\arrow[hookrightarrow, "P", swap]{ur} & 
\end{tikzcd}
\end{equation*}
commutative.
\end{compactitem}
\end{prop}

\begin{proof} (i) Let $x\in D(\mathcal{H})$. Then $P_nx\in X_n\subseteq D(\mathcal{H})$. By Remark \ref{DEF-REM}(i) and since $P_n\rightarrow \id_{X}$ holds \red{in the strong operator topology}, we have $\mathcal{H}P_n x=P_n \mathcal{H} x \to \mathcal{H}x$ and $P_nx\rightarrow x$ for $n\rightarrow\infty$. This means that 
$$
\|P_n x - x\|_{\mathcal{H}}^2=\langle P_n x - x, P_nx-x\rangle_{\mathcal{H}}=\langle P_n x - x, \mathcal{H}P_nx-\mathcal{H}x\rangle \to 0
$$
for $n\rightarrow\infty$ as both arguments of the last scalar product tend to zero in $X$. Consequently, $(P_nx)_{n\in\mathbb{N}}$ is a Cauchy sequence in $(\cup_{n\in \mathbb{N}} X_n, \langle \cdot, \mathcal{H}\cdot\rangle)$. This shows that $P$ is well-defined. For the injectivity, let $Px=0$. Then
$$
0=\|Px\|_{X_{\mathcal{H}}}=\|(P_nx)_{n\in\mathbb{N}}\|_{X_{\mathcal{H}}} = \lim_{n\rightarrow\infty}\|P_nx\|_{\mathcal{H}}=\lim_{n\rightarrow\infty}\|\mathcal{H}^{1/2}P_nx\|_{X}
$$
follows. Therefore $\mathcal{H}^{1/2}P_nx\rightarrow0$ in $X$. Since $P_nx\rightarrow x$ for every $x\in X$ and $\mathcal{H}^{1/2}$ is closed we conclude $x=0$, by Lemma \ref{LEM-INJ}.

\medskip

(ii) We put $K\coloneqq\bigcup_{n\in\mathbb{N}}K_n\subseteq D(\mathcal{H})$. Using Lemma \ref{LEM-INJ} we get immediately, that $(K,\langle\cdot,\mathcal{H}\cdot\rangle)$ is a pre-Hilbert space. Indeed, let $x\in K$ be given with $\langle x,\mathcal{H}x\rangle=0$. Then we have $\langle \mathcal{H}^{1/2}x,\mathcal{H}^{1/2}x\rangle=\langle x,\mathcal{H}x\rangle=0$. Since $\mathcal{H}^{1/2}$ is injective we can conclude $x=0$. The other properties are clear.

\smallskip

We see that $(K,\langle\cdot,\mathcal{H}\cdot\rangle)$ embeds isometrically into $X_{\mathcal{H}}$. More precisely, the map
$$
K\rightarrow X_{\mathcal{H}},\;\;x\mapsto (P_n x)_{n\in\mathbb{N}}
$$
is well-defined by (i) and isometric on a dense subset and hence extends isometrically to $J\colon K_{\mathcal{H}}\rightarrow X_{\mathcal{H}}$. In order to prove that $J$ is an isomorphism, it suffices to show that the image of $J$ is dense in $X_\mathcal{H}$. For this let $x\in \bigcup_{n\in\mathbb{N}} X_n$ be given. That is, $x\in X_n$ holds for some $n\in\mathbb{N}$. Since $\bigcup_{\ell\in\mathbb{N}} K_\ell$ is a core for $\mathcal{H}$, we find a sequence $(y_\ell)_{\ell\in\mathbb{N}}$ with $y_\ell\in K_\ell$ and $y_\ell \to x$ in $D(\mathcal{H})$. This means $y_{\ell}\rightarrow x$ and $\mathcal{H}y_{\ell}\rightarrow\mathcal{H}x$ in $X$ for $\ell\rightarrow\infty$. Since $P_n\in L(X)$ we obtain $P_ny_{\ell}\rightarrow P_nx$ and $P_n\mathcal{H}y_{\ell}\rightarrow P_n\mathcal{H}x$ in $X$ for $\ell\rightarrow\infty$. By Remark \ref{DEF-REM}(iii) we can interchange $P_n$ and $\mathcal{H}$ in the last statement and get $\mathcal{H}P_ny_{\ell}\rightarrow \mathcal{H}P_nx$ in $X$ for $\ell\rightarrow\infty$. Thus, we deduce
\[
     \|P_ny_\ell - x\|_{X_{\mathcal{H}}}^2=\langle P_ny_\ell - x, P_ny_\ell - x\rangle _{X_{\mathcal{H}}}=\langle P_ny_\ell - x, \mathcal{H}P_ny_\ell - \mathcal{H}x\rangle \to 0
\]
for $\ell\to\infty$. Thus, $J$ is an isometric isomorphism.

\smallskip

For the commutativity consider $x\in K=\bigcup_{n\in\mathbb{N}}K_n\subseteq D(\mathcal{H})$. The map $Q$ then sends $x$ to a sequence that is eventually constant $x$. In $K_{\mathcal{H}}$ this sequence is equivalent to the sequence with every entry being $x$. But this shows that $JQx$ and $Px$ coincide. As $K\subseteq D(\mathcal{H})$ is dense, this finishes the proof.
\end{proof}

\smallskip

\begin{rem} In the situation of Proposition \ref{upshot-lemma} we have the following. If $\mathcal{H}|_{K}^K$ is bounded and $\mathcal{H}|_{K}^K\geqslant c$ holds for some $c>0$, then $K\subseteq X_{\mathcal{H}}$ is closed if and only if $K\subseteq X$ is closed. Indeed, $\langle\cdot,\mathcal{H}\cdot\rangle$ is equivalent to $\langle\cdot,\cdot\rangle$ on $K$ due to the additional assumption.
\end{rem}

\smallskip

In the remainder, we consider in addition to $\mathcal{H}$ another operator $H\colon D(H)\subseteq X\rightarrow X$. For the moment we assume that we are given a sequence $(X_n)_{n\in\mathbb{N}}$ of subspaces that are invariant under $H$, i.e., $H(D(H))\cap X_n)\subseteq X_n$ holds for every $n\in\mathbb{N}$. Then we consider $H_n\colon D(H_n)\subseteq X_n\rightarrow X_n$ with $D(H_n)=D(H)\cap X_n$ and $H_nx=Hx$ for every $n\in\mathbb{N}$.

\smallskip

\begin{thm}\label{thm:mada} Let $X$ be a Hilbert space and let $(X_n)_{n\in\mathbb{N}}$ be an increasing family of closed subspaces such that $\bigcup_{n\in\mathbb{N}} X_n \subseteq X$ is dense. Denote by $P_n \in L(X)$ the orthogonal projection onto $X_n$ and assume that $H\colon D(H)\subseteq X\to X$ is densely defined such that $H$ leaves $X_n$ invariant and $P_nH \subseteq HP_n$ holds for every $n\in\mathbb{N}$. Then the following are equivalent.\vspace{3pt}
\begin{compactitem}
 \item[(i)] $\overline{H}$ is maximal dissipative.\vspace{3pt}
 \item[(ii)] $H_n$ is dissipative and $\ran(1-H_n)\subseteq X_n$ is dense for every $n\in\mathbb{N}$.
\end{compactitem}
\end{thm}

\begin{proof} (i)$\,\Rightarrow\,$(ii) Since $\overline{H}$ is an extension of $H_n$, it follows that $H_n$ is dissipative. As $\overline{H}$ is maximal dissipative, Lemma \ref{lem:mdr2} implies that $\ran(1-\overline{H})=X$ holds and it follows that $\ran(1-H)\subseteq X$ is dense. Let $n\in\mathbb{N}$ and $y_n \in X_n$. We find $(x_k)_{k\in\mathbb{N}}$ in $D(H)$ such that 
 \[
    x_k - Hx_k \to y_n
 \]
holds for $k\to\infty$ in $X$. Since $P_n x_k \in D(H_n)$ and $HP_n x_k = P_n H x_k$ for all $k\in \mathbb{N}$, we infer
 \[
   P_n x_k - H_nP_n x_k = P_n x_k - H P_n x_k = P_n x_k - P_n H x_k = P_n \left(x_k - H x_k\right) \to P_n y_n= y_n
 \]
for $k\to\infty$ in $X_n$. Hence, $\ran(1-H_n)\subseteq X_n$ is dense.

\smallskip

(ii)$\,\Rightarrow\,$(i) We first show that $H$ is dissipative. Let $x\in D(H)$ be given. For $n\in \mathbb{N}$ we have $P_nx\in D(H)\cap D(H_n)$ and we compute
 \[
   0\geqslant \Re \langle P_n x,H_nP_n x\rangle = \Re \langle P_n x, P_n Hx\rangle = \Re \langle P_n x, Hx\rangle \to \Re \langle x,Hx\rangle
 \]
for $n\to\infty$, where we used that $P_n \to \id_X$ holds in the strong operator topology as $\bigcup_{n\in\mathbb{N}} X_n$ is dense in $X$. Since $H$ is densely defined, we obtain by Lemma \ref{lem:mdr1} that $H$ is closable. Moreover, we have $\bigcup_{n\in\mathbb{N}}\ran(1-H_n)\subseteq\ran(1-H)$. Here, the left-hand side is dense in $\bigcup_{n\in\mathbb{N}} X_n$ and this set is in turn dense in $X$. Therefore, we deduce that $1-H$ has dense range. As $H$ is closable, so is $1-H$. Thus, we  conclude that $\overline{1-H}=1-\overline{H}$ has closed range and is onto. Hence, for $y\in X$ and $n\in\mathbb{N}$, we find $y_n \in \ran(1-H_n)$ such that $y_n\to y$. Thus, there exists $x_n \in D(H_n)\subseteq D(\overline{H})$ such that $y_n=(1-\overline{H})x_n$ holds for every $n\in\mathbb{N}$. For $n$, $m\in\mathbb{N}$ we compute
$$
\|x_n-x_m\|\leqslant\|(1-\overline{H})(x_n-x_m)\|=\|y_n-y_m\|
$$
since $\overline{H}$ is dissipative. The latter shows that $(x_n)_{n\in\mathbb{N}}$ is a Cauchy sequence and hence convergent to some $x\in X$. As $1-\overline{H}$ is closed it follows $(1-\overline{H})x=y$ which establishes that $1-\overline{H}$ is a bijection. This implies that $\overline{H}$ is maximal dissipative.
\end{proof}

\medskip

Next, we want to study the operator $H\mathcal{H}$ on the space $X_{\mathcal{H}}$. We keep in mind that $\mathcal{H}$ and $H$ live on the space $X$, that $X$ and $X_{\mathcal{H}}$ are a priori not comparable but that the map
$$
P\colon D(\mathcal{H})\rightarrow X_{\mathcal{H}}, \;x\mapsto(P_nx)_{n\in\mathbb{N}}
$$
is well-defined and injective. We define $H\mathcal{H}\colon X_{\mathcal{H}}\supseteq D(H\mathcal{H})\rightarrow X_{\mathcal{H}}$ as follows.
\begin{equation}\label{HH}
\begin{array}{rcl}
D(H\mathcal{H}) &:= &\bigl\{\tilde{x}\in X_{\mathcal{H}} \:;\:\exists\:x\in D(\mathcal{H})\colon Px=\tilde{x},\, \mathcal{H}x \in D(H) \text{ and } (P_nH\mathcal{H}x)_{n\in\mathbb{N}}\in X_{\mathcal{H}}\bigl\}\vspace{5pt}\\
H\mathcal{H}\tilde{x} &:=& (P_nH\mathcal{H}x)_{n\in\mathbb{N}}
\end{array}
\end{equation}
We denote by $X_{n,\mathcal{H}}$ the space $X_n$ considered as a subspace of $X_{\mathcal{H}}$. $\tilde{P}_n\in L(X_{\mathcal{H}})$ denotes the corresponding orthogonal projection. Observe that $P$ induces a bijection $X_n\rightarrow X_{n,\mathcal{H}}$. Indeed, $(P_kx)_{k\in\mathbb{N}}$ is eventually constant and thus Cauchy in $(\bigcup_{n\in\mathbb{N}}X_n,\langle\cdot,\mathcal{H}\cdot\rangle{})$. Observe further that we have
$$
\langle Px,Py\rangle_{X_{n,\mathcal{H}}}=\langle x,\mathcal{H}y\rangle_{X_n}
$$
for all $x$, $y\in X_n$.

\medskip

\begin{prop}\label{lem:mada} Let $X$ be a Hilbert space, let $\mathcal{H}$ be a locally finite weight and let $(X_n)_{n\in\mathbb{N}}$ be a corresponding family of subspaces. Let $H\colon D(H)\subseteq X\to X$ be densely defined and assume that $P_n H \subseteq HP_n$ holds for every $n\in\mathbb{N}$. Assume in addition that every $X_n$ is invariant under $H$.  We consider the operator $H\mathcal{H}\colon X_{\mathcal{H}}\supseteq D(H\mathcal{H})\rightarrow X_{\mathcal{H}}$ as defined in \eqref{HH}. Then the following holds. \vspace{5pt}

\begin{compactitem}

\item[(i)] The operator $H\mathcal{H}$ is densely defined.\vspace{3pt}

\item[(ii)] We have $\tilde{P}_n H\mathcal{H}\subseteq H\mathcal{H}\tilde{P}_n$ for every $n\in\mathbb{N}$.
\end{compactitem}
\end{prop}

\begin{proof} (i) Since $\bigcup_{n\in\mathbb{N}}X_{n,\mathcal{H}}\subseteq X_{\mathcal{H}}$ is dense and $X_{n,\mathcal{H}}$ carries the topology induced by $X_{\mathcal{H}}$ it is enough to show that $D(H\mathcal{H})\cap X_{n,\mathcal{H}}$ is dense in $X_{n,\mathcal{H}}$ for every $n\in\mathbb{N}$. Let $\tilde{x}\in X_{n,\mathcal{H}}$ be such that
\begin{equation}\label{EQ-20}
\forall\:\tilde{y}\in D(H\mathcal{H})\cap X_{n,\mathcal{H}} \colon \langle \tilde{x},\tilde{y}\rangle_{X_{n,\mathcal{H}}} =0
\end{equation}
holds. We need to show $\tilde{x}=0$. We find $x\in D(\mathcal{H})$ with $Px=\tilde{x}$. Since $P$ is linear, it is enough to show $x=0$. Firstly, we claim that
\begin{equation}\label{EQ-22}
\forall\:z\in D(H)\cap X_{n} \colon \langle x,z\rangle=0
\end{equation}
holds. For this let $z\in D(H)\cap X_{n}$ be given. By Remark \ref{DEF-REM}(ii) we can select $y\in X_n\subseteq D(\mathcal{H})$ with $\mathcal{H}y=z\in D(H)$. Since $X_n$ is invariant under $H$ by assumption we get $H\mathcal{H}y=Hz\in X_n$ and in particular, $(P_nH\mathcal{H}y)_{n\in\mathbb{N}}\in X_{\mathcal{H}}$. This means $\tilde{y}:=Py\in D(H\mathcal{H})$ and since $P(X_n)\subseteq X_{n,\mathcal{H}}$, we obtain $\tilde{y}\in D(H\mathcal{H})\cap X_{n,\mathcal{H}}$. Employing \eqref{EQ-20} we get
$$
0=  \langle \tilde{x},\tilde{y}\rangle_{X_{n,\mathcal{H}}} = \langle Px,Py\rangle_{X_{n,\mathcal{H}}}=\langle x,\mathcal{H}y\rangle_{X_n} = \langle x,z\rangle =0
$$
which establishes \eqref{EQ-22}. Thus, we obtain $x=0$ if we show that $D(H)\cap X_n$ is dense in $X_n$. For this let $x_n\in X_n$ be given. We find $(y_k)_{k\in\mathbb{N}}$ in $D(H)$ such that $y_k \to x_n$ for $k\to\infty$. Since $P_n y_k\in D(H)\cap X_n$ and $P_ny_k\to P_n x_n=x_n$ for $k\to\infty$, we obtain that $D(H)\cap X_n\subseteq X_n$ is dense.

\medskip

(ii) In order to show $\tilde{P}_n H\mathcal{H}x=H\mathcal{H}\tilde{P}_nx$ for every $x\in D(H\mathcal{H})$ we need first to understand how $\tilde{P}_n$ acts on $D(\mathcal{H}H)$ and on $\ran(\mathcal{H}H)$. In order to do this, we define the auxiliary space
$$
Y:=\bigl\{x\in X\:;\:(P_nx)_{n\in\mathbb{N}} \text{ Cauchy sequence in }({\textstyle\bigcup_{n\in\mathbb{N}}}X_n,\langle\cdot,\mathcal{H}\cdot\rangle{})\bigr\}
$$
which is by construction a subspace of $X$. We extend the map $P\colon Y\rightarrow X_{\mathcal{H}}$, $Px=(P_nx)_{n\in\mathbb{N}}$ to $Y$ and show that it is also injective with this larger domain. Indeed, if $Px=0$ then
$$
0=\|Px\|_{X_{\mathcal{H}}}=\|(P_nx)_{n\in\mathbb{N}}\|_{X_{\mathcal{H}}} = \lim_{n\rightarrow\infty}\|P_nx\|_{\mathcal{H}}=\lim_{n\rightarrow\infty}\|\mathcal{H}^{1/2}P_nx\|_{X}
$$
follows. Therefore $\mathcal{H}^{1/2}P_nx\rightarrow0$ in $X$. In addition, we know $P_nx\rightarrow x$ for every $x\in X$. Since $\mathcal{H}^{1/2}$ is closed and injective by Lemma \ref{LEM-INJ}, we conclude $x=0$. We can now think of $Y$ to be a replacement for the \textquotedblleft{}intersection $X\cap X_{\mathcal{H}}$\textquotedblright{}. We claim that $PP_n=\tilde{P}_nP$ holds on $Y$ or, in other words, that the upper part of the diagram
\begin{equation}\label{DIAG-1}
\begin{tikzcd}
X_n\arrow[rightarrow, "P"]{rr}[swap]{\stackrel{\phantom{.}}{\sim}}& & X_{n,\mathcal{H}} \\
X\arrow[hookleftarrow, "\id"]{rd}\arrow[rightarrow, "P_n"]{u}& & X_{\mathcal{H}}\arrow[rightarrow, "\tilde{P}_n", swap]{u}\\
& Y \arrow[hookleftarrow, "P^{-1}"]{rd}\arrow[hookrightarrow, "P"]{ru}& \\
D(\mathcal{H})\arrow[hookrightarrow, "\id"]{uu}\arrow[hookrightarrow, "\id"]{ru}\arrow[hookleftarrow, "P^{-1}", swap]{rr} & & D(H\mathcal{H})\arrow[hookrightarrow, "\id", swap]{uu}
\end{tikzcd}
\end{equation}
is commutative. Let $x\in Y$ and $\tilde{x}:=Px \in X_{\mathcal{H}}$. We put $\tilde{y}:=\tilde{P}_n\tilde{x}=\tilde{P}_nPx$. Then $\tilde{y}\in X_{n,\mathcal{H}}$ is characterized by the condition
$$
\forall\;\tilde{z}\in X_{n,\mathcal{H}}\colon \langle \tilde{x}-\tilde{y},\tilde{z}\rangle_{X_{n,\mathcal{H}}} = 0.
$$
Let $\tilde{z}\in X_{n,\mathcal{H}}$. We put $z:=P^{-1}\tilde{z}\in X_n$ and compute
\begin{align*}
\langle \tilde{x}-PP_nx,\tilde{z}\rangle_{X_{n,\mathcal{H}}} &= \langle P(x-P_nx),Pz\rangle_{X_{n,\mathcal{H}}}
     \\  &=\langle x-P_nx,\mathcal{H}z\rangle_{X_n}
     \\ & =       \langle x-P_nx,\mathcal{H}P_nz\rangle
     \\ & =       \langle  x-P_nx,P_n\mathcal{H}z\rangle
     \\ & =       \langle  P_nx-P_n^2x,\mathcal{H}z\rangle
     \\ & =       \langle  P_nx-P_nx,\mathcal{H}z\rangle = 0
  \end{align*}
and conclude $\tilde{P}_nPx=\tilde{P}_n\tilde{x}=\tilde{y}=PP_nx$.

\medskip

Before we can finish the proof, we observe that whenever we have $\tilde{z}\in D(H\mathcal{H})$, i.e., $\tilde{z}=Pz$ with $z\in D(\mathcal{H})$, $\mathcal{H}z\in D(H)$, we can consider $H\mathcal{H}z\in X$ and $H\mathcal{H}\tilde{z}\in X_{\mathcal{H}}$. This implies however that $H\mathcal{H}z$ belongs to $Y$. We thus can consider $PH\mathcal{H}z\in X_{\mathcal{H}}$ and we see immediately that
\begin{equation}\label{EQ-N}
H\mathcal{H}Pz=H\mathcal{H}\tilde{z}=(P_nH\mathcal{H}z)_{n\in\mathbb{N}}=PH\mathcal{H}z
\end{equation}
holds.

\medskip

Now we show $\tilde{P}_n H\mathcal{H}\subseteq H\mathcal{H}\tilde{P}_n$. For this let $\tilde{x}\in D(H\mathcal{H})$. We apply \eqref{EQ-N} to $z=x$ and $z=P_nx$, use the commutativity of \eqref{DIAG-1}, employ the fact that $P_nH\subseteq HP_n$ holds by assumption, and use Remark \ref{DEF-REM}(iv) to obtain
$$
\tilde{P}_nH\mathcal{H}\tilde{x}=\tilde{P}_nH\mathcal{H}Px=\tilde{P}_nPH\mathcal{H}x=PP_nH\mathcal{H}x=PHP_n\mathcal{H}x=PH\mathcal{H}P_nx=H\mathcal{H}PP_nx=H\mathcal{H}\tilde{P}_n\tilde{x}
$$
as desired.
\end{proof}

\medskip

\begin{rem}\label{rmk-49}\begin{compactitem}\item[(i)] Notice that a priori there might exist $x\in D(\mathcal{H})$ with $\mathcal{H}x\in D(H)$ such that $(P_nH\mathcal{H}x)_{n\in\mathbb{N}}$ is---though convergent to $H\mathcal{H}x$ in $X$---not a Cauchy sequence in the space $(\bigcup_{n\in\mathbb{N}}X_n,\langle{}\cdot,\mathcal{H}\cdot\rangle{})$.

\vspace{4pt}

\item[(ii)] The proof of Proposition \ref{lem:mada}  might seem to be a bit tedious since we kept on using the map $P$, relating the elements of $X$ with the elements of $X_{\mathcal{H}}$, until its very end. Indeed, at some point---and definitely now that the result is established---we can identify $x$ and $\tilde{x}=Px$. The diagram \eqref{DIAG-1} then collapses to
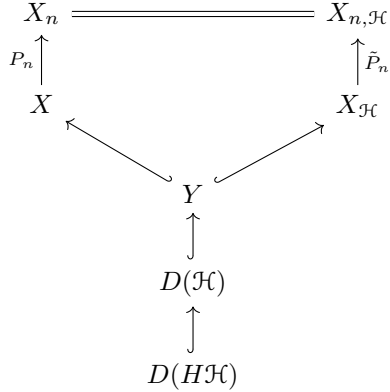
\begin{equation*}
\begin{tikzcd}
X_n\ar[equal]{rr}& & X_{n,\mathcal{H}} \\
X\arrow[hookleftarrow]{rd}\arrow[rightarrow, "P_n"]{u}& & X_{\mathcal{H}}\arrow[rightarrow, "\tilde{P}_n", swap]{u}\\
& Y \arrow[hookrightarrow]{ru}& \\
&D(\mathcal{H})\arrow[hookrightarrow]{u} & \\
&D(H\mathcal{H})\arrow[hookrightarrow]{u}&
\end{tikzcd}
\end{equation*}
and the definition of $H\mathcal{H}\colon X_{\mathcal{H}}\supseteq D(H\mathcal{H})\rightarrow X_{\mathcal{H}}$ simplifies to
\[
D(H\mathcal{H})=\bigl\{ x\in D(\mathcal{H})\:;\:\mathcal{H}x \in D(H) \text{ and } H\mathcal{H}x\in X_{\mathcal{H}}\bigl\},\;\;x\mapsto H\mathcal{H}x.
\]
We emphasize however, that for the proof it was essential that the identification of $Y\subseteq X$ with a subspace of $X_{\mathcal{H}}$ is compatible with the way we identified $D(\mathcal{H})$ with a subspace of $X_{\mathcal{H}}$. If necessary, in order to make it easier to keep track if we work in the space $X_n$ or in $X_{n,\mathcal{H}}$ later, we will reintroduce the map $P$.

\end{compactitem}

\end{rem}

\medskip

\begin{lem}\label{lem:ele} Let $X$ be a Hilbert space and let $H\colon D(H)\subseteq X\to X$ be linear and densely defined. Then the following are equivalent.\vspace{3pt}
\begin{compactitem}
\item[(i)] $H$ is dissipative and $\ran(1-H)\subseteq X$ is dense. \vspace{2pt}
\item[(ii)] $\overline{H}$ is maximal dissipative.
\end{compactitem}
\end{lem}

\begin{proof} (i)\,$\Rightarrow$\,(ii)\; By Lemma \ref{lem:mdr1} we obtain that ${H}$ is closable. It is thus easy to see that $\overline{H}$ is dissipative. Moreover, $\ran(\overline{1-H})=\ran(1-\overline{H})$ is closed and contains by assumption a dense subset of $X$. It follows that $1-\overline{H}$ is onto. Any dissipative extension $K$ of $\overline{H}$ leads to $1-\overline{H}\subseteq 1-K$. Since $1-K$ is injective, and $1-\overline{H}$ is onto, we obtain $1-\overline{H}=1-K$, which implies (ii).

\medskip
 
(ii)\,$\Rightarrow$\,(i)\; Since $H$ is a restriction of $\overline{H}$ it is clearly dissipative. Furthermore, $\ran(1-H)$ is dense in $\ran(1-\overline{H})$. By Lemma \ref{lem:mdr2} we get that $\ran(1-\overline{H})=X$. Thus, we infer that (i) holds.
\end{proof}

\smallskip

The desired theorem now reads as follows.

\smallskip

\begin{thm}\label{thm:lfw} Let $X$ be a Hilbert space and let $\mathcal{H}$ be a locally finite weight with corresponding subspaces $(X_n)_{n\in\mathbb{N}}$. Let $H\colon D(H)\subseteq X\to X$ be linear, densely defined and such that $H$ leaves each $X_n$ invariant. Assume moreover that $P_nH \subseteq HP_n$ holds for every $n\in\mathbb{N}$, where $P_n\in L(X)$ is the orthogonal projection onto $X_n$. Then the following are conditions equivalent.

\vspace{3pt}

\begin{compactitem}
\item[(i)] $\overline{H}$ is maximal dissipative in $X$.\vspace{3pt}

\item[(ii)] $H_n$ is dissipative in $X_{n}$ and $\ran(1-H_n)\subseteq X_n$ is dense for every $n\in \mathbb{N}$.\vspace{3pt}

\item[(iii)] $\overline{H}_n$ is maximal dissipative in $X_n$ for every $n\in \mathbb{N}$.\vspace{3pt}
 
\item[(iv)] $H_n\mathcal{H}$ is dissipative in $X_{n,\mathcal{H}}$ and $\ran(1-H_n\mathcal{H})\subseteq X_{n,\mathcal{H}}$ is dense  for every $n\in \mathbb{N}$.\vspace{2.5pt}

\item[(v)]  $\overline{H_n\mathcal{H}}$ is maximal dissipative in $X_{n,\mathcal{H}}$ for every $n\in\mathbb{N}$.\vspace{3pt}

\item[(vi)] $\overline{H\mathcal{H}}$ is maximal dissipative in $X_\mathcal{H}$.
\end{compactitem}
\end{thm}

\begin{proof} (i)\,$\Leftrightarrow$\,(ii) Theorem \ref{thm:mada}.
\smallskip
\\(iv)\,$\Leftrightarrow$\,(vi) By Proposition \ref{lem:mada} the assumptions of Theorem \ref{thm:mada}, but for $H\mathcal{H}$ instead of $H$ and $\tilde{P}_n\in L(X_{\mathcal{H}})$ projecting on $X_{n,\mathcal{H}}$, are satisfied. Therefore the equivalence follows again from Theorem \ref{thm:mada}.
\smallskip
\\(ii)\,$\Leftrightarrow$\,(iii) Lemma \ref{lem:ele}.
\smallskip
\\(iv)\,$\Leftrightarrow$\,(v) By Proposition \ref{lem:mada}, the operators $H_{n}$ are densely defined in $X_{n,\mathcal{H}}$. Thus, this equivalence also follows from Lemma \ref{lem:ele}.
\smallskip
\\(iii)\,$\Leftrightarrow$\,(v) Theorem \ref{thm:mdw}. 
\end{proof}

\begin{rem} The difference between Theorem \ref{thm:mdw} and Theorem \ref{thm:lfw} is the assumption on $\mathcal{H}$. In Theorem \ref{thm:lfw} we relaxed the condition that $\mathcal{H}\geqslant c>0$ needs to hold. The trade-off is that we have to confine ourselves to operators $H$ that interact in a certain sense well with $\mathcal{H}$. Observe that the moral of both theorems is however the same: When it comes to maximal dissipativity of $\overline{H\mathcal{H}}$, then we can \textquotedblleft{}assume $\mathcal{H}={1}$ without loss of generality\textquotedblright{}.
\end{rem}

\smallskip

As a corollary we get that also for $\overline{H\mathcal{H}}$ being skew-self-adjoint we can assume w.l.o.g.~that $\mathcal{H}=1$ holds.

\smallskip

\begin{cor}\label{cor:essX} Let $X$ be a Hilbert space and $\mathcal{H}$ be a locally finite weight corresponding to $(X_n)_{n\in\mathbb{N}}$. Let $H\colon D(H)\subseteq X\to X$ be linear and densely defined with $P_nH\subseteq HP_n$.
Then the following conditions are equivalent.\vspace{3pt}
\begin{compactitem}
 \item[(i)] $H$ is essentially skew-self-adjoint in $X$.\vspace{3pt}
 \item[(ii)] $H\mathcal{H}$ is essentially skew-self-adjoint in $X_{\mathcal{H}}$.
\end{compactitem}
\end{cor}
\begin{proof} The claim follows from Theorem \ref{thm:lfw} in view of the fact that $H$ is essentially skew-self-adjoint if and only if $\overline{H}$ and $-\overline{H}$ are maximal dissipative, see Waurick \cite[Proposition 4.5]{Waurick2017}.
\end{proof}

Next, we establish a formula that allows to compute the adjoint of $H\mathcal{H}$ in the weighted space $X_{\mathcal{H}}$. Notice that in the proof we will use again the map $P$ from Proposition \ref{upshot-lemma}. The particular point that $\mathcal{H}$ is neither assumed to be bounded nor bounded below is the most important part of the next statement. It can be considered as the key abstract result of this contribution.

\smallskip

\begin{thm}\label{thm:adj} Let $X$ be a Hilbert space, let $\mathcal{H}$ be a locally finite weight corresponding to $(X_n)_{n\in\mathbb{N}}$. Let $H\colon D(H)\subseteq X\to X$ be densely defined and satisfy $P_nH\subseteq HP_n$ for all $n\in \mathbb{N}$. Assume that $H$ leaves every $X_n$ invariant. Then $(H\mathcal{H})^{\star}=\overline{H^{\star}\mathcal{H}}$ where the adjoint on the left is taken with respect to $X_{\mathcal{H}}$ and the adjoint on the right with respect to $X$.
\end{thm}
\begin{proof} Recall that by \eqref{HH} the operator $H\mathcal{H}$ is given by
$$
\begin{array}{rcl}
D(H\mathcal{H}) &= &\bigl\{\tilde{x}\in X_{\mathcal{H}} \:;\:\exists\:x\in D(\mathcal{H})\colon Px=\tilde{x},\, \mathcal{H}x \in D(H) \text{ and } (P_nH\mathcal{H}x)_{n\in\mathbb{N}}\in X_{\mathcal{H}}\bigl\}\vspace{5pt}\\
H\mathcal{H}\tilde{x} &=& (P_nH\mathcal{H}x)_{n\in\mathbb{N}}
\end{array}
$$
and that $H^{\star}\mathcal{H}$ is defined analogously via
$$
\begin{array}{rcl}
D(H^{\star}\mathcal{H})&=&\bigl\{\tilde{x}\in X_{\mathcal{H}} \:;\:\exists\:x\in D(\mathcal{H})\colon Px=\tilde{x},\, \mathcal{H}x \in D(H^{\star}) \text{ and } (P_nH^{\star}\mathcal{H}x)_{n\in\mathbb{N}}\in X_{\mathcal{H}}\bigl\}\vspace{5pt}\\
H^{\star}\mathcal{H}\tilde{x}&=&(P_nH^{\star}\mathcal{H}x)_{n\in\mathbb{N}}.
\end{array}
$$
Let $\tilde{x}\in D(H^{\star}\mathcal{H})$ and $\tilde{y}\in D(H\mathcal{H})$. We select $x$, $y\in D(\mathcal{H})$ according to the above. Employing the same arguments as in the proof of Proposition \ref{lem:mada} we compute 
\begin{align*}
\langle \tilde{x}, H\mathcal{H} \tilde{y}\rangle_{X_{\mathcal{H}}} & = \langle (P_nx)_{n\in\mathbb{N}}, (P_nH\mathcal{H}y)_{n\in\mathbb{N}}\rangle_{X_{\mathcal{H}}}=\lim_{n\rightarrow\infty}\langle{}P_nx,P_nH\mathcal{H}y\rangle_{\mathcal{H}}\\
& =\lim_{n\rightarrow\infty}\langle{}P_nx,\mathcal{H}P_nH\mathcal{H}y\rangle=\lim_{n\rightarrow\infty}\langle{}\mathcal{H}P_nx,HP_n\mathcal{H}y\rangle\\
& =\lim_{n\rightarrow\infty}\langle{}H^{\star}P_n\mathcal{H}x,\mathcal{H}P_ny\rangle=\lim_{n\rightarrow\infty}\langle{}P_nH^{\star}\mathcal{H}x,P_ny\rangle_{\mathcal{H}}\\
& =\langle{}H^{\star}\mathcal{H}\tilde{x},\tilde{y}\rangle_{X_{\mathcal{H}}}
\end{align*}
where we used in addition that $H^{\star}P_n=(P_nH)^{\star}\supseteq(HP_n)^{\star}=P_nH^{\star}$ holds by our assumptions on $H$. This shows $H^{\star}\mathcal{H}\subseteq (H\mathcal{H})^{\star}$. Since the right-hand side operator is closed, we deduce that $\overline{H^{\star}\mathcal{H}}\subseteq (H\mathcal{H})^{\star}$. Thus, it remains to show
\begin{equation}\label{EQ-P-0}
D((H\mathcal{H})^{\star})\subseteq D(\overline{H^{\star}\mathcal{H}}).
\end{equation}

\smallskip

From Proposition \ref{lem:mada} we know that $\tilde{P}_nH\mathcal{H}\subseteq H\mathcal{H}\tilde{P}_n$ holds for every $n$. This implies $(H\mathcal{H})^{\star}\tilde{P}_n=(\tilde{P}_nH\mathcal{H})^{\star}\supseteq (H\mathcal{H}\tilde{P}_n)^{\star}=\tilde{P}_n(H\mathcal{H})^{\star}$. This implies in particular that $(H\mathcal{H})^{\star}$ leaves $X_{n,\mathcal{H}}$ invariant. Now we consider the operator
$$
H_{n,\mathcal{H}}\colon D(H_{n,\mathcal{H}})\subseteq X_{n,\mathcal{H}}\rightarrow X_{n,\mathcal{H}} \, \text{ with } \, D(H_{n,\mathcal{H}})=D(H\mathcal{H})\cap X_{n,\mathcal{H}}  \, \text{ and } \, H_{n,\mathcal{H}}\tilde{x}=H\mathcal{H}\tilde{x}.
$$
We denote by $H_{n,\mathcal{H}}^{\star}$ the adjoint of $H_{n,\mathcal{H}}$ with respect to the scalar product of $X_{n,\mathcal{H}}$. Still for $\tilde{x}\in D((H\mathcal{H})^{\star})$, $\tilde{y}\in D(H\mathcal{H})$ and $n\in\mathbb{N}$ we compute
$$
\langle \tilde{P}_n \tilde{x}, H_{n,\mathcal{H}} \tilde{P}_n \tilde{y} \rangle_{X_{n,\mathcal{H}}} = \langle \tilde{P}_n \tilde{x}, H\mathcal{H}\tilde{P}_n \tilde{y} \rangle_{\mathcal{H}} = \langle (H\mathcal{H})^{\star}\tilde{P}_n \tilde{x},\tilde{P}_n \tilde{y} \rangle_{\mathcal{H}} = \langle (H\mathcal{H})^{\star}\tilde{P}_n \tilde{x},\tilde{P}_n \tilde{y} \rangle_{X_{n,\mathcal{H}}}
$$
where we used that $X_{n,\mathcal{H}}$ is invariant under $(H\mathcal{H})^{\star}$. From this it follows $\tilde{P}_n \tilde{x} \in D(H_{n,\mathcal{H}}^{\star})$ and  
\begin{equation}\label{EQ-P-1}
H_{n,\mathcal{H}}^{\star}\tilde{P}_n \tilde{x} = (H\mathcal{H})^{\star}\tilde{P}_n \tilde{x}.
\end{equation}
Next we establish the following equality
\begin{equation}\label{EQ-P}
D(H_{n,\mathcal{H}}) = \{ P\mathcal{H}^{-1}P_nz\:;\: z \in D(H) \}.
\end{equation}

\textquotedblleft{}$\supseteq$\textquotedblright{} Let $z\in D(H)$ and consider $\mathcal{H}^{-1}P_nz\in X$. Since $\mathcal{H}$ is an isomorphism from $X_n$ onto itself, it follows that $\mathcal{H}^{-1}P_nz\in X_n$. Thus, $P\mathcal{H}^{-1}P_nz\in X_{n,\mathcal{H}}$. On the other hand $\mathcal{H}^{-1}P_nz\in D(\mathcal{H})$. Now $\mathcal{H}(\mathcal{H}^{-1}P_nz)=P_nz\in D(H)$ holds since $z\in D(H)$ by using $P_nH\subseteq HP_n$. Finally we see that
$$
(P_kH\mathcal{H}(\mathcal{H}^{-1}P_nz))_{k\in\mathbb{N}}=(P_kHP_nz)_{k\in\mathbb{N}}\in X_{\mathcal{H}}
$$
since $P_kHP_nz=HP_nz\in X_n$ for $k\geqslant n$. Therefore it follows that $P\mathcal{H}^{-1}P_nz\in D(H\mathcal{H})$.

\smallskip

\textquotedblleft{}$\subseteq$\textquotedblright{} Let $\tilde{y}\in D(H_{n,\mathcal{H}})=D(H\mathcal{H})\cap X_{n,\mathcal{H}}$. According to the definition of $D(H\mathcal{H})$ we select $y\in D(\mathcal{H})$ such that $Py=\tilde{y}$, $\mathcal{H}y\in D(H)$ and $H\mathcal{H}\tilde{y}=(P_kH\mathcal{H}y)_{k\in\mathbb{N}}\in X_{\mathcal{H}}$. Since $P\colon X_n\rightarrow X_{n,\mathcal{H}}$ is an isomorphism,  we conclude $y\in X_n$. We put $z:=\mathcal{H}y$. Then $z\in D(H)$ and
$$
P\mathcal{H}^{-1}(P_nz)=P\mathcal{H}^{-1}(P_n\mathcal{H}y)=P\mathcal{H}^{-1}(\mathcal{H}P_ny)=Py=\tilde{y}
$$
establishes \eqref{EQ-P}.

\smallskip

Next we claim
\begin{equation}\label{EQ-P-2}
H_{n,\mathcal{H}}^{\star}\tilde{x} = H^{\star}\mathcal{H}\tilde{x} \;\text{ for } \; \tilde{x}\in D(H_{n,\mathcal{H}}^{\star})
\end{equation}

For this, let $z\in D(H)$. Employing \eqref{EQ-P} we get $\tilde{y}:=P\mathcal{H}^{-1}P_nz\in D(H_{n,\mathcal{H}})$. Since $\tilde{y}\in X_{n,\mathcal{H}}$ we get that $y:=P^{-1}\tilde{y}$ belongs to $X_n$ from whence it follows that $\tilde{y}\in D(H\mathcal{H})$ with $Py=\tilde{y}$, $y\in D(\mathcal{H})$, $\mathcal{H}y\in D(H)$ and $(P_kH\mathcal{H}y)_{k\in\mathbb{N}}\in X_{\mathcal{H}}$. Since $\mathcal{H}$ and $H$ leave $X_n$ invariant, we have $H\mathcal{H}\tilde{y}=(P_kH\mathcal{H}y)_{k\in\mathbb{N}}=PH\mathcal{H}y\in X_{n,\mathcal{H}}\subseteq X_{\mathcal{H}}$.

\smallskip

For $\tilde{x}\in D(H_{n,\mathcal{H}}^{\star})$ we get analogously that $x:=P^{-1}\tilde{x}$ belongs to $X_n$. 
%Therefore $\mathcal{H}x\in X_n$ and since we showed already that $P_nH^{\star}\subseteq H^{\star}P_n$ holds, we get $\mathcal{H}x\in D(H^{\star})$ and $H^{\star}\mathcal{H}x\in X_n$. The sequence $(P_kH^{\star}\mathcal{H}x)_{k\in\mathbb{N}}$ is thus eventually constant and defines an element of $X_{\mathcal{H}}$. We thus have $H^{\star}\mathcal{H}\tilde{x}=(P_kH^{\star}\mathcal{H}x)_{k\in\mathbb{N}}=PH^{\star}\mathcal{H}x\in X_{n,\mathcal{H}}\subseteq X_{\mathcal{H}}$.
Using $P_n\mathcal{H}\subseteq\mathcal{H}P_n$ and $P_nH\subseteq HP_n$ we compute
\begin{align*}
\langle \mathcal{H}x, H z \rangle & = \langle x, HP_n z\rangle_{\mathcal{H}} = \langle x, H\mathcal{H}\mathcal{H}^{-1}P_n z\rangle_{\mathcal{H}}=\langle x, H\mathcal{H}P^{-1}P\mathcal{H}^{-1}P_n z\rangle_{\mathcal{H}}= \langle x, H\mathcal{H}P^{-1}\tilde{y}\rangle_{\mathcal{H}}\\
& = \langle x, H\mathcal{H}y\rangle_{\mathcal{H}} = \langle P^{-1}Px, P^{-1}PH\mathcal{H}y\rangle_{\mathcal{H}} = \langle P^{-1}\tilde{x}, P^{-1}H\mathcal{H}\tilde{y}\rangle_{\mathcal{H}}=\langle \tilde{x}, H_{n,\mathcal{H}}\tilde{y}\rangle{}_{X_{n,\mathcal{H}}}\\
& = \langle H_{n,\mathcal{H}}^{\star}\tilde{x},\tilde{y}\rangle{}_{X_{n,\mathcal{H}}} = \langle H_{n,\mathcal{H}}^{\star}\tilde{x},\tilde{y}\rangle{}_{X_{\mathcal{H}}} = \langle H_{n,\mathcal{H}}^{\star}\tilde{x},P\mathcal{H}^{-1}P_nz\rangle{}_{X_{\mathcal{H}}} \\
& = \langle PP^{-1}H_{n,\mathcal{H}}^{\star}\tilde{x},P\mathcal{H}^{-1}P_nz\rangle{}_{X_{\mathcal{H}}}=  \langle P^{-1}H_{n,\mathcal{H}}^{\star}\tilde{x},\mathcal{H}^{-1}P_nz\rangle{}_{\mathcal{H}}\\
&= \langle P^{-1}H_{n,\mathcal{H}}^{\star}\tilde{x},P_nz\rangle{}
\end{align*}

which shows that $\mathcal{H}x\in D(H^{\star})$ holds since $P_n\colon X\rightarrow X_n$ is continuous. We observe that in view of \eqref{EQ-P} we can find $z\in X_n$ with $\tilde{y}=P\mathcal{H}^{-1}P_nz$. With such a $z$ we compute
\begin{align*}
\langle H_{n,\mathcal{H}}^{\star}\tilde{x},\tilde{y}\rangle{}_{X_{n,\mathcal{H}}} & = \langle \mathcal{H}x, H z \rangle = \langle H^{\star}\mathcal{H}x, z \rangle =  \langle H^{\star}\mathcal{H}x, P_nz \rangle = \langle P^{-1}PH^{\star}\mathcal{H}x, P^{-1}P\mathcal{H}^{-1}P_nz \rangle_{\mathcal{H}}\\
& =  \langle P^{-1}H^{\star}\mathcal{H}\tilde{x}, P^{-1}\tilde{y} \rangle_{\mathcal{H}} =  \langle H^{\star}\mathcal{H}\tilde{x}, \tilde{y} \rangle_{X_{n,\mathcal{H}}} 
\end{align*}
which implies that \eqref{EQ-P-2} holds.

\smallskip

Combining \eqref{EQ-P-2} with \eqref{EQ-P-1} shows that $(H\mathcal{H})^{\star}\tilde{P}_n=H^{\star}\mathcal{H}\tilde{P}_n$ holds for every $n\in\mathbb{N}$. Now we prove \eqref{EQ-P-0}. Let $\tilde{x}\in D((H\mathcal{H})^{\star})$. We consider $\tilde{P}_n\tilde{x}\rightarrow \tilde{x}$ for $n\rightarrow\infty$ in $X_{\mathcal{H}}$. On the other hand we compute
$$
H^{\star}\mathcal{H}\tilde{P}_n\tilde{x}=(H\mathcal{H})^{\star}\tilde{P}_n\tilde{x}=\tilde{P}_n(H\mathcal{H})^{\star}\tilde{x}\rightarrow (H\mathcal{H})^{\star}\tilde{x} 
$$
for $n\rightarrow\infty$ in $X_{\mathcal{H}}$ since $\tilde{x}\in D((H\mathcal{H})^{\star})$ holds by assumption. This shows $(\tilde{x},(H\mathcal{H})^{\star}\tilde{x})\in\overline{\gr(H^{\star}\mathcal{H})}$ and thus $\tilde{x}\in D(\overline{H^{\star}\mathcal{H}})$.
\end{proof}

For later use we mention the following very easy case of Theorem \ref{thm:adj}.

\smallskip

\begin{cor}\label{thm:adj-easy-cor} Let $X$ be a Hilbert space, let $\mathcal{H}$ be a uniformly finite weight, see Section \ref{SEC:4:1}. Let $H\colon D(H)\subseteq X\to X$ be densely defined. Then $(H\mathcal{H})^{\star}=H^{\star}\mathcal{H}$ where the adjoint on the left is taken with respect to $X_{\mathcal{H}}$ and the adjoint on the right with respect to $X$.
\end{cor}
\begin{proof} It is enough to apply Theorem \ref{thm:adj} with $X_n=X$ and $P_n=\id_X$ for all $n\in\mathbb{N}$. As $\mathcal{H}\colon X\rightarrow X$ is an isomorphism, $H^{\star}\mathcal{H}$ is then already closed.
\end{proof}

\medskip

%%%%%%%%%%%%%%%%%%%%%%%%%%%%%%%%%%%%%%%%%%%%%%%%%%%
%%%%%%%%%%%%%%%%%%%%%%%%%%%%%%%%%%%%%%%%%%%%%%%%%%%
%%                                               %%
%% 6 The port-hamiltonian operator on intervals  %%
%%                                               %%
%%%%%%%%%%%%%%%%%%%%%%%%%%%%%%%%%%%%%%%%%%%%%%%%%%%
%%%%%%%%%%%%%%%%%%%%%%%%%%%%%%%%%%%%%%%%%%%%%%%%%%%

\section{The port-Hamiltonian operator on intervals}\label{SEC:3}\smallskip

In this section we define the port-Hamiltonian operator on finite intervals $[0,b]$ and on the semi-axis $[0,\infty)$. Then we compute its adjoints as this is necessary in order to apply Theorem \ref{CLASS-THM} later on in Section \ref{SEC:5}.

\medskip

Let $d\geqslant1$ be a fixed integer and let $I\subseteq\mathbb{R}$ be a possibly unbounded interval. Let $\mathcal{H}\colon I\rightarrow\mathbb{C}^{d\times d}$ be measurable such that for almost every $\xi\in I$ the matrix $\mathcal{H}(\xi)$ is Hermitian. Assume that
\begin{equation}\label{STANDARD-ASS}
\begin{array}{c}\vspace{3pt}
\forall\;K\subset{}I\;\text{bounded}\;\exists\:m,\,M>0\;\forall\:\zeta\in\mathbb{C}^d \colon\\m|\zeta|^2\leqslant \zeta^{\star}\mathcal{H}(\cdot)\zeta\leqslant M|\zeta|^2 \text{ holds almost everywhere on } K\cap I
\end{array}
\end{equation}
holds. This implies that the standard assumptions of, e.g., \cite{AJ2014, GZM2005, JMZ2015, JZ} are satisfied on every bounded interval. Next we define weighted and unweighted $\Ls^2$-spaces. Unless  otherwise stated, the functions in these spaces will always be $\mathbb{C}^{d}$-valued. We consider
$$
\Ls^2_{\mathcal{H}}(I) := \bigl\{x \in \Ls_{\textnormal{loc}}^2(I)\:;\:x \text{ measurable and } \|x\|_{\Ls^2_{\mathcal{H}}(I)}^2=\int_{I}x(\xi)^{\star}\mathcal{H}(\xi)x(\xi)\dd\xi<\infty\bigr\}
$$
which is a Hilbert space with respect to the scalar product
$$
\langle{}x,y\rangle{}_{\Ls^2_{\mathcal{H}}(I)}:=\int_{I}x(\xi)^{\star}\mathcal{H}(\xi)y(\xi)\dd\xi,
$$
where $x(\xi)^\star$ denotes the transpose of the complex conjugate vector of $x(\xi)$.
We note that $\mathcal{H}$ is a locally finite weight and that $\Ls^2_{\mathcal{H}}(I) = \Ls^2(I)_{\mathcal{H}}$ holds if we employ our previous notation of weighted Hilbert spaces $X_\mathcal{H}$. We mention the following fact for later use; we will have occasion to look into a more refined variant of the continuity statement in Lemma \ref{lem:SE} below.

\smallskip

\begin{lem}\label{lem:Sob} Let $x\in \Ls_{\textnormal{loc}}^2(I)$ and $x'\in \Ls_{\textnormal{loc}}^2(I)$ in the distributional sense. Then $x\in \operatorname{C}(I)$. Moreover for all $a,b\in I$ with $a<b$ and $x,y, x',y'\in \Ls_{\textnormal{loc}}^2(I)$ we have
\[
   \int_a^b x'y = x(b)y(b)-x(a)y(a)-\int_a^b xy'.
\]
\end{lem}
\begin{proof}
 The first statement follows from the fact that $\operatorname{H}^1(c,d)\subseteq \operatorname{C}[c,d]$ continuously whenever $-\infty<c<d<\infty$. The integration by parts formula follows from the density of $\operatorname{C}^1[c,d]$ in $\operatorname{H}^1(c,d)$ and from the fact that the point evaluation of $\operatorname{H}^1$-functions is continuous.
\end{proof}

Let now $I$ be either $[0,b]$ for $0<b<\infty$ or $[0,\infty)$. Let $P_1$, $P_0\in\mathbb{C}^{d\times d}$ with $P_1^{\star}=P_1$ invertible and $P_0^{\star}=-P_0$. We define the port-Hamilonian operator $A\colon D(A)\subseteq\Ls^2_{\mathcal{H}}(I)\rightarrow\Ls^2_{\mathcal{H}}(I)$ via
$$
Ax=\displaystyle P_1(\mathcal{H}x)'+P_0\mathcal{H}x \;\text{ with }\; D(A)=\bigl\{x\in \Ls^2_{\mathcal{H}}(I)\:;\: \mathcal{H}x,\,(\mathcal{H}x)'\in\Ls^2_{\mathcal{H}}(I)\text{ and } \mathcal{H}x|_{\partial I}=0\bigr\}
$$
where we understand $(\mathcal{H}x)'\in\Ls^2_{\mathcal{H}}(I)$ in the sense of distributions and notice that the evaluation at zero, or, respectively, at zero and $b$, is well-defined in view of Lemma \ref{lem:Sob}. In addition we notice that the condition $\mathcal{H}x\in \Ls^2_{\mathcal{H}}(I)$ in the definition of $D(A)$ can be dropped if $I=[0,b)$ or if $\mathcal{H}$ is bounded; the latter we will assume for a part of our results below.

\smallskip

In view of the boundary condition $\mathcal{H}x|_{\partial I}=0$ the operator above corresponds to the minimal operator $H_0$ in the context of Section \ref{SEC:2}. We will turn back to this notation in  Section \ref{SEC:5} when we actually consider extensions. For this, however, we firstly need to compute the adjoint $A^{\star}$ of the port-Hamiltonian operator.

\smallskip

We start with the finite interval.

\smallskip

\begin{lem}\label{B-LEM} Let $0<b<\infty$. Let $A\colon D(A)\subseteq\Ls_{\mathcal{H}}^2(0,b)\rightarrow\Ls_{\mathcal{H}}^2(0,b)$ be given by
\begin{equation*}
\begin{array}{rcl}
Ax&=&\displaystyle P_1({\mathcal{H}}x)'+P_0{\mathcal{H}}x\\[9pt]
D(A)&=&\bigl\{x\in \Ls_{\mathcal{H}}^2(0,b)\:;\: (\mathcal{H}x)'\in \Ls^2_{\mathcal{H}}(0,b) \text{ and } (\mathcal{H}x)(0)=(\mathcal{H}x)(b)=0\bigr\}.
\end{array}
\end{equation*}
Then $A$ is densely defined, closed and its adjoint $A^{\star}\colon D(A^{\star})\subseteq\Ls_{\mathcal{H}}^2(0,b)\rightarrow\Ls_{\mathcal{H}}^2(0,b)$ is given by
\begin{equation*}
\begin{array}{rcl}
A^{\star}x&=&-P_1(\mathcal{H}x)'-P_0 \mathcal{H}x\\[9pt]
D(A^{\star})&=&\bigl\{x\in \Ls_{\mathcal{H}}^2(0,b)\:;\: (\mathcal{H}x)'\in \Ls_{\mathcal{H}}^2(0,b) \bigr\}.
\end{array}
\end{equation*}
\end{lem}
\begin{proof} We consider $\mathcal{H}={1}$, i.e., $A\colon D(A)\subseteq\Ls^2(0,b)\rightarrow\Ls^2(0,b)$ with $Ax=\displaystyle P_1x'+P_0x$ for $x\in D(A)=\{x\in \operatorname{H}^1(0,b)\:;\: x(0)=x(b)=0\}$. Then it is well-known that $A$ is densely defined and closed. Due to Proposition \ref{lem:mada} this carries over to the case where $\mathcal{H}\not={1}$. In view of Corollary \ref{thm:adj-easy-cor} it suffices also to consider $\mathcal{H}={1}$ in order to compute the adjoint. Indeed, it is enough to show that $A^{\star}\colon D(A)\subseteq\Ls^2_{}(0,b)\rightarrow\Ls^2_{}(0,b)$
$$
A^{\star}x=-P_1x'-P_0 x \,\text{ for }\, x\in D(A^{\star})=\operatorname{H}^1(0,b)
$$
holds. Since $P_0 \colon\Ls^2_{}(0,b)\rightarrow\Ls^2_{}(0,b)$, the operator of multiplication by the matrix $P_0$, is bounded, it follows that $A^{\star}=(A-P_0+P_0)^{\star}=(A-P_0)^{\star}+(P_0)^{\star}=(A-P_0)^{\star}-P_0$ and in particular $D((A-P_0)^{\star})=D(A^{\star})$ holds. Thus, without loss of generality $P_0=0$. By Lemma \ref{lem:Sob}, it is easy to see that $-P_1\partial|_{\operatorname{H}^1(0,b)}\subseteq A^{\star}$, where $\partial$ denotes the distributional derivative in $\operatorname{L}^1_{\textnormal{loc}}(0,b)$. Thus, it remains to show the other inclusion. For this, we let $y\in D(A^{\star})$. Then for all $x\in \operatorname{C}_c^\infty(0,b)\subseteq D(A)$, we obtain
\[
\langle A^{\star}y,x\rangle = \langle y,Ax\rangle = \langle y,P_1\partial x\rangle = \langle P_1y,\partial x\rangle.
\]
Hence, $P_1y \in \operatorname{H}^1(0,b)$ and $-\partial P_1 y = A^{\star}y$. Since $P_1$ is an invertible matrix, we deduce $\partial P_1 = P_1\partial$ on $\operatorname{H}^1(0,b)$ and the statement is proved.
\end{proof}

\smallskip

Now we treat the case of a semi-axis and start with the following inclusion.

\smallskip

\begin{lem}\label{UB-LEM} Let $A\colon D(A)\subseteq\Ls^2_{\mathcal{H}}(0,\infty)\rightarrow\Ls^2_{\mathcal{H}}(0,\infty)$ be given by
\begin{equation}\label{OP-INTRO-INFINITE}
\begin{array}{rcl}
Ax&=&\displaystyle P_1(\mathcal{H}x)'+P_0\mathcal{H}x\\[9pt]
D(A)&=&\bigl\{x\in \Ls^2_{\mathcal{H}}(0,\infty)\:;\:(\mathcal{H}x)'\in\Ls^2_{\mathcal{H}}(0,\infty) \text{ and }(\mathcal{H}x)(0)=0\bigr\}.
\end{array}
\end{equation}
Then $A$ is densely defined, closed and its adjoint $A^{\star}\colon D(A^{\star})\subseteq\Ls^2_{\mathcal{H}}(0,\infty)\rightarrow\Ls^2_{\mathcal{H}}(0,\infty)$ satisfies
\begin{equation}\label{DUAL-INFINITE}
\begin{array}{rcl}
A^{\star}x&=&-P_1(\mathcal{H}x)'-P_0\mathcal{H}x\\[9pt]
D(A^{\star})&\subseteq&\bigl\{x\in \Ls^2_{\mathcal{H}}(0,\infty)\:;\: (\mathcal{H}x)' \in \Ls_{\textnormal{loc}}^1(0,\infty)\bigr\}.
\end{array}
\end{equation}
\end{lem}

\begin{proof} It is straightforward to check that $A$ is densely defined and closed; for this we use that $\Ls_{\mathcal{H}}^2(0,\infty)$ embeds continuously into $\Ls_{\textnormal{loc}}^1(0,\infty)$ and that $\mathcal{H}$ continuously maps $\Ls_{\textnormal{loc}}^1(0,\infty)$ into itself. For $\phi\in \operatorname{C}_c^\infty(0,\infty)$ and $x\in D(A^\star)$ we compute
\begin{align*}
  \langle \phi, A^\star x \rangle & =  \langle \mathcal{H}^{-1}\phi, A^\star x \rangle_{\mathcal{H}} 
   \\ & =
   \langle A \mathcal{H}^{-1}\phi, x\rangle_{\mathcal{H}} 
   \\ & =    \langle P_1 (\mathcal{H} \mathcal{H}^{-1}\phi)'+P_0\phi, \mathcal{H}x\rangle 
   \\ & =    \langle P_1 \phi'+P_0\phi,\mathcal{H}x\rangle
   \\ & =    \langle \phi',P_1 \mathcal{H}x\rangle- \langle P_0\phi,\mathcal{H}x\rangle .
\end{align*}
Hence
$$
\langle \phi, A^\star x + P_0 \mathcal{H}x \rangle = \langle \phi',P_1 \mathcal{H}x\rangle
$$
and we see that $P_1\mathcal{H}x$ is weakly differentiable with $-\partial P_1\mathcal{H}x = A^\star x + P_0\mathcal{H} x$. This establishes the formula  $A^\star x =-\partial P_1\mathcal{H}x- P_0\mathcal{H}x$.
\end{proof}

If $\mathcal{H}={1}$, then a boundary condition at infinity comes automatically via the classical Barb{\u a}lat lemma, see, e.g., Farkas, Wegner \cite[Theorem 5]{FW}, that states that $x(\xi)\rightarrow0$ for $\xi\rightarrow\infty$ holds if $x$, $x'\in\Ls^2(0,\infty)$. An adapted version of this result can be seen as follows.

\smallskip

\begin{lem}\label{LEM-0} Let $x\in\Ls^2_{\mathcal{H}}(0,\infty)$ such that $(\mathcal{H}x)'\in\Ls^2_{\mathcal{H}}(0,\infty)$. (a) Then $\mathcal{H}x$ is bounded. (b) If $\mathcal{H}$ is bounded, then $\mathcal{H}x$ vanishes at infinity. 
\end{lem}

\begin{rem}
Note that the conditions in Lemma \ref{LEM-0} are sharp. Indeed, let $\mathcal{H}(\xi)=1+\xi^2$ and $x(\xi)=1/(1+\xi^2)$. Then $x\in L_\mathcal{H}^2(0,\infty)$ and $(\mathcal{H}x)'=0\in L_\mathcal{H}^2(0,\infty)$. $\mathcal{H}x=1$ is bounded, but does not vanish at $\infty$. 
\end{rem}

\begin{proof}[Proof of Lemma \ref{LEM-0}] (a)~The assumptions imply that $(\mathcal{H}x)'\in \Ls^2_{\textnormal{loc}}[0,\infty)$. Consequently, from
$$
|\mathcal{H}x|^2=\langle{}\mathcal{H}x,\mathcal{H}x\rangle{}_{\mathbb{C}^d}=\sum_{k=1}^n(\mathcal{H}x)_k(\overline{\mathcal{H}x})_k,
$$ we read off $(|\mathcal{H}x|^2)'\in \Ls_{\textnormal{loc}}^1[0,\infty)$. For $\xi,\,\eta\in[0,\infty)$ we thus have
$$
 |(\mathcal{H}x)(\xi)|^2-|(\mathcal{H}x)(\eta)|^2 = \int_{\eta}^{\xi}\frac{\dd}{\dd\zeta}|(\mathcal{H}x)(\zeta)|^2\dd\zeta.
$$
Now we compute
\begin{eqnarray*}
\frac{\dd}{\dd\xi}|(\mathcal{H}x)(\xi)|^2 & = & \langle{}(\mathcal{H}x)'(\xi),(\mathcal{H}x)(\xi)\rangle{}_{\mathbb{C}^d}+\langle{}(\mathcal{H}x)(\xi),(\mathcal{H}x)'(\xi)\rangle{}_{\mathbb{C}^d}\\
& = & 2\Re \langle{}(\mathcal{H}x)'(\xi),(\mathcal{H}x)(\xi)\rangle{}_{\mathbb{C}^d}\\
& = & 2\Re \langle{}(\mathcal{H}^{1/2}(\xi)(\mathcal{H}x)'(\xi),\mathcal{H}^{1/2}(\xi)x(\xi)\rangle{}_{\mathbb{C}^d}.
\end{eqnarray*}
for almost every $\xi\in(0,\infty)$. \red{Now we obtain
$$
\int_{0}^{\xi}\Bigl|\frac{\dd}{\dd\zeta}|(\mathcal{H}x)(\zeta)|^2\Bigr|\dd\zeta \leqslant  2\|(\mathcal{H}x)'\|_{\Ls^2_{\mathcal{H}}(0,\infty)}\|x\|_{\Ls^2_{\mathcal{H}}(0,\infty)}
$$
for $\xi\in[0,\infty)$ by applying the Cauchy--Schwarz inequality.} This shows that $\mathcal{H}x$ is bounded.

\smallskip

(b)~The limit
$$
\lim_{\xi\rightarrow\infty}|(\mathcal{H}x)(\xi)|^2 = |(\mathcal{H}x)(0)|^2+\lim_{\xi\rightarrow\infty}\int_{0}^{\xi}\frac{\dd}{\dd\zeta}|(\mathcal{H}x)(\zeta)|^2\dd\zeta
$$
exists, since $\int_0^{\infty}\bigl|\frac{\dd}{\dd\zeta}|(\mathcal{H}x)(\zeta)|^2\bigr|\dd\zeta$ is finite by the above estimates. In view of
\begin{eqnarray*}
\int_0^{\infty}|(\mathcal{H}x)(\xi)|^2\dd\xi & = & \int_0^{\infty}|\mathcal{H}^{1/2}(\xi)(\mathcal{H}^{1/2}x)(\xi)|^2\dd\xi\\
& \leqslant & \int_0^{\infty}|\mathcal{H}^{1/2}(\xi)|_{L(\mathbb{C}^d,\mathbb{C}^d)}^2|(\mathcal{H}^{1/2}x)(\xi)|^2\dd\xi\\
& \leqslant & \sup_{\xi\in[0,\infty)} |\mathcal{H}(\xi)|_{L(\mathbb{C}^d,\mathbb{C}^d)}\int_0^{\infty}\bigl\langle{}(\mathcal{H}^{1/2}x)(\xi),(\mathcal{H}^{1/2}x)(\xi)\bigr\rangle{}_{\mathbb{C}^d}\dd\xi\\
& = & \sup_{\xi\in[0,\infty)} |\mathcal{H}(\xi)|_{L(\mathbb{C}^d,\mathbb{C}^d)}\int_0^{\infty}\bigl\langle{}x(\xi),(\mathcal{H}x)(\xi)\bigr\rangle{}_{\mathbb{C}^d}\dd\xi\\
& = & \sup_{\xi\in[0,\infty)} |\mathcal{H}(\xi)|_{L(\mathbb{C}^d,\mathbb{C}^d)}\cdot\|x\|_{\Ls^2_{\mathcal{H}}[0,\infty)}
\end{eqnarray*}
we see that $\mathcal{H}x\in \Ls^2(0,\infty)$. Since $\lim_{\xi\rightarrow\infty}|(\mathcal{H}x)(\xi)|$ exists by the above, it needs then to be zero.
\end{proof}

Using the above we can now prove the remaining inclusion and determine the adjoint of the port-Hamiltonian operator on the semi-axis.

\smallskip 

\begin{lem}\label{thm:P00} Let $\mathcal{H}$ be as in \eqref{STANDARD-ASS} and be bounded. Let $A\colon D(A)\subseteq \Ls_\mathcal{H}^2(0,\infty)\to  \Ls_\mathcal{H}^2(0,\infty)$ be given as in \eqref{OP-INTRO-INFINITE}. Then $A^\star$ is given by
\begin{equation}\label{EQ-NEU}
\begin{array}{rcl}
A^{\star}x&=&-P_1(\mathcal{H}x)'-P_0\mathcal{H}x\\[9pt]
D(A^{\star})&=&\bigl\{x\in \Ls^2_{\mathcal{H}}(0,\infty)\:;\: (\mathcal{H}x)' \in \Ls^2_\mathcal{H}(0,\infty)\bigr\}.
\end{array}
\end{equation}
\end{lem}
\begin{proof} If $\mathcal{H}$ is bounded and $x\in \Ls^2_\mathcal{H}(0,\infty)$, then $P_0\mathcal{H}x\in \Ls^2_{\mathcal{H}}(0,\infty)$. Thus, by the formula for the adjoint from Lemma \ref{UB-LEM}, we infer that $D(A^\star)\subseteq \bigl\{x\in \Ls^2_{\mathcal{H}}(0,\infty)\:;\: (\mathcal{H}x)' \in \Ls^2_\mathcal{H}(0,\infty)\bigr\}$. For the remaining inclusion we take $x\in \Ls_\mathcal{H}^2(0,\infty)$ with $\partial\mathcal{H}x\in \Ls^2_\mathcal{H}(0,\infty)$. By Lemma \ref{LEM-0} we get that $\mathcal{H}x(R)\to 0$ as $R\to\infty$. By using integration by parts it follows $x\in D(A^{\star})$.
\end{proof}

\smallskip

As a preparation for the Section \ref{SEC:5} we reformulate and summarize the above results as follows. The only remaining cases of $I=(-\infty,0]$ and $I=\mathbb{R}$ can be dealt with analogously.

\smallskip

\begin{prop}\label{SKEW-INT} Let $I\subseteq\mathbb{R}$ be an interval and let $\mathcal{H}$ satisfy \eqref{STANDARD-ASS} and be bounded. The operator $A\colon D(A)\subseteq\Ls^2_{\mathcal{H}}(I)\rightarrow\Ls^2_{\mathcal{H}}(I)$ given by
$$
Ax=\displaystyle P_1(\mathcal{H}x)'+P_0\mathcal{H}x \;\text{ with }\; D(A)=\bigl\{x\in \Ls^2_{\mathcal{H}}(I)\:;\:  (\mathcal{H}x)' \in L_\mathcal{H}^2(I)\text{ and } \mathcal{H}x|_{\partial I}=0\bigr\}
$$
is skew-symmetric and for its skew-adjoint $-A^{\star}\colon D(-A^{\star})\subseteq\Ls^2_{\mathcal{H}}(I)\rightarrow\Ls^2_{\mathcal{H}}(I)$
$$
-A^{\star}x=\displaystyle P_1(\mathcal{H}x)'+P_0\mathcal{H}x \;\text{ with }\; D(-A^{\star})=\bigl\{x\in \Ls^2_{\mathcal{H}}(I)\:;\: (\mathcal{H}x)'\in\Ls^2_{\mathcal{H}}(I)\bigr\}.
$$
holds.\hfill\qed
\end{prop}

\smallskip

Finally we want to treat the case of unbounded $\mathcal{H}$. This is possible if we assume that $P_0=0$. 

\smallskip

\begin{lem}\label{NUN-DOCH-UB-LEM} Let $\mathcal{H}$ be as in \eqref{STANDARD-ASS}. Let $A\colon D(A)\subseteq \Ls_\mathcal{H}^2(0,\infty)\to  \Ls_\mathcal{H}^2(0,\infty)$ be given as in \eqref{OP-INTRO-INFINITE} with $P_0=0$. Then $A^\star$ is given by
\begin{equation}\label{EQ-NEU2}
\begin{array}{rcl}
A^{\star}x&=&-P_1(\mathcal{H}x)'\\[9pt]
D(A^{\star})&\subseteq&\bigl\{x\in \Ls^2_{\mathcal{H}}(0,\infty)\:;\: (\mathcal{H}x)' \in \Ls^2_\mathcal{H}(0,\infty)\bigr\}.
\end{array}
\end{equation}
\end{lem}
\begin{proof} That $\mathcal{H}$ is bounded is only needed to conclude from $A^{\star}x=-\partial P_1\mathcal{H}x-P_0\mathcal{H}x\in \Ls^2_{\mathcal{H}}(0,\infty)$ that $(\mathcal{H}x)'\in \Ls^2_{\mathcal{H}}(0,\infty)$. But this does not require even a proof if $P_0=0$.
\end{proof}

\smallskip

Using the above we can establish the adjoint for the case that $\mathcal{H}$ is arbitrary and $P_0=0$. Recall that a function $x\colon I\rightarrow\mathbb{C}^d$ \textit{vanishes at infinity} of $I$, if for each $\epsilon>0$ there exists $K\subseteq I$ compact such that $|x(\xi)|<\epsilon$ holds for all $\xi\in I\backslash K$.

\smallskip

\begin{prop}\label{thm:P00-1} Let $I\in\{[a,b], [a,\infty),\,(-\infty,b], \mathbb{R}\}$ for $a$, $b\in\mathbb{R}$ and let $\mathcal{H}$ satisfy \eqref{STANDARD-ASS}. The operator $A\colon D(A)\subseteq\Ls^2_{\mathcal{H}}(I)\rightarrow\Ls^2_{\mathcal{H}}(I)$ given by
$$
Ax=\displaystyle P_1(\mathcal{H}x)' \;\text{ with }\; D(A)=\bigl\{x\in \Ls^2_{\mathcal{H}}(I)\:;\:  (\mathcal{H}x)'\in L_\mathcal{H}^2(I)\text{ and } \mathcal{H}x|_{\partial I}=0\bigr\}
$$
is skew-symmetric and for its skew-adjoint $-A^{\star}\colon D(-A^{\star})\subseteq\Ls^2_{\mathcal{H}}(I)\rightarrow\Ls^2_{\mathcal{H}}(I)$
$$
-A^{\star}x=\displaystyle P_1(\mathcal{H}x)' \;\text{ with }\; D(-A^{\star})=\bigl\{x\in \Ls^2_{\mathcal{H}}(I)\:;\: (\mathcal{H}x)'\in\Ls^2_{\mathcal{H}}(I) \text{ and } \mathcal{H}x \text{ vanishes at } \infty \text{ of } I\bigr\}.
$$
holds.
\end{prop}
\begin{proof} It is enough to consider $I=[0,\infty)$. By Lemma \ref{NUN-DOCH-UB-LEM} it suffices to show that for $x\in \Ls_\mathcal{H}^2(0,\infty)$ with $\partial\mathcal{H}x\in \Ls^2_\mathcal{H}(0,\infty)$ we have $x \in D(A^\star)$ if and only if $\mathcal{H}x(R)\to 0$ as $R\to\infty$. For this, let $x\in \Ls_\mathcal{H}^2(0,\infty)$ with $\partial\mathcal{H}x\in \Ls^2_\mathcal{H}(0,\infty)$. Assume $x\in D(A^\star)$. We find $\tilde{x}\in \Ls_\mathcal{H}^2(0,\infty)$ with $\partial\mathcal{H}\tilde x \in \Ls^2_\mathcal{H}(0,\infty)$ and $\mathcal{H}\tilde{x}(0)=0$ and $\mathcal{H}x=P_1^{-1}\mathcal{H}\tilde{x}$ on $[1,\infty)$. It follows that $\tilde{x}\in D(A)$. Hence, by the rule of integration by parts, we deduce that for all $R\geqslant 1$ 
\begin{align*}
    \langle AP_1^{-1}\tilde x,P_1^{-1}\chi_{[0,R]}\mathcal{H}x\rangle & = \langle \mathcal{H}\tilde{x},-\chi_{[0,R]} (\mathcal{H}x)' \rangle + (\mathcal{H}x)(R)(\mathcal{H}x)(R) \\ &=\langle P_1^{-1}\mathcal{H}\tilde x,-\chi_{[0,R]} P_1(\mathcal{H}x)' \rangle + (\mathcal{H}x)(R)(\mathcal{H}x)(R)
    \\ & = \langle P_1^{-1}\mathcal{H}\tilde x,\chi_{[0,R]}A^\star x\rangle+ (\mathcal{H}x)(R)(\mathcal{H}x)(R).
\end{align*}
Letting $R\to\infty$ yields
\[
       \langle A\tilde x, P_1^{-1}\mathcal{H}x\rangle=\langle P_1^{-1}\mathcal{H}\tilde x,A^\star x\rangle,
\]
and so, necessarily, $\mathcal{H}x(R)\to 0$ as $R\to\infty$. The remaining implication follows from the fact that $\mathcal{H}x$ is bounded for all $x\in \Ls_\mathcal{H}^2(0,\infty)$ with $(\mathcal{H}x)'\in \Ls_\mathcal{H}^2(0,\infty)$ by Lemma \ref{LEM-0} and again using integration by parts.
\end{proof}

\bigskip

%%%%%%%%%%%%%%%%%%%%%%%%%%%%%%%%%%%%%%%%%%%%%%%%%%%
%%%%%%%%%%%%%%%%%%%%%%%%%%%%%%%%%%%%%%%%%%%%%%%%%%%
%%                                               %%
%% 6 Port-Hamiltonian systems on networks      I  %%
%%                                               %%
%%%%%%%%%%%%%%%%%%%%%%%%%%%%%%%%%%%%%%%%%%%%%%%%%%%
%%%%%%%%%%%%%%%%%%%%%%%%%%%%%%%%%%%%%%%%%%%%%%%%%%%

\section{Port-Hamiltonian operators on networks I}\label{SEC:5}\smallskip

In the remainder of this article we study when the port-Hamiltonian operator generates a contraction semigroup. This question has attracted a lot of interest in the past. We restrict to the latest papers and mention that Jacob, Morris, Zwart \cite{JMZ2015} treat finite intervals and finite graphs with finite intervals as edges, Jacob, Kaiser \cite{JK} treat the semi-axis under the assumption that $\mathcal{H}$ is bounded and bounded away from zero, and infinite networks with finite intervals as edges. Results on the semi-axis with $\mathcal{H}$ being bounded but not necessarily being bounded away from zero also exist, see Jacob, Wegner \cite{JW2018}, but characterize when the port-Hamiltonian operator generates a (possibly non-contractive) $\Cnull$-semigroup.

\smallskip

In this first part of our discussion on port-Hamiltonian operators on networks, we revisit the semi-axis and then consider networks, where we restrict to the case that the edge lengths have a positive lower bound. This appears to be the ``standard assumption'' in most of the literature, see Schubert, Veseli\'c, Lenz \cite{LSV2014}. The next section will be devoted to the technically more demanding case of networks with arbitrarily small edges. Here rather little seems to be known, see \cite[Appendix]{LSV2014}. 

\smallskip

For our convenience we shall use $\langle x,y\rangle_{\Ls^2(0,\infty)}$ or expressions similar to that also for $x,y\notin \Ls^2(0,\infty)$, but $x^\star y \in \Ls^1(0,\infty)$.

\smallskip

%%%%%%%%%%%%%%%%%%%%%%%%%%%%%%%%%%%%%%%%%%%%%%%%%%%
%%%%%%%%%%%%%%%%%%%%%%%%%%%%%%%%%%%%%%%%%%%%%%%%%%%
%%                                               %%
%% 6.1 The special case of the semi-axis         %%
%%                                               %%
%%%%%%%%%%%%%%%%%%%%%%%%%%%%%%%%%%%%%%%%%%%%%%%%%%%
%%%%%%%%%%%%%%%%%%%%%%%%%%%%%%%%%%%%%%%%%%%%%%%%%%%

\subsection{The semi-axis}\label{SEC:4}\smallskip

Consider the situation of Proposition \ref{SKEW-INT} with $I=[0,\infty)$. Let $S\in\mathbb{C}^{d\times{}d}$ be unitary be such that $P_1=S^{\star}\Delta{}S$ holds and $\Delta=\diag(\Lambda,-\Theta)$ where $\Lambda=\diag(\lambda_1,\dots,\lambda_{n_+})\in\mathbb{C}^{n_+\times n_+}$ and $\Theta=\diag(\theta_1,\dots,\theta_{n_-})\in\mathbb{C}^{n_-\times n_-}$ with $\lambda_k$, $\theta_j>0$ for $k\in\{1,\dots,n_+\}$, $j\in\{1,\dots,n_-\}$. We define $\mathcal{G}_1:=\mathbb{C}^{n_+}$ with scalar product $\langle{}\cdot,\cdot\rangle_{\mathcal{G}_1}:=\langle{}\cdot,\Lambda{}\,\cdot\rangle{}$ and $\mathcal{G}_2:=\mathbb{C}^{n_-}$ with scalar product $\langle{}\cdot,\cdot\rangle_{\mathcal{G}_2}:=\langle{}\cdot,\Theta{}\,\cdot\rangle{}$. We put $G_1=\{0\}$ if $n_-=0$ and $G_2=\{0\}$ if $n_+=0$. We denote by
$$
\pr_+\colon\mathbb{C}^d\rightarrow\mathbb{C}^{n_+},\;\text{ and }\;\pr_+(x):=(x_k)_{k\in\{1,\dots,{n_+}\}},\;\;\pr_-\colon\mathbb{C}^d\rightarrow\mathbb{C}^{n_-},\;\;\pr_-(x):=(x_k)_{k\in\{n_++1,\dots,{d}\}}
$$
the projections on the first $n_+$ and the last $n_-$ coordinates, respectively. Let $\Omega$ be the standard symmetric form and let $\omega$ be the standard unitary form on $\mathcal{G}_1\oplus\mathcal{G}_2$. Now we define
$$
F\colon \gr(A^{\star})\rightarrow\mathcal{G}_1\oplus\mathcal{G}_2, \;\; (x,A^{\star}x)\mapsto\bigl(\pr_+(S\mathcal{H}x)(0),\pr_-(S\mathcal{H}x)(0)\bigr).
$$
This defines a boundary system.

\smallskip

\begin{lem}\label{BS-SEMI-AXIS} The quadrupel $(\Omega, \mathcal{G}_1, \mathcal{G}_2, F, \omega)$ is a boundary system for $A$.
\end{lem}

\begin{proof} We compute
\begin{eqnarray*}
\Omega\bigl((x,A^{\star}x),(y,A^{\star}y)\bigr) & = & \langle{}x,A^{\star}y\rangle_{\Ls^2_{\mathcal{H}}(0,\infty)} + \langle{}A^{\star}x,y\rangle_{\Ls^2_{\mathcal{H}}(0,\infty)}\\
& = & \langle{}x,-P_1(\mathcal{H}y)'-P_0\mathcal{H}y\rangle_{\Ls^2_{\mathcal{H}}(0,\infty)} + \langle{}-P_1(\mathcal{H}x)'-P_0\mathcal{H}x,y\rangle_{\Ls^2_{\mathcal{H}}(0,\infty)}\\
& = & -\langle{}x,P_1(\mathcal{H}y)'\rangle_{\Ls^2_{\mathcal{H}}(0,\infty)}-\langle{}x,P_0\mathcal{H}y\rangle_{\Ls^2_{\mathcal{H}}(0,\infty)}\\
& & \phantom{xxxxxxx}- \langle{}P_1(\mathcal{H}x)',y\rangle_{\Ls^2_{\mathcal{H}}(0,\infty)}-\langle{}P_0\mathcal{H}x,y\rangle_{\Ls^2_{\mathcal{H}}(0,\infty)}\\
& = & -\langle{}x,\mathcal{H}P_1(\mathcal{H}y)'\rangle_{\Ls^2(0,\infty)}-\langle{}x,\mathcal{H}P_0\mathcal{H}y\rangle_{\Ls^2(0,\infty)}\\
& & \phantom{xxxxxxx}- \langle{}P_1(\mathcal{H}x)',\mathcal{H}y\rangle_{\Ls^2(0,\infty)}-\langle{}P_0\mathcal{H}x,\mathcal{H}y\rangle_{\Ls^2(0,\infty)}\\
& = & -\langle{}x,\mathcal{H}P_1(\mathcal{H}y)'\rangle_{\Ls^2(0,\infty)}-\langle{}x,\mathcal{H}P_0\mathcal{H}y\rangle_{\Ls^2(0,\infty)}\\
& & \phantom{xxxxxxx}- \langle{}P_1(\mathcal{H}x)',\mathcal{H}y\rangle_{\Ls^2(0,\infty)}+\langle{}x,\mathcal{H}P_0\mathcal{H}y\rangle_{\Ls^2(0,\infty)}\\
& = & -\int_0^{\infty} x(\xi)^{\star}\mathcal{H}(\xi)P_1(\mathcal{H}y)'(\xi)\dd\xi - \int_0^{\infty}(P_1(\mathcal{H}x)'(\xi))^{\star}\mathcal{H}(\xi)y(\xi)\dd\xi\\
& = & -\lim_{R\rightarrow\infty}\int_0^{R} (\mathcal{H}x)(\xi)^{\star}P_1(\mathcal{H}y)'(\xi)\dd\xi - \int_0^{\infty}(P_1(\mathcal{H}x)'(\xi))^{\star}(\mathcal{H}y)(\xi)\dd\xi\\
& = & -\lim_{R\rightarrow\infty}\Bigl( (\mathcal{H}x)(\xi)^{\star}P_1(\mathcal{H}y)(\xi)\Big|_0^R -\int_0^{R}(\mathcal{H}x)'(\xi)^{\star}P_1(\mathcal{H}y)(\xi)\dd\xi \Bigr)\\
& & \phantom{xxxxxxx}- \int_0^{\infty}(P_1(\mathcal{H}x)'(\xi))^{\star}(\mathcal{H}y)(\xi)\dd\xi\\
& = & -\lim_{R\rightarrow\infty}(\mathcal{H}x)(R)^{\star}P_1(\mathcal{H}y)(R) +(\mathcal{H}x)(0)^{\star}P_1(\mathcal{H}y)(0)\\
& & \phantom{xxxxxxx}+\int_0^{\infty}(\mathcal{H}x)'(\xi)^{\star}P_1(\mathcal{H}y)(\xi)\dd\xi - \int_0^{\infty}(\mathcal{H}x)'(\xi)^{\star}P_1(\mathcal{H}y)(\xi)\dd\xi \\
& = & (S\mathcal{H}x)(0)^{\star}\Delta(S\mathcal{H}y)(0)\\
& = & \begin{bmatrix}\pr_+(S\mathcal{H}x)(0)\\\pr_-(S\mathcal{H}x)(0)\\\end{bmatrix}^{\star} \begin{bmatrix} \Lambda & 0  \\0 & -\Theta  \\\end{bmatrix} \begin{bmatrix}\pr_+(S\mathcal{H}y)(0)\\\pr_-(S\mathcal{H}y)(0)\\\end{bmatrix}\\
& = & \bigl\langle{}\pr_+(S\mathcal{H}x)(0),\pr_+(\mathcal{H}y)(0)\bigr\rangle{}_{\mathcal{G}_1}-\bigl\langle{}\pr_-(S\mathcal{H}x)(0),\pr_-(\mathcal{H}y)(0)\bigr\rangle{}_{\mathcal{G}_2}\\
& = & \omega\bigl(F(x,A^{\star}x),F(y,A^{\star}y)\bigr)
\end{eqnarray*}
and we observe that $F$ is surjective.
\end{proof}

\smallskip

Applying Theorem \ref{CLASS-THM}, we get the following classification of generators of contraction semigroups.

\smallskip

\begin{thm}\label{GEN-SEMI-AXIS} Let $\mathcal{H}$, $P_1$, $P_0$ be as in Proposition \ref{SKEW-INT} and $I=[0,\infty)$. Let $\Theta$, $\Lambda$, $n_+$ and $n_-$, $A$ be as above. The operator $H\colon D(H)\subseteq \Ls^2_{\mathcal{H}}(0,\infty)\rightarrow\Ls^2_{\mathcal{H}}(0,\infty)$, $-Hx=P_1(\mathcal{H}x)'+P_0\mathcal{H}x$ for $x\in D(H)\supseteq D(A)$ generates a $\Cnull$-semigroup of contractions if and only if there is a contraction $T\colon\mathcal{G}_2\rightarrow\mathcal{G}_1$ such that
$$
D(H)=\bigl\{x\in \Ls^2_{\mathcal{H}}(0,\infty)\:;\: (\mathcal{H}x)'\in\Ls^2_{\mathcal{H}}(0,\infty)\text{ and } T\pr_-(S\mathcal{H}x(0))=\pr_+(S\mathcal{H}x(0))\bigr\}
$$
holds.\hfill\qed
\end{thm}

\smallskip

Notice that $T\in\mathbb{C}^{n_+\times n_-}$ is a contraction from $\mathcal{G}_2$ to $\mathcal{G}_1$ if and only if $\langle{}Tx,\Lambda{}Tx\rangle{}_{\mathbb{C}^{n_+}}\leqslant\langle{}x,\Theta{}x\rangle{}_{\mathbb{C}^{n_-}}$ holds for all $x\in\mathbb{C}^{n_-}$. Therefore, the question if the port-Hamiltonian operator generates a contraction semi-group can be answered via a matrix condition that involves only the diagonalization of the matrix $P_1$.

%%%%%%%%%%%%%%%%%%%%%%%%%%%%%%%%%%%%%%%%%%%%%%%%%%%%%%%
%%%%%%%%%%%%%%%%%%%%%%%%%%%%%%%%%%%%%%%%%%%%%%%%%%%%%%%
%%                                                   %%
%% 6.2 Networks with uniformly positive edge length  %%
%%                                                   %%
%%%%%%%%%%%%%%%%%%%%%%%%%%%%%%%%%%%%%%%%%%%%%%%%%%%%%%%
%%%%%%%%%%%%%%%%%%%%%%%%%%%%%%%%%%%%%%%%%%%%%%%%%%%%%%%

\subsection{Networks with uniformly positive edge length}\label{SEC:positive-edge}

Let $\Gamma=(V,E)$ be a graph, $V$ the set of vertices and $E$ the set of edges. We restrict ourselves to countable many edges and vertices. Multiple edges, infinitely many vertices or edges, as well as edges for which source and target coincide are allowed. Let $a$, $b\colon E\rightarrow\mathbb{R}\cup\{+\infty,-\infty\}$ be maps such that $a(e)<b(e)$ holds for every $e\in E$. Let $E_{\ell}=\{e\in E\:;\:a(e)>-\infty\}$ and $E_{r}=\{e\in E\:;\:b(e)<+\infty\}$ be the sets of edges where the associated interval $(a(e),b(e))$ is bounded from below, respectively from above.

\smallskip

Let
$$
\mathcal{H}\colon \bigcupdot_{e\in E}(a(e),b(e))\rightarrow \mathbb{C}^{d\times d}
$$
be a map. Denote by $\mathcal{H}_e\colon(a(e),b(e))\rightarrow\mathbb{C}^{d\times d}$ the restriction $ \mathcal{H}_e=\mathcal{H}|_{(a(e),b(e))}$. Let each $\mathcal{H}_e$ be measurable and such that for almost every $\xi\in(a(e),b(e))$ the matrix $\mathcal{H}_e(\xi)$ is Hermitian. Assume that each $\mathcal{H}_e$ is bounded and strictly positive on bounded subsets of $(a(e),b(e))$.

\smallskip

We consider the Hilbert space direct sum
$$
X\coloneqq \Ls^2_{\mathcal{H}}(\Gamma)=\bigoplus_{e\in E}\Ls^2_{\mathcal{H}_e}(a(e),b(e))
$$
and we denote its elements by $x=(x_e)_{e\in E}$, $x_e\in\Ls^2_{\mathcal{H}_e}(a(e),b(e))$ for $e\in E$. For $P_1$, $P_0\in\mathbb{C}^{d\times d}$ with $P_1^{\star}=P_1$ and $P_0^{\star}=-P_0$ we consider the operator $A\colon D(A)\subseteq X\rightarrow X$ given by
\begin{equation}\label{eq:net}
\begin{array}{rcl}
Ax&=&\bigl(P_1(\mathcal{H}_ex_e)'+P_0\mathcal{H}_ex_e\bigl)_{e\in E} \; \text{ for } \; x=(x_e)_{e\in E}\\[9pt]
D(A)&=&\bigl\{x\in X\:;\: (P_1(\mathcal{H}_ex_e)'+P_0\mathcal{H}_ex_e)_{e\in E}\in X \text{ and } \mathcal{H}_ex_e|_{\partial(a(e),b(e))}=0\bigr\}.
\end{array}
\end{equation}
Note that in particular, for all $x\in D(A)$, we have $(\mathcal{H}_ex_e)'\in \Ls^2_{\mathcal{H}_e}(a(e),b(e))$ so that the second condition in \eqref{eq:net} is well-defined. 

\medskip

For the proof of the next theorem we need the following result by Picard \cite[Lemma 1.1]{Picard2012}.

\begin{lem}\label{lem:pi12} Let $X$ be Hilbert space, $H_i\colon D(H_i)\subseteq X\to X$ be densely defined and closed operators for $i\in\{1,2\}$. Let $(Q_n)_{n\in\mathbb{N}}$ be a family of orthogonal projections that converge strongly to $\id_X$. Assume that $Q_n H_1\subseteq H_1 Q_n$, $Q_n H_2 \subseteq H_2 Q_n$ and $Q_nH_2Q_n \in L(X)$ holds for all $n\in \mathbb{N}$. Then we have
\[
\overline{H_1^{\star}+H_2^{\star}} = (H_1 + H_2)^{\star}= \text{s\,--}\hspace{-3.5pt}\lim_{n\to\infty} (Q_n(H_1+H_2)Q_n)^{\star}=\text{s\,--}\hspace{-3.5pt}\lim_{n\to\infty} (Q_nH_1Q_n)^{\star}+(Q_nH_2Q_n)^{\star},
\]
\[
(Q_nH_1Q_n)^{\star} = Q_nH_1^{\star}Q_n\;\text{ and }\;(Q_nH_2Q_n)^{\star}=Q_nH_2^{\star}Q_n
\]
for every $n\in\mathbb{N}$.\hfill\qed
\end{lem}

\smallskip

\begin{thm}\label{SKEW-NET} Let $\Gamma$, $\mathcal{H}$, and $A$ be defined as in \eqref{eq:net}. Then $A$ is densely defined, closed and skew-symmetric. Its skew-adjoint is given by
\begin{equation}
\begin{array}{rcl}
-A^{\star}x&=&\bigl(P_1(\mathcal{H}_ex_e)'+P_0\mathcal{H}_ex_e\bigl)_{e\in E} \; \text{ for } \; x=(x_e)_{e\in E}\\[9pt]
D(-A^{\star})&=&\bigl\{x\in X\:;\: (P_1(\mathcal{H}_ex_e)'+P_0\mathcal{H}_ex_e)_{e\in E}\in X\bigr\}.\end{array}
\end{equation}
\end{thm}
\begin{proof} The assertion follows from Proposition \ref{SKEW-INT} and Lemma \ref{lem:pi12}. Indeed, since the set $E$ is countable, we may let $(E_n)_{n\in\mathbb{N}}$ be an increasing family of finite subsets of $E$ with the property that $E=\bigcup_{n\in\mathbb{N}}E_n$. Denote $Q_n \coloneqq \chi_{E_n}$, the characteristic function for the set $E_n\subseteq E$ and apply Lemma \ref{lem:pi12} to $H_1 x = (P_1(\mathcal{H}_ex_e)')_{e\in E}$ and $H_2 x = (P_0\mathcal{H}_ex_e)_{e\in E}$ with
\begin{align*}
   D(H_1) &= \bigl\{x\in X\:;\: ((\mathcal{H}_ex_e)')_{e\in E}\in X \text{ and } \mathcal{H}_ex_e|_{\partial(a(e),b(e))}=0 \text{ for all }e\in E\bigr\}\\
   D(H_2) & = \bigl\{x\in X\:;\: (P_0\mathcal{H}_ex_e)_{e\in E}\in X\bigr\}.
\end{align*}
Then it is easy to see that $A=\overline{H_1+H_2}$. In particular, we have $A^\star = (H_1+H_2)^\star$. The conditions in Lemma \ref{lem:pi12} are easily checked; the formula for the adjoint now follows from Proposition \ref{SKEW-INT}.
\end{proof}

Let $S\in\mathbb{C}^{d\times{}d}$ be a unitary matrix such that $P_1=S^{\star}\Delta{}S$ holds and $\Delta=\diag(\Lambda,-\Theta)$ where $\Lambda=\diag(\lambda_1,\dots,\lambda_{n_+})\in\mathbb{C}^{n_+\times n_+}$ and $\Theta=\diag(\theta_1,\dots,\theta_{n_-})\in\mathbb{C}^{n_-\times n_-}$ with $\lambda_k$, $\theta_j>0$ for $k\in\{1,\dots,n_+\}$, $j\in\{1,\dots,n_-\}$. We define
$$
\mathcal{G}_1=\bigoplus_{e\in E_{\ell}} \mathbb{C}^{n_+} \oplus \bigoplus_{e\in E_{r}} \mathbb{C}^{n_-} \;\text{ and }\; \mathcal{G}_2=\bigoplus_{e\in E_{\ell}} \mathbb{C}^{n_-} \oplus \bigoplus_{e\in E_{r}} \mathbb{C}^{n_+}
$$
in the sense of a Hilbert space direct sum where each $\mathbb{C}^{n_+}$ carries the scalar product $\langle{}\cdot,\cdot\rangle_{\Lambda}:=\langle{}\cdot,\Lambda{}\,\cdot\rangle{}$ and each $\mathbb{C}^{n_-}$ carries the scalar product $\langle{}\cdot,\cdot\rangle_{\Theta}:=\langle{}\cdot,\Theta{}\,\cdot\rangle{}$. Moreover, we read $\mathbb{C}^0=\{0\}$. In other words we have
\begin{equation}\label{EQ-GG}
\mathcal{G}_1=\ell^2(E_{\ell},\lambda)\oplus\ell^2(E_{r},\theta) \;\text{ and }\; \mathcal{G}_2=\ell^2(E_{\ell},\theta)\oplus\ell^2(E_{r},\lambda)
\end{equation}
with
$$
\ell^2(E_{\ell},\lambda):=\bigl\{(x_i)_{i\in N}:=(\underbrace{x_1,x_2,\dots,x_{n_+},x_1,x_2,\dots,x_{n_+},\dots}_{|E_{\ell}|-\text{times}}) \:;\:\|(x_i)_{i\in N}\|_{\ell^2(E_{\ell},\lambda)}^2=\sum_{i\in N}\lambda_i|x_i|^2<\infty\bigr\}
$$
where 
$$
\lambda=(\lambda_i)_{i\in N}=(\underbrace{\lambda_1,\lambda_2,\dots,\lambda_{n_+},\lambda_1,\lambda_2,\dots,\lambda_{n_+},\dots}_{|E_{\ell}|-\text{times}})
$$
for $N=\{1,2,\dots,n_+\cdot|E_{\ell}|\}$ if $n_+\geqslant1$ and $1\leqslant|E_{\ell}|<\infty$. If $n_+\geqslant1$ and $|E_{\ell}|=\infty$ then $N=\mathbb{N}$ and if $n_+=0$ or $|E_{\ell}|=0$, we put $\ell^2(E_{\ell},\lambda):=\{0\}$. The spaces $\ell^2(E_{r},\theta)$, $\ell^2(E_{\ell},\theta)$ and $\ell^2(E_{r},\lambda)$ are defined analogously.

\smallskip

We denote by $\pr_+$ and $\pr_-$ the projections as defined in Section \ref{SEC:4}. Let $\Omega$ be the standard symmetric form and let $\omega$ be the standard unitary form on $\mathcal{G}_1\oplus\mathcal{G}_2$. Now we define
$$
F\colon \gr(A^{\star})\rightarrow\mathcal{G}_1\oplus\mathcal{G}_2,\; F(x,A^{\star}x):=\biggl(\begin{bmatrix}(\pr_+(S\mathcal{H}_ex_e)(a(e))_{e\in E_{\ell}}\\(\pr_-(S\mathcal{H}_ex_e)(b(e)))_{e\in E_r}\\\end{bmatrix},
\begin{bmatrix}(\pr_-(S\mathcal{H}_ex_e)(a(e)))_{e\in E_{\ell}}\\(\pr_+(S\mathcal{H}_ex_e)(b(e)))_{e\in E_r}\\\end{bmatrix}\biggr).
$$
Under the additional assumption $\inf_{e\in E} \diam(a(e),b(e))>0$ we get a boundary system. For this, we need a quantitative version of the Sobolev embedding theorem, see e.g.~\cite[Lemma 2.3]{LSV2014}. We particularly refer to Arendt, Chill, Seifert, Vogt, Voigt \cite[Theorem 4.9]{Arendt2015} for the method of the proof.

\begin{lem}\label{lem:SE} Let $\mathcal{H}\in \Ls^\infty(0,\mu)^{d\times d}$, $\mathcal{H}({\xi})^\star=\mathcal{H}(\xi)\geqslant c>0$ for a.e. $\xi\in (0,\mu) $, $\mu>0$ and some $c>0$. Assume that $x\in \Ls^2(0,\mu)$ and $(\mathcal{H}x)'\in  \Ls^2(0,\mu)$. Then for all $0<\alpha\leqslant \mu$ we obtain
\[
    \|\mathcal{H}x(0)\|^2\leqslant 2\biggl(\frac1{{\mu}}\cdot \|\mathcal{H}^{1/2}\|_{\Ls^\infty(0,\mu)}^2 +  \mu\cdot\|\mathcal{H}^{-1/2}\|_{\Ls^\infty(0,\mu)}^2\biggr)\left(\|x\|^2_{\Ls^2_{\mathcal{H}}(0,\mu)}+\|(\mathcal{H}x)'\|^2_{\Ls^2_{\mathcal{H}}(0,\mu)}\right)
\]
\end{lem}
\begin{proof} We use that $\mathcal{H}x(\xi)=\mathcal{H}x(0)+\int_0^\xi (\mathcal{H}x)'(\xi)\dd \xi$ holds for all $\xi$ and compute
\begin{align*}
  \|\mathcal{H}x(0)\| & \leqslant    \inf_{\xi \in (0,\mu)} \|\mathcal{H}x(\xi)\| + \int_0^\mu \|(\mathcal{H}x)'(\xi)\| \dd \xi \\
  &   \leqslant \frac1\mu \int _0^\mu \|\mathcal{H}x(\xi)\| \dd \xi + \int_0^\mu \|(\mathcal{H}x)'(\xi)\| \dd \xi \\
  &    \leqslant \frac1{\sqrt{\mu}} \left(\int _0^\mu \|\mathcal{H}x(\xi)\|^2 \dd \xi\right)^{1/2} + \mu^{1/2} \left( \int_0^\mu \|(\mathcal{H}x)'(\xi)\|^2 \dd \xi \right)^{1/2}\\
  &   \leqslant   \frac1{\sqrt{\mu}} \|\mathcal{H}^{1/2}\|_{\Ls^\infty} \|x\|_{\Ls^2_\mathcal{H}} + \|\mathcal{H}^{-1/2}\|_{\Ls^\infty}\mu^{1/2} \|(\mathcal{H}x)'(\xi)\|_{\Ls^2_\mathcal{H}},
\end{align*}
which implies the desired inequality.
\end{proof}
\smallskip

\begin{lem}\label{BS-NETWORK-AXIS} If $\inf_{e\in E} \diam(a(e),b(e))>0$ and $\mathcal{H}$ is uniformly strictly positive and uniformly bounded, then the quadrupel $(\Omega, \mathcal{G}_1, \mathcal{G}_2, F, \omega)$ defined above is a boundary system for $A$.
\end{lem}

\begin{proof} We firstly observe that in \eqref{EQ-GG} all weighted $\ell^2$-spaces are isomorphic to unweighted $\ell^2$-spaces since the weight sequences only attain finitely many values. Using $\mu:=\inf_{e\in E} \diam(a(e),b(e))>0$ we can therefore derive from Lemma \ref{lem:SE} that the entries of $F(x,A^{\star}x)$ belong to $\ell^2$. Indeed, $x\in D(A^{\star})$ means $\mathcal{H}_ex_e|_{(a(e),a(e)+\mu)}\in\operatorname{H}^1(a(e),a(e)+\mu)$ for $e\in E_{\ell}$ and $\mathcal{H}_ex_e|_{(b(e)-\mu,b(e))}\in\operatorname{H}^1(b(e)-\mu,b(e))$ for $e\in E_{r}$. In view of the above, this implies that $F\colon \gr(A^{\star})\rightarrow\mathcal{G}_1\oplus\mathcal{G}_2$ is well-defined. Now we compute
\begin{align*}
\phantom{xxxxxx}&\hspace{-40pt}\Omega\bigl((x,A^{\star}x),(y,A^{\star}y)\bigr) \\
& =  \langle{}x,A^{\star}y\rangle_{\Ls^2_{\mathcal{H}}(\Gamma)} + \langle{}A^{\star}x,y\rangle_{\Ls^2_{\mathcal{H}}(\Gamma)}\\
& =  -\langle{}(x_e)_e,(P_1(\mathcal{H}_ey_e)'+P_0\mathcal{H}_ey_e)_e\rangle_{\Ls^2_{\mathcal{H}}(\Gamma)} -  \langle{}(P_1(\mathcal{H}_ex_e)'+P_0\mathcal{H}_ex_e)_e,(y_e)_e\rangle_{\Ls^2_{\mathcal{H}}(\Gamma)}\\
& =  -\langle{}(x_e)_e,P_1(\mathcal{H}_ey_e)'_e\rangle_{\Ls^2_{\mathcal{H}}(\Gamma)} - \langle{}P_1(\mathcal{H}_ex_e)')_e,(y_e)_e\rangle_{\Ls^2_{\mathcal{H}}(\Gamma)}\\
& = \sum_{e\in E}\bigg[-\int_{a(e)}^{b(e)}x_e(\xi)^{\star}\mathcal{H}_e(\xi)P_1(\mathcal{H}_ey_e)'(\xi)\operatorname{d}\!\xi - \int_{a(e)}^{b(e)}(P_1(\mathcal{H}_ex_e)'(\xi))^{\star}\mathcal{H}_e(\xi)y_e(\xi)\operatorname{d}\!\xi\bigg]
\\
& \stackrel{(\circ)}{=} \sum_{e\in E_{\ell}\backslash E_{r}}(\mathcal{H}_ex_e(a(e))^{\star}P_1(\mathcal{H}_ey_e(a(e)) + \sum_{e\in E_{r}\backslash E_{\ell}}-(\mathcal{H}_ex_e(b(e))^{\star}P_1(\mathcal{H}_ey_e(b(e))\\
& \phantom{++} + \sum_{e\in E_{\ell}\cap E_r}(\mathcal{H}_ex_e(a(e))^{\star}P_1(\mathcal{H}_ey_e(a(e))-(\mathcal{H}_ex_e(b(e))^{\star}P_1(\mathcal{H}_ey_e(b(e))\\
& = \sum_{e\in E_{\ell}}(S\mathcal{H}_ex_e(a(e))^{\star}\Delta(S\mathcal{H}_ey_e(a(e)) - \sum_{e\in E_{r}}(S\mathcal{H}_ex_e(b(e))^{\star}\Delta(S\mathcal{H}_ey_e(b(e))\\
& = \sum_{e\in E_{\ell}} \begin{bmatrix}\pr_+(S\mathcal{H}_ex_e)(a(e))\\\pr_-(S\mathcal{H}_ex_e)(a(e))\\\end{bmatrix}^{\star} \begin{bmatrix} \Lambda & 0  \\0 & -\Theta  \\\end{bmatrix} \begin{bmatrix}\pr_+(S\mathcal{H}_ey_e)(a(e))\\\pr_-(S\mathcal{H}_ey_e)(a(e))\\\end{bmatrix}\\
& \phantom{++} - \sum_{e\in E_{r}}\begin{bmatrix}\pr_+(S\mathcal{H}_ex_e)(b(e))\\\pr_-(S\mathcal{H}_ex_e)(b(e))\\\end{bmatrix}^{\star} \begin{bmatrix} \Lambda & 0  \\0 & -\Theta  \\\end{bmatrix} \begin{bmatrix}\pr_+(S\mathcal{H}_ey_e)(b(e))\\\pr_-(S\mathcal{H}_ey_e)(b(e))\\\end{bmatrix}\\
& = \sum_{e\in E_{\ell}}\langle{}\pr_+(S\mathcal{H}_ex_e)(a(e)),\pr_+(S\mathcal{H}_ey_e)(a(e))\rangle{}_{\Lambda}\\
& \phantom{++} + \sum_{e\in E_{\ell}}-\langle{}\pr_-(S\mathcal{H}_ex_e)(a(e)),\pr_-(S\mathcal{H}_ey_e)(a(e))\rangle{}_{\Theta}\\
& \phantom{++} -\sum_{e\in E_{r}}\langle{}\pr_+(S\mathcal{H}_ex_e)(b(e)),\pr_+(S\mathcal{H}_ey_e)(b(e))\rangle{}_{\Lambda}\\
& \phantom{++} -\sum_{e\in E_{r}}-\langle{}\pr_-(S\mathcal{H}_ex_e)(b(e)),\pr_-(S\mathcal{H}_ey_e)(b(e))\rangle{}_{\Theta}\\
& =  \bigg\langle{}\begin{bmatrix}(\pr_+(S\mathcal{H}_ex_e)(a(e)))_{e\in E_{\ell}}\\(\pr_-(S\mathcal{H}_ex_e)(b(e)))_{e\in E_{r}}\\\end{bmatrix}, \begin{bmatrix}(\pr_+(S\mathcal{H}_ey_e)(a(e)))_{e\in E_{\ell}}\\(\pr_-(S\mathcal{H}_ey_e)(b(e)))_{e\in E_{r}}\\\end{bmatrix}\bigg\rangle{}_{\mathcal{G}_1}\\
&\phantom{++}-\bigg\langle{}\begin{bmatrix}(\pr_-(S\mathcal{H}_ex_e)(a(e)))_{e\in E_{\ell}}\\(\pr_+(S\mathcal{H}_ex_e)(b(e)))_{e\in E_{r}}\\\end{bmatrix}, \begin{bmatrix}(\pr_-(S\mathcal{H}_ey_e)(a(e)))_{e\in E_{\ell}}\\(\pr_+(S\mathcal{H}_ey_e)(b(e)))_{e\in E_{r}}\\\end{bmatrix}\bigg\rangle{}_{\mathcal{G}_2}\\
&=\omega\bigl(F(x,A^{\star}x),F(y,A^{\star}y)\bigr)
\end{align*}
where we used integration by parts in $(\circ)$. For the right semi-axis' associated with $e\in E_{\ell}\backslash E_{r}$ we can proceed as in the proof of Lemma \ref{BS-SEMI-AXIS}. The left semi-axis' associated with $e\in E_{r}\backslash E_{\ell}$ we treat analogously and  the finite intervals for $e\in E_{\ell}\cap E_{r}$ are straightforward. It is easy to see that $F$ is surjective by constructing piecewise linear functions.
\end{proof}

As in the semi-axis case treated in Section \ref{SEC:4} we get the classification of generators of contraction semigroups via Theorem \ref{CLASS-THM}.

\smallskip

\begin{thm}\label{super-thm} Let $\Gamma$, $\mathcal{H}$, $P_1$, $P_0$, $S$, $n_+$ and $n_-$, $\mathcal{G}_1$ and $\mathcal{G}_2$ be as above and assume $\mathcal{H}$ to be uniformly bounded and uniformly strictly positive. Assume that we have $\inf_{e\in E} \diam(a(e),b(e))>0$. The operator $A\colon D(A)\subseteq \Ls^2_{\mathcal{H}}(\Gamma)\rightarrow\Ls^2_{\mathcal{H}}(\Gamma)$, $Ax=(P_1(\mathcal{H}_ex_e)'+P_0\mathcal{H}_ex_e)_{e\in E}$ for $x=(x_e)_{e\in E}$, generates a $\Cnull$-semigroup of contractions if and only if there is a contraction $T\colon\mathcal{G}_2\rightarrow\mathcal{G}_1$ such that
$$
D(A)=\biggl\{x\in \Ls^2_{\mathcal{H}}(\Gamma)\:;\: (\mathcal{H}x)'\in\Ls^2_{\mathcal{H}}(\Gamma)\text{ and } T\begin{bmatrix}(\pr_-(S\mathcal{H}_ex_e)(a(e)))_{e\in E_{\ell}}\\(\pr_+(S\mathcal{H}_ex_e)(b(e)))_{e\in E_{r}}\\\end{bmatrix}=\begin{bmatrix}(\pr_+(S\mathcal{H}_ex_e)(a(e)))_{e\in E_{\ell}}\\(\pr_-(S\mathcal{H}_ex_e)(b(e)))_{e\in E_{r}}\\\end{bmatrix}\biggr\}
$$
holds.\hfill\qed
\end{thm}

We point out that the above characterization looks technical, but that in addition to the given data only the matrix $S$ and the numbers $\lambda_k$, $\theta_j$ have to be computed by diagonalizing the matrix $P_1$. In view of \eqref{EQ-GG} testing the contraction property then boils down to estimations in weighted sequence spaces.

\medskip

%%%%%%%%%%%%%%%%%%%%%%%%%%%%%%%%%%%%%%%%%%%%%%%%%%%%%%%
%%%%%%%%%%%%%%%%%%%%%%%%%%%%%%%%%%%%%%%%%%%%%%%%%%%%%%%
%%                                                   %%
%% 7. Port-Hamiltonian operators in networks II -- Arbitrary edge lengths                        %%
%%                                                   %%
%%%%%%%%%%%%%%%%%%%%%%%%%%%%%%%%%%%%%%%%%%%%%%%%%%%%%%%
%%%%%%%%%%%%%%%%%%%%%%%%%%%%%%%%%%%%%%%%%%%%%%%%%%%%%%%

\section{Port-Hamiltonian operators on networks II}\label{sec:phsII}

The results of Section \ref{SEC:positive-edge} allow to treat the port-Hamiltonian operator on a wide range of networks. Comparing to the previous results we were able to relax several boundedness conditions. In view of the network we here catch up with results of Lenz, Schubert, Veseli\'c \cite{LSV2014} who treat the Laplacian on almost arbitrary networks: The only restriction needed in \cite{LSV2014} is the same that we need in Theorem \ref{super-thm}, namely that the edge lenghts are bounded away from zero.

\smallskip

In this section we discuss the case of arbitrary edge lengths. Moreover, we will revisit the case of unbounded $\mathcal{H}$. In the first of the following three subsections we will have a look at a set of assumptions, that will enable us to apply the theory developed earlier directly. The second subsection is concerned with a classification theorem usable for arbitrary edge lengths and rather general $\mathcal{H}$ as well as $P_0\neq 0$. We conclude the present section with a model case for $P_0=0$ and $\mathcal{H}=1$ but arbitrary edge lengths. It will turn out that we can define a model network with uniform edge lengths for a given network with arbitrary edge lengths. This strategy might be viewed as a `regularization' method similar to the ideas developed by Gernandt, Trunk \cite{Gernandt2018} and the references therein. Here, however, we extend the theory to systems of first derivative operators and detour the theory of boundary triples.

\smallskip

%%%%%%%%%%%%%%%%%%%%%%%%%%%%%%%%%%%%%%%%%%%%%%%%%%%%%%%
%%%%%%%%%%%%%%%%%%%%%%%%%%%%%%%%%%%%%%%%%%%%%%%%%%%%%%%
%%                                                   %%
%% 7.1 Arbitrary Edge lenghts                        %%
%%                                                   %%
%%%%%%%%%%%%%%%%%%%%%%%%%%%%%%%%%%%%%%%%%%%%%%%%%%%%%%%
%%%%%%%%%%%%%%%%%%%%%%%%%%%%%%%%%%%%%%%%%%%%%%%%%%%%%%%

\subsection{A divide-and-conquer approach for arbitrary edge lenghts}\label{SEC:case-by-case}

For the first treatment of networks with edges being arbitrarily small, we begin with a collection of some abstract result. As before let $(X_n)_{n\in\mathbb{N}}$ be a sequence of closed, increasing subspaces of $X$ with $(P_n)_{n\in\mathbb{N}}$ being the corresponding orthogonal projections. Furthermore, assume that $\bigcup_{n\in\mathbb{N}} X_n$ is dense in $X$ and put $Q_n:=1-P_n$.

\begin{ass}\label{ass:2} Let $H\colon D(H)\subseteq X\to X$ be densely defined and linear. Below we consider the following three assumptions.\vspace{3pt}
  \begin{compactitem}
  \item[(A1)] $H^\star P_n$ is densely defined for every $n\in \mathbb{N}$. \vspace{2pt}
  \item[(A2)] For every $n\in \mathbb{N}$, $x\in D(H)$ and $\varepsilon>0$ there is $z\in Q_n[D(H)]$ such that
$$
\|z\| \leqslant \varepsilon,\;P_n x+ z \in D(H) \text{ and } P_n(H(P_n x+z)-Hx) =0
$$
\item[(A3)] For every $n\in\mathbb{N}$, $x\in D(H)$, $\varepsilon>0$ there is $w \in Q_n[D(H)]$ such that
$$
P_nx+w\in D(H) \text{ and } \|Q_n((1-H)(P_nx+w))\|\leqslant \varepsilon
$$
\end{compactitem}
\end{ass}

\begin{lem}\label{lem:wdd} Assume (A1) and (A2) to be in effect. Then 
\[
H_n \colon P_n[D(H)]\subseteq X_n \to X_n
\]
given by $H_n P_n x \coloneqq P_n H x$ for all $x\in D(H)$ is well-defined. Moreover, $H_n$ is densely defined.
\end{lem}
\begin{proof} Let $x\in D(H)$ with $P_nx=0$. By (A2) we find $z_k\in Q_n[D(H)]$ such that $z_k \in D(H)$, $\|z_k\|\leqslant 1/k$ and 
  \[
      P_n H  x=  P_n H  z_k.
  \]
  Since $H^\star P_n$ is densely defined, it follows that $P_nH$ is closable. Thus, as $z_k \to 0$ and $(P_nHz_k)_{k\in\mathbb{N}}$ is convergent to $P_nHx$, it follows that $P_nH x=0$, which yields the first assertion of the lemma. Since $P_n$ is continuous and $H$ is densely defined it follows
  \[
      X_n =  P_n [X]=P_n[\overline{D(H)}]\subseteq \overline{P_n[D(H)]}\subseteq X_n,
  \]
  which implies that $H_n$ is densely defined.
\end{proof}

\begin{thm}\label{thm:lfw2} Assume that (A1)--(A3) hold. \vspace{3pt}
\begin{compactitem}
 \item[(i)] If the operator $H_n$ from Lemma \ref{lem:wdd} is dissipative for every $n\in \mathbb{N}$, then $H$ is dissipative.\vspace{2pt}
 \item[(ii)] We have $\overline{\ran(1-H)}\supseteq \ran(1-H_n)$ for every $n\in \mathbb{N}$.\vspace{2pt}
 \item[(iii)] If $H_n$ is dissipative and $\ran(1-H_n)\subseteq X_n$ is dense for every $n\in\mathbb{N}$, then $\overline{H}$ is maximal dissipative.
\end{compactitem}
\end{thm}
\begin{proof} (i) Let $x\in D(H)$ and observe that
\[
0\geqslant \Re \langle P_n x,H_n x\rangle =\Re \langle P_n x,P_nH x\rangle=\Re \langle P_n x,H x\rangle\mathop{\longrightarrow}^{n\to\infty} \Re \langle x,H x\rangle
\]
holds. This yields that $H$ is dissipative.

\smallskip

(ii) Let $y\in \ran(1-H_n)$. We find $x\in D(H)$ such that
\[
   y = P_n x - P_n Hx = P_n(1-H)x
\] 
holds. For every $k\in \mathbb{N}$ there is $w_k\in Q_n[D(H)]$ with $P_n x+ w_k \in D(H)$ and
\[
\|Q_n((1-H)(P_nx+w_k))\|\leqslant 1/k
\]
by (A3). Next, we show that for all $k\in\mathbb{N}$ we have
\[
P_n(1-H)x = P_n (1-H)(P_nx+w_k).
\]
For this, note that $P_n x = P_n(P_nx+w_k)$ and therefore $P_n H x = P_n H(P_n x+w_k)$ by Lemma~\ref{lem:wdd}. Thus,
  \begin{align*}
    P_n (1-H)(P_nx + w_k) & =     P_n (P_nx + w_k) - P_n H(P_nx + w_k) \\
      & = P_nx -P_n Hx\\
      & = P_n(1-H)x.
  \end{align*}
  So,
  \begin{align*}
    (1-H)(P_n x+w_k) & = P_n (1-H)(P_nx+w_k) + Q_n (1-H)(P_nx+w_k) \\
    & =  P_n(1-H)x+ Q_n (1-H)(P_nx+w_k)\phantom{\mathop{\longrightarrow}^{k\to\infty}}\\
    & = y + Q_n (1-H)(P_nx+w_k) \mathop{\longrightarrow}^{k\to\infty} y.
  \end{align*}

\smallskip

(iii) We observe that $H$ is dissipative by (i). Since $H$ is also densely defined, we infer that $H$ is closable. Moreover, it follows by (ii) that $\ran(1-H)$ is dense in $X$. Hence, $\overline{H}$ is maximal dissipative by Lemma \ref{lem:ele}.  
\end{proof}

We may now apply Theorem \ref{thm:lfw2} to port-Hamiltonian operators as discussed in the previous section. The main theorem in this section provides a sufficient condition for extensions being maximal dissipative. It can be viewed as a `localization' of Theorem \ref{super-thm}. Consequently, we have to assume all the conditions stated in Theorem \ref{super-thm}. We will however dispose of the condition that $\mathcal{H}$ is being both uniformly strictly positive and uniformly bounded as well as the uniform positive lower bound for the edge lengths.

\smallskip

\begin{thm}\label{super-thm2} Let $\Gamma$, $\mathcal{H}$, $P_1$, $P_0$, $S$, $n_+$ and $n_-$ be as above. The operator $A\colon D(A)\subseteq \Ls^2_{\mathcal{H}}(\Gamma)\rightarrow\Ls^2_{\mathcal{H}}(\Gamma)$, $Ax=(P_1(\mathcal{H}_ex_e)'+P_0\mathcal{H}_ex_e)_{e\in E}$ for $x=(x_e)_{e\in E}$, generates a $\Cnull$-semigroup of contractions if the following conditions hold.

\begin{enumerate}

\item[(i)] There exists a sequence $(E_n)_{n\in\mathbb{N}}$ with $E_n \subseteq E$, $E_n\subseteq E_{n+1}$ for all $n\in\mathbb{N}$ and $E=\bigcup_{n\in\mathbb{N}} E_n$ such that $\mathcal{H}_n \coloneqq \mathcal{H}|_{\mathcal{E}_n}$ is bounded and uniformly strictly positive, where $\mathcal{E}_n \coloneqq \bigcupdot_{e\in E_n} (a(e),b(e))$.\vspace{3pt}

\item[(ii)] We have $\inf_{e\in E_n} (b(e)-a(e))>0$ for all $n\in\mathbb{N}$. \vspace{3pt}

\item[(iii)] For every $n\in \mathbb{N}$ there exists $F_n \subseteq E\setminus E_n$ finite such that given $x\in D(A)$ we find $\tilde{y}\in  \bigoplus_{e\in E_n\cup F_n} \Ls^2_{\mathcal{H}_e}(a(e),b(e))$ such that $\tilde{y}_e=x_e$ for all $e\in E_n$ and  the vector
\[
y\coloneqq \begin{cases} \;\tilde{y}_e& \text{ if } e\in E_n \cup F_n, \\ \;\hspace{1.5pt}0 & \text{ if } e \in E\setminus (E_n\cup F_n) \end{cases}
\]
belongs to $D(A)$.\vspace{3pt}

\item[(iv)] For every $n\in\mathbb{N}$ the operator $A_n \colon D(A_n)\subseteq \bigoplus_{e\in E_n} \Ls^2_{\mathcal{H}_e}(a(e),b(e))\to  \bigoplus_{e\in E_n} \Ls^2_{\mathcal{H}_e}(a(e),b(e))$ given by $A_n x^{(n)} \coloneqq (P_1(\mathcal{H}_ex_e)'+P_0\mathcal{H}_ex_e)_{e\in E_n}$ for $x\in D(A)$ such that $x^{(n)}=x$ on $E_n$ satisfies
 \begin{align*}
    D(A_n)& = \biggl\{ x \in  \bigoplus_{e\in E_n} \Ls^2_{\mathcal{H}_e}(a(e),b(e))\:; \:\exists\: \tilde{x}\in D(A)\; \forall\: e\in E_n \colon x_e=\tilde{x}_e \biggr\} \\
& = \biggl\{ x \in  \bigoplus_{e\in E_n} \Ls^2_{\mathcal{H}_e}(a(e),b(e))\:;\: (\mathcal{H}x)'\in\ \bigoplus_{e\in E_n} \Ls^2_{\mathcal{H}_e}(a(e),b(e)) \\
& \quad\quad\quad\quad\text{ and } \: T_n\begin{bmatrix}(\pr_-(S\mathcal{H}_ex_e)(a(e)))_{e\in E_{n,\ell}}\\(\pr_+(S\mathcal{H}_ex_e)(b(e)))_{e\in E_{n,r}}\\\end{bmatrix}=\begin{bmatrix}(\pr_+(S\mathcal{H}_ex_e)(a(e)))_{e\in E_{n,\ell}}\\(\pr_-(S\mathcal{H}_ex_e)(b(e)))_{e\in E_{n,r}}\\\end{bmatrix} \biggr\}
\end{align*}
for some contraction $T_n \colon \mathcal{G}_{n,2}\rightarrow\mathcal{G}_{n,1}$ where $E_{n,\ell}= E_n \cap E_{\ell}$, $E_{n,r}= E_n \cap E_{r}$. Here, $\mathcal{G}_{n,1}$ and $\mathcal{G}_{n,2}$ are defined as $\mathcal{G}_2,\mathcal{G}_1$ in \eqref{EQ-GG} with $E_{n,\ell}$, $E_{n,r}$ replacing $E_{\ell}$, $E_r$, respectively.

\vspace{3pt}

 \item[(v)] For every $n\in \mathbb{N}$, $x\in D(A)$ and $\varepsilon>0$ we find ${y}\in D(A)$ such that ${y}_e=x_e$ for all $e\in E_n$ and  \[\|(1-A)y\|_{\oplus_{e\in E\setminus E_n}L^2_{\mathcal{H}_e}(a(e),b(e))}\leqslant \varepsilon.\]
\end{enumerate}
\end{thm}

\begin{proof} We  apply Theorem \ref{thm:lfw2} to $H=A$. The sequence of projections $(P_n)_{n\in \mathbb{N}}$ are the restriction operators $\tilde{P}_n x \coloneqq (x_e)_{e\in E_n} \in \bigoplus_{e\in E_n} \Ls^2_{\mathcal{H}_e}(a(e),b(e))\eqqcolon X_n$ for $x\in\Ls^2_\mathcal{H}(\Gamma)$. Since $D(A^\star)\supseteq \bigoplus_{e\in E} \operatorname{C}_c^\infty(a(e),b(e))$, we infer that $A^{\star}\tilde{P}_n$ is densely defined and hence (A1) holds in Assumption \ref{ass:2}. Next, we show (A2). For this let $n\in\mathbb{N}$, $x\in D(A)$ and $\varepsilon>0$. We choose $F_n\subseteq (E\setminus E_n)$ finite and $\tilde{y}$ according to the assumptions in this theorem. Since $F_n$ is finite and $\bigoplus_{e\in E}\operatorname{C}_c^\infty(a(e),b(e))$ belongs to the domain in \eqref{eq:net}, we find
\[
\tilde{z}\in\bigoplus_{e\in F_n} \operatorname{C}_c^\infty(a(e),b(e))
\]
such that $\|\tilde{y}_e-\tilde{z}_e\|_{\Ls^2_{\mathcal{H}_e}}\leqslant \varepsilon/(|F_n|+1)$. Then
\[
z_e\coloneqq \begin{cases} \;\tilde{y}_e - \tilde{z}_e&\text{ if } e\in F_n ,\\ \;\hspace{12pt}0 & \text{ if } e\in  E\setminus F_n\end{cases}
\]
has the desired properties. In particular, we obtain by Lemma \ref{lem:wdd} that $A_n$ is well-defined. Moreover, we can apply Theorem \ref{super-thm} to $A_n$ and deduce that $A_n$ is maximal dissipative. In order to conclude that $A$ is maximal dissipative by Theorem \ref{thm:lfw2}, it thus remains to show that condition (A3) in Assumption \ref{ass:2} is satisfied. This however follows easily from the last condition that we assumed above.
\end{proof}

\begin{rem} We shall see in Section \ref{SEC:Examples} that all the above conditions are easy to verify in practice except for the last one. The last condition really depends on the ctopology of the graph and the corresponding boundary conditions.  A more systematic approach will be presented in the next subsection.
\end{rem}

\medskip

\subsection{The canonical boundary system}\label{SEC:arbitrary-edge}

The theorem above treats networks without uniformly positive edge lenghts by applying the methods of Section \ref{SEC:positive-edge}, i.e., for uniformly positive edge lenghts, locally. This requires that the graph can be divided into pieces where the results of Section \ref{SEC:positive-edge} are applicable---and that the results for each piece can later be put together. To cover this, we needed the assumptions stated in Theorem \ref{super-thm2}.

\smallskip

In this subsection we want to outline a completely different approach to networks with arbitrary edge lenghts. Firstly, we recall the reason why Theorem \ref{super-thm} cannot be applied without assuming that the edge lenghts are uniformly positive. Indeed, without that it is impossible even to write down the boundary system used in the latter theorem, since then the map $F$ then would not be well-defined. On the other hand the article \cite{BTBS} provides a canonical boundary system \textit{for every skew-symmetric operator $A$}. Since we are able to compute $A^{\star}$ without the aforementioned restriction on the graph, \cite[Theorem 3.1]{BTBS} provides a boundary system and \cite[Theorem A]{BTBS} then describes all contraction semigroups in a similar way as we did it above. For convenience of the reader we recall {\cite[Theorem 3.1]{BTBS}}.  For the direct decomposition used in this result, we refer to  \cite[Lemma 2.5]{BT}.
 
\smallskip

\begin{thm}\label{thm:canBS} Let $H_0$ be skew-symmetric. Let $\mathcal{G}_1:= \ker(1-H_0^\star)$, $\mathcal{G}_2:= \ker(1+H_0^\star)$ and let $P_j\colon D(H_0^\star)\to \mathcal{G}_j$ be the projection for $j\in\{1,2\}$ according to the direct decomposition $D(H_0^\star)=D(H_0)\dot+\mathcal{G}_1\dot+\mathcal{G}_2$. Let $\Omega$ be the standard symmetric form, $\omega$ be the standard unitary form and let $F\colon \gr(H_0^\star)\to \mathcal{G}_1\oplus \mathcal{G}_2$ be defined by
$$
F(x,H_0^{\star})=({\sqrt{2}}P_1 x,{\sqrt{2}}P_2x).
$$
Then $(\Omega, \mathcal{G}_1,\mathcal{G}_2,F,\omega)$ is a boundary systems for $H_0$.\hfill\qed 
\end{thm}

\smallskip

Applying the above to port-Hamiltonian systems we get the following.
\smallskip

\begin{thm}\label{super-thm-2} Let $\Gamma$, $\mathcal{H}$, $P_1$, $P_0$ be as at the beginning of Section \ref{SEC:positive-edge}. Let  $\mathcal{G}_1\coloneqq \ker(1-A^{\star})$ and $\mathcal{G}_2\coloneqq \ker(1+A^{\star})$. The operator $H\colon D(H)\subseteq \Ls^2_{\mathcal{H}}(\Gamma)\rightarrow\Ls^2_{\mathcal{H}}(\Gamma)$, $Hx=(P_1(\mathcal{H}_ex_e)'+P_0\mathcal{H}_ex_e)_{e\in E}$ for $x=(x_e)_{e\in E}$, generates a $\Cnull$-semigroup of contractions if and only if there is a contraction $T\colon\mathcal{G}_2\rightarrow\mathcal{G}_1$ such that
$$
D(H)=\bigl\{x\in D(A^\star)\:;\: (\mathcal{H}x)'\in\Ls^2_{\mathcal{H}}(\Gamma)\text{ and } TF_2x=F_1x\bigr\}
$$
holds. Here, $F=(F_1,F_2)$ is given as in Theorem \ref{thm:canBS} with $A$ in place of $H_0$. The adjoint $A^{\star}$ is explicitly given in Theorem \ref{SKEW-NET}.\hfill\qed
\end{thm}

\smallskip

\subsection{A model network approach}\label{SEC:model-network}

The description of all maximal dissipative extensions in the previous subsection might not be explicit enough for applications. For this reason, we shall study a particular case in greater detail here. We focus on the case $P_0=0$ and $P_1 = P_1^\star\in \mathbb{C}^{d\times d}$ invertible and $\mathcal{H}=1$. We will thus concentrate on networks with infinitely many edges without strictly positive lower bound.  For the final characterization of maximal dissipative extensions, we need some prerequisites. For $\alpha>0$ we use the abbreviations
\[
    \Ls^2_{1/\alpha}(0,1)=\bigl(\Ls^2(0,1), \| \sqrt{1/\alpha} \cdot \|_{\Ls^2(0,1)}\bigr),\quad \operatorname{H}^1_{1/\alpha}(0,1) = \{ f\in \Ls^2_{1/\alpha}(0,1); f'\in \Ls^2_{1/\alpha}(0,1)\}.
\]

\begin{prop}\label{prop:geunge} \begin{compactitem}\item[(i)] Let $\alpha,\lambda>0$ and $[\partial \lambda]_\alpha \colon \operatorname{H}^1_{1/\alpha}(0,1)\subseteq \Ls^2_{1/\alpha}(0,1)\to \Ls^2_{1/\alpha}(0,1)$  be defined by
\[
    [\partial \lambda]_\alpha f = (\lambda f)'.
\]
Then the mapping
\begin{align*}
   S_{\pm,\lambda/\alpha} \colon \ker(1\pm [\partial \textstyle{\frac\lambda\alpha}]_\alpha) & \to  \ker(1\pm [\partial \lambda]_1), \\
\phi & \mapsto \textstyle{ \frac{c^\pm_{\lambda/\alpha}}{\alpha^{1/2}c^\pm_{\lambda}} }\phi(\textstyle{\frac1\alpha \cdot})
\end{align*}
is unitary, where $c^{\pm}_{\mu}\coloneqq \left( \pm\frac\mu2\left(e^{\mp 2/\mu}-1\right)\right)^{1/2}$, $\mu>0$.\vspace{3pt}

\item[(ii)] Let $\alpha>0$, $P_1 = S^\star\Delta S$ as above. Define $[\partial P_1]_\alpha \colon \operatorname{H}^1_{1/\alpha}(0,1)\subseteq \Ls^2_{1/\alpha}(0,1)\to \Ls^2_{1/\alpha}(0,1)$  be defined by
\[
[\partial P_1]_\alpha f = (P_1 f)'.
\]
Then the mapping
\begin{align*}
   S_{\pm,P_1/\alpha} \colon \ker(1\pm [\partial P_1\textstyle{\frac1\alpha}]_\alpha) & \to  \ker(1\pm [\partial P_1]_1), \\
\phi & \mapsto S^\star\diag\left(S_{\pm,\lambda_1/\alpha},\ldots,S_{\pm,\lambda_{n_+}/\alpha}, S_{\pm,-\theta_1/\alpha},\ldots S_{\pm,-\theta_{n_-}/\alpha}\right)Sf
\end{align*}
is unitary.
\end{compactitem}
\end{prop}
\begin{proof} (i) The mapping is easily seen to be well-defined and it is elementary to see that it is unitary. 

\smallskip

(ii) The map under consideration is a composition of unitary operators by (i) and hence unitary itself.
\end{proof}

\begin{prop}\label{prop:smallunit} Let $-\infty<a<b<\infty $ and $\alpha\coloneqq b-a$.\vspace{2pt}
\begin{compactitem}
\item[(i)] The mapping
\begin{align*}
   S_{a,b} \colon L^2_{1/\alpha}(0,1) & \to L^2(a,b), \\
      g& \mapsto \frac1{\alpha}g(\frac{\cdot-a}{\alpha})
\end{align*}
is unitary.

\vspace{3pt}

\item[(ii)] We have
\[
    \partial_{0,(a,b)}P_1 = S_{a,b}^\star [\partial_0P_1\frac1\alpha]_\alpha S_{a,b},
\]
where $\partial_{0,(a,b)}$ has the vector-valued $\operatorname{H}^1$-functions that are zero in $a$ and $b$ as its domain. Similarly we define $\partial_0$ on the right-hand side.
\end{compactitem}
\end{prop}
\begin{proof} (i) We compute for $g\in L^2_{1/\alpha}(0,1)$
\begin{align*}
    \|S_{a,b}g\|_{\Ls^2(a,b)}^2& = \int_a^b \frac1{\alpha^2} \|g(\frac{s-a}{\alpha}) \|^2 \dd s \\ &= \int_0^1 \frac1{\alpha} \|g(t) \|^2 \dd t \\ & = \|g\|^2_{\Ls^2_{1/\alpha}(0,1)}.
\end{align*}
(ii) The assertion is easy and follows from the chain rule.
\end{proof}

\begin{rem}
 In Proposition \ref{prop:smallunit}(ii) a similar equation also holds if we dispose of the Dirichlet boundary conditions on either side. In fact, by unitary equivalence the mentioned equation follows from computing the adjoint on either side of the equation in Proposition \ref{prop:smallunit}(ii).
\end{rem}

\smallskip

Equipped with these preliminaries, we are now in the position to apply the techniques to the characterization of all maximal dissipative extensions of first order systems on networks. We shall obtain a one-to-one correspondence: Theorem \ref{thm:arb_edg} below relates all maximal dissipative extensions of first derivative operators in networks without strictly positive lower bound for the edge lengths and all maximal dissipative extensions for first derivative operators on networks with the same cardinality but all edges having unit length. 

\smallskip

\begin{thm}\label{thm:arb_edg} Let $E$ be a countable set, $a\colon E\to \mathbb{R}$ and $b\colon E\to \mathbb{R}$ with $0< \alpha_e \coloneqq   b(e) - a(e)$ for all $e\in E$. Let 
$$
A \colon \bigoplus_{e\in E} \operatorname{H}_0^1(a(e),b(e))  \subseteq \bigoplus_{e\in E} \operatorname{L}^2(a(e),b(e)) \to \bigoplus_{e\in E} \operatorname{L}^2(a(e),b(e)), \;\; (\phi_e)_{e\in E} \mapsto ((P_1\phi_e)')_{e\in E}.
$$
Let $\tilde{A}$ be the operator $A$ for the case $a(e)=0=1-b(e)$ for all $e\in E$. Then a linear operator $H$ with  $A \subseteq H \subseteq -A^\star$ is maximal dissipative if and only if there is a contraction $T\colon \tilde{\mathcal{G}}_2 \to \tilde{\mathcal{G}}_1$ such that
\[
   D(H) = \bigl\{ x\in D(A^\star); T \mathbf{S}_{-,P_1/\alpha} {F}_2(\mathbf{S}_{a,b}^{-1}x) = \mathbf{S}_{+,P_1/\alpha} {F}_1(\mathbf{S}_{a,b}^{-1}x) \bigr\}
\] 
where $(\Omega,\mathcal{G}_1,\mathcal{G}_2,F,\omega)$ is the canonical boundary system given in Theorem \ref{super-thm-2} according to the choice $a(e)=1-b(e)=0$, $\mathcal{H}_e = 1/\alpha_e$,
\begin{align*}
 \tilde{ \mathcal{G}}_2 & = \ker(1 + \tilde{A}^\star) =  \bigoplus_{e\in E} \ker(1-[\partial P_1]) \subseteq \bigoplus_{e\in E} \operatorname{L}^2(0,1) \\ \tilde{\mathcal{G}}_1 & = \ker(1-\tilde{A}^\star) =\bigoplus_{e\in E} \ker(1+[\partial P_1]) \subseteq \bigoplus_{e\in E} \operatorname{L}^2(0,1)
\end{align*}
and
$$
\mathbf{S}_{a,b} \colon \bigoplus_{e\in E} \Ls^2_{1/\alpha_e}(0,1) \to  \bigoplus_{e\in E}\Ls^2(a(e),b(e)),\; (\mathbf{S}_{a,b}x)_e = S_{a_e,b_e}x_e,
$$
$$
\mathbf{S}_{\pm,P_1/\alpha_e} \colon \bigoplus_{e\in E} \ker(1\pm [\partial P_1/\alpha_e]_{\alpha_e}) \to  \bigoplus_{e\in E}  \ker(1\pm [\partial P_1]),\; x\mapsto   (S_{\pm,P_1/\alpha_e} x_e)_{e\in E}.
$$
\end{thm}

\begin{proof} The claim follows from the considerations above relying on Proposition \ref{prop:smallunit} and Proposition \ref{prop:geunge} together with the classification result Theorem \ref{super-thm-2}. Note that for the computation of the adjoint $\tilde{A}^\star$ we used Theorem \ref{SKEW-NET}. The explicit computation of the kernels is elementary.
\end{proof}

\smallskip

The most important consequence of Theorem \ref{thm:arb_edg} is that now all the maximal dissipative extensions of $A$ can be characterized by the maximal dissipative extensions of $\tilde{A}$.

\smallskip

\begin{cor} \label{cor:arb_edg} Let $A$ and $\tilde{A}$ as in Theorem \ref{thm:arb_edg}. There is a one-to-one correspondence $\Phi$ of all maximal dissipative extensions $H$ of $A$ and $\tilde{H}$ of $\tilde{A}$. In particular, for every maximal dissipative extension $H$ of $A$, we find a unique contraction $\tilde{T}\colon \mathcal{G}_2 \to \mathcal{G}_1$, with $\mathcal{G}_i$ being given in Section \ref{SEC:positive-edge} for $a(e)=1-b(e)=0$ such that
\[
D(\Phi(H)) = \biggl\{ x\in D(\tilde{A}^\star)\:;\: \tilde{T} \begin{bmatrix}(\pr_-(Sx_e)(0))_{e\in E_{\ell}}\\(\pr_+(Sx_e)(1))_{e\in E_{r}}\\\end{bmatrix}=\begin{bmatrix}(\pr_+(Sx_e)(0))_{e\in E_{\ell}}\\(\pr_-(Sx_e)(1))_{e\in E_{r}}\\\end{bmatrix}\biggr\}.
\]
\end{cor}

\smallskip

\begin{rem} Note that the results Theorem \ref{thm:arb_edg} and Corollary \ref{cor:arb_edg} naturally extend to the cases that consist of graphs with infinitely long edges.   
\end{rem}

%%%%%%%%%%%%%%%%%%%%%%%%%%%%%%%%%%%%%%%%%%%%%%%%%%%
%%%%%%%%%%%%%%%%%%%%%%%%%%%%%%%%%%%%%%%%%%%%%%%%%%%
%%                                               %%
%% 8 Examples                                    %%
%%                                               %%
%%%%%%%%%%%%%%%%%%%%%%%%%%%%%%%%%%%%%%%%%%%%%%%%%%%
%%%%%%%%%%%%%%%%%%%%%%%%%%%%%%%%%%%%%%%%%%%%%%%%%%%

\section{Examples}\label{SEC:Examples}\smallskip

We begin with the vibrating string on a semi-axis, cf.~Jacob, Wegner \cite[Example 6.3]{JW2018} for a result on possibly non-contractive semigroups and Jacob, Kaiser \cite[Example 5.2.3]{JK} for the generation of contraction semigroups. Observe that in the latter articles different, and in both cases more restrictive, conditions on the maps $\rho$ and $T$ are required than we need below.

\smallskip

\begin{ex} \textbf{(Vibrating string)} Consider an undamped vibrating string of infinite length, i.e.,
\begin{equation}\label{VIB-1}
\frac{\partial^2w}{\partial{}^2t}(\xi,t)=-\frac{1}{\rho(\xi)}\frac{\partial}{\partial\xi}\Bigl(T(\xi)\frac{\partial{}w}{\partial\xi}(\xi,t)\Bigr)
\end{equation}
where $\xi\in[0,\infty)$ is the spatial variable, $w(\xi,t)$ is the vertical displacement of the string at place $\xi$ and time $t$, $T(\xi)>0$ is Young's modulus of the string, and $\rho(\xi)>0$ is the mass density. Both may vary along the string in a way that $\mathcal{H}$ as we define it below is bounded and satisfies \eqref{STANDARD-ASS}. We choose the momentum $x_1=\rho\frac{\partial w}{\partial t}$ and the strain $x_2=\frac{\partial w}{\partial \xi}$ as the state variables. Then, \eqref{VIB-1} can be written as
\begin{equation}\label{VIB-2}
\frac{\partial}{\partial t}\begin{bmatrix}x_1(\xi,t)\\x_2(\xi,t)\end{bmatrix}=\begin{bmatrix}0 &1\\1&0\end{bmatrix}\frac{\partial}{\partial\xi}\Bigl(\begin{bmatrix}\rho(\xi)^{-1} & 0\\0&T(\xi)\end{bmatrix}\begin{bmatrix}x_1(\xi,t)\\x_2(\xi,t)\end{bmatrix}\Bigr)
\end{equation}
which is of the form considered in Theorem \ref{GEN-SEMI-AXIS} if we put
$$
P_1:=\begin{bmatrix}0 & 1\\1&0\end{bmatrix},\;\;P_0:=\begin{bmatrix}0 & 0\\0&0\end{bmatrix},\,\text{ and }\;\mathcal{H}(\xi):=\begin{bmatrix}\rho(\xi)^{-1} & 0\\0&T(\xi)\end{bmatrix}.
$$
We decompose $P_1=S^{\star}\Delta S$ with
$$
S=S^{\star}=\frac{1}{\sqrt{2}}\begin{bmatrix}1 & 1\\1&-1\end{bmatrix}\;\text{ and }\;\Delta=\begin{bmatrix}1 & 0\\0&-1\end{bmatrix}.
$$
In particular, we have $n_+=n_-=1$, $\Lambda$ and $\Theta$ are $1\!\!\times\!\!1$-matrices with entries $\lambda_1=\theta_1=1$. The map $T\colon\mathcal{G}_2\rightarrow\mathcal{G}_1$ in Theorem \ref{GEN-SEMI-AXIS} is thus given by multiplication with a number $t\in\mathbb{C}$ and is contractive if and only if $|t|\leqslant1$. If we follow the notation in \eqref{VIB-1}, the boundary condition of Theorem \ref{GEN-SEMI-AXIS} reads as
$$
t\cdot\pr_-(S\mathcal{H}x(0))=\pr_+(S\mathcal{H}x(0))
$$
which we can rewrite this in the form $W_B(\mathcal{H}x)(0)=0$ given by the matrix $W_B=[t-1\;\;t+1]\in\mathbb{C}^{1\times2}$. Notice that \cite[Example 5.2.3]{JK} yields exactly the same condition, but requires that $\rho$, $\rho^{-1}$, $T$ and $T^{-1}\in\operatorname{L}^{\infty}(0,\infty)$---whereas we need only that $T$ and $\rho^{-1}$ are bounded. Using the results of \cite[Example 6.3]{JW2018} we can cover also cases where $T$ and $\rho^{-1}$ need not to be bounded and get a (possibly non-contractive) $\Cnull$-semigroup with the boundary condition $[w_1\;\;w_2](\mathcal{H}x)(0)=0$ if and only if $w_1(T(0)/\rho(0))^{1/2}\not=w_2T(0)$. For this we need however to assume that $\rho$ and $T$ are continuously differentiable and that certain other technical conditions hold, see \cite[Theorem 4.10]{JW2018}.
\end{ex}

\smallskip

\begin{rem} We point out that the techniques developed in the current article also allow to treat unbounded Hamiltonians in the context of Example \ref{VIB-1}. For this, however, the boundary system needs to be changed in that a \textquotedblleft{}point evaluation at infinity\textquotedblright{} has to be added in order to be able to apply Proposition \ref{thm:P00-1}.
\end{rem}

\smallskip

Next we study the coupled transport equation on an infinite tree in which every edge is an interval of length one, compare also \cite[Example 6.2]{JW2018} and \cite[Example 5.2]{JK} where a different method is applied.

\smallskip

\begin{ex} \textbf{(Transport on a tree)} Let $\Gamma$ be an infinite, complete binary tree where we use the following notation to enumerate the edges.

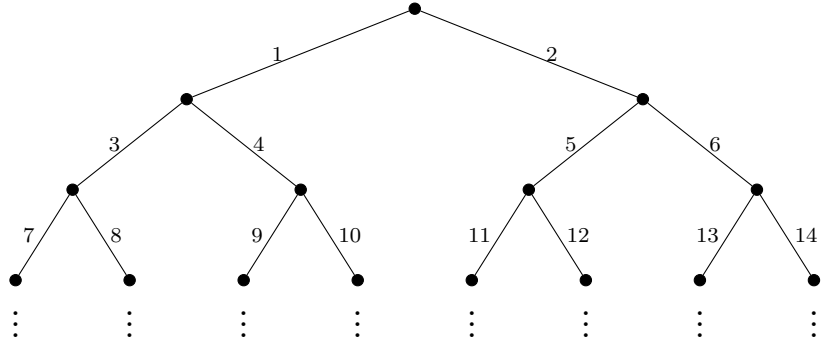
\begin{center}
\begin{tikzpicture}[scale=1.5,font=\footnotesize]
% Specify spacing for each level of the tree
\tikzstyle{level 1}=[level distance=8mm,sibling distance=40mm]
\tikzstyle{level 2}=[level distance=8mm,sibling distance=20mm]
\tikzstyle{level 3}=[level distance=8mm,sibling distance=10mm]

\tikzset{
% Two node styles for game trees: solid and hollow
solid node/.style={circle,draw,inner sep=1.5,fill=black},
hollow node/.style={circle,draw,inner sep=1.5}
}

% The Tree
\node(0)[solid node,label=above:{}]{}

child{node(1)[solid node,label=right:{}]{}
child{node[solid node,label=right:{}]{}
child{node[solid node,label=below:{\huge$\vdots$}]{} edge from parent node[left]{$7$}}
child{node[solid node,label=below:{\huge$\vdots$}]{} edge from parent node[right]{$8$}}
 edge from parent node[left]{$3$}
}
child{node[solid node,label=right:{}]{}
child{node[solid node,label=below:{\huge$\vdots$}]{} edge from parent node[left]{$9$}}
child{node[solid node,label=below:{\huge$\vdots$}]{} edge from parent node[right]{$10$}}
 edge from parent node[right]{$4$}}
edge from parent node[left,xshift=-3]{$1$}
}
child{node(2)[solid node,label=right:{}]{}
child{node[solid node,label=right:{}]{}
child{node[solid node,label=below:{\huge$\vdots$}]{} edge from parent node[left]{$11$}}
child{node[solid node,label=below:{\huge$\vdots$}]{} edge from parent node[right]{$12$}}
 edge from parent node[left]{$5$}}
child{node[solid node,label=right:{}]{}
child{node[solid node,label=below:{\huge$\vdots$}]{} edge from parent node[left]{$13$}}
child{node[solid node,label=below:{\huge$\vdots$}]{} edge from parent node[right]{$14$}}
edge from parent node[right]{$6$}}
edge from parent node[right,xshift=3]{$2$}
};

\end{tikzpicture}
\end{center}
On every edge $e\in E=\{1,2,3\dots\}$ we want to consider the transport equation
$$
\frac{\partial x}{\partial t}(\xi,t)=\frac{\partial x}{\partial \xi}(\xi,t) \,\text{ with }\, t\geqslant0,\;\xi\in[0,1]
$$
where we want that transport happens downwards on all edges. In order to fit this into the framework of port-Hamiltonian operators we put $d=1$, $P_1=1$ and $P_0=0$. That is $n_+=1$ and $n_-=0$. For $e\in E=\{1,2,3\dots\}$ we put $a(e)=0$, $b(e)=1$ and we define $\mathcal{H}_e(\xi)=1$ for $\xi\in[0,1]$. The direction of transport we still need to built into the boundary conditions: If we for instance look at the vertex that connects the edges $1$, $3$ and $4$ then the boundary condition should be such that it relates what comes in from edge $1$ with what goes out into $3$ and $4$. With $\alpha$, $\beta\in\mathbb{R}$ we can for instance put
\begin{equation}\label{TREE}
x_1(0)=\textstyle \alpha x_3(1)+\beta x_4(1),
\end{equation}
since the transport on $[0,1]$ goes from the right to the left. To apply Theorem \ref{super-thm} we need to formulate the boundary conditions via
$$
T(x_1(1),x_2(1),x_3(1),\dots)=(x_1(0),x_2(0),x_3(0),\dots)
$$
where $T\colon\ell^2\rightarrow\ell^2$ is a contraction. If we want to follow the pattern \eqref{TREE}, then we can put
$$
T((z_i)_{i\in\mathbb{N}}):=(\alpha z_{2i+1}+\beta z_{2i+2})_{i\in\mathbb{N}}
$$
which can for instance easily checked to be a contraction if $0\leqslant\alpha+\beta\leqslant1/2$.

\end{ex}

\smallskip

Next we discuss an example where the edge length is not bounded away from zero. 

\begin{ex} \textbf{(Transport on a star graph with edge length tending to zero)}\label{STAR} We consider the graph $\Gamma=(V,E)$ given by  $V=\{0, 1, \frac{1}{2}, \frac{1}{3}, \frac{1}{4},\dots \}$ and $E=\{ (0,\frac{1}{k})\:;\:k\in\{1,2,3,\dots\}\}$.

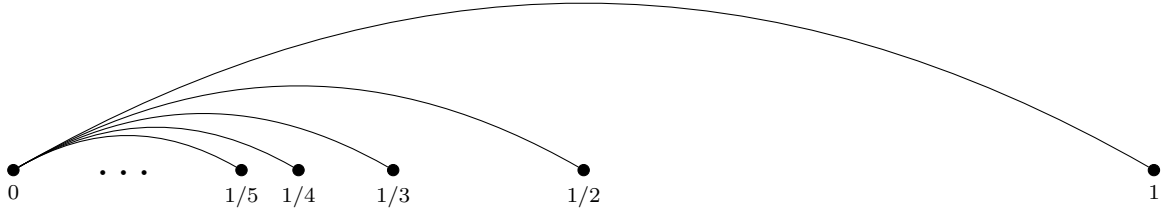
\begin{center}
\begin{tikzpicture}[scale=1.5,font=\footnotesize]
% Specify spacing for each level of the tree
\tikzstyle{level 1}=[level distance=15mm,sibling distance=40mm]
\tikzstyle{level 2}=[level distance=15mm,sibling distance=20mm]
\tikzstyle{level 3}=[level distance=15mm,sibling distance=10mm]

\tikzset{
% Two node styles for game trees: solid and hollow
solid node/.style={circle,draw,inner sep=1.5,fill=black},
hollow node/.style={circle,draw,inner sep=1.5}
}

% The star
\node(0)[solid node, label=below:{$0$}] at(0, 0) {};
\node(1)[solid node, label=below:{$1/5$}] at (2.0, 0) {};
\node(2)[solid node, label=below:{$1/4$}] at (2.5, 0) {};
\node(3)[solid node, label=below:{$1/3$}] at (3.33, 0) {};
\node(4)[solid node, label=below:{$1/2$}] at (5, 0) {};
\node(5)[solid node, label=below:{$1$}] at (10, 0){};

\node(6)[] at(1, -0.03) {\huge$\cdots$};

  \path[]
(0) edge [bend left]  (1)

(0) edge [bend left]  (2)
(0) edge [bend left]  (3)
(0) edge [bend left]  (4)
(0) edge [bend left]  (5);
\end{tikzpicture}
\end{center}

We put $a((0,\frac{1}{k}))=0$ and $b((0,\frac{1}{k}))=\frac{1}{k}$ for $k\geqslant1$. We further put $d=1$, $P_1=-1$, $P_0=0$ and $\mathcal{H}_e=1$ for each $e\in E$. That is we have on each edge a transport equation and transport happens from left to the right. We select as a boundary condition that $x_e(0)=0$ should hold on all edges. That is we consider the operator
$$
A\colon D(A)\subseteq\bigoplus_{k=1}^{\infty}\Ls^2(0,{\textstyle\frac{1}{k}})\longrightarrow \bigoplus_{k=1}^{\infty}\Ls^2(0,{\textstyle\frac{1}{k}})
$$
given by
\[
    A x = (-x'_k)_{k=1,2,\dots}\;\text{ for } \; x\in D(A)=\big\{ x\in \bigoplus_{k=1}^{\infty}\Ls^2(0,{\textstyle\frac{1}{k}})\:;\: (x'_k)_{k=1,2,\dots}\in\bigoplus_{k=1}^{\infty}\Ls^2(0,{\textstyle\frac{1}{k}}),\;x_k(0)=0\big\}
\]
where we use the simplification $x_k:=x_{\{0,1/k\}}$ for $k\geqslant1$. We put $P_n\colon \bigoplus_{k=1}^{\infty}\Ls^2(0,\frac{1}{k})\rightarrow \bigoplus_{k=1}^{n}\Ls^2(0,\frac{1}{k})$ for $n\geqslant1$, $H:=A$ and observe that $P_nH\subseteq HP_n$ holds for $n=1,2,\dots$ in view of the boundary condition. By Theorem \ref{thm:lfw} we get that $A$ is maximal dissipative if and only if $AP_n$ is maximal dissipative for every $n\geqslant1$. But since we here work over a finite graph, we may apply Theorem \ref{super-thm} with $n_+=1$, $n_-=0$, $|E_{\ell}|=n$. The boundary condition that leads to a generator of a contraction semigroup, and thus to a maximal dissipative operator, is given by $T(x_1(1),\dots,x_n(1/n))=(x_1(0),\dots,x_n(0))$ with a contraction $T\colon\mathbb{C}^n\rightarrow\mathbb{C}^n$. If we here choose $T=0$ we are done and get that $A$ as defined above generates a contraction semigroup.
\end{ex}

\smallskip

The reason for involving the next example is twofold. On the one hand it demonstrates how to apply Theorem \ref{super-thm2} and at the same time it shows that proving the last condition (Theorem \ref{super-thm2}(v)) assumed on the graph might be a difficult thing in practice. 

\smallskip

\begin{ex} \textbf{(Transport on the unit interval---considered as a network)}\label{LINE} Let $V=\{ \frac1{k}\:;\: k\in \{1,2,3,\ldots\} \}$ and $E=\{ (\frac1{k+1},\frac1{k})\:;\: k\in \{1,2,3,\ldots \}\}$.

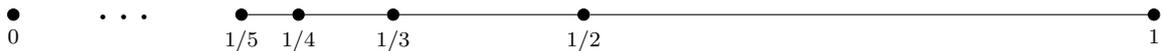
\smallskip

\begin{center}
\begin{tikzpicture}[scale=1.5,font=\footnotesize]
% Specify spacing for each level of the tree
\tikzstyle{level 1}=[level distance=15mm,sibling distance=40mm]
\tikzstyle{level 2}=[level distance=15mm,sibling distance=20mm]
\tikzstyle{level 3}=[level distance=15mm,sibling distance=10mm]

\tikzset{
% Two node styles for game trees: solid and hollow
solid node/.style={circle,draw,inner sep=1.5,fill=black},
hollow node/.style={circle,draw,inner sep=1.5}
}

% The line
\node(0)[solid node, label=below:{$0$}] at(0, 0) {};
\node(1)[solid node, label=below:{$1/5$}] at (2.0, 0) {};
\node(2)[solid node, label=below:{$1/4$}] at (2.5, 0) {};
\node(3)[solid node, label=below:{$1/3$}] at (3.33, 0) {};
\node(4)[solid node, label=below:{$1/2$}] at (5, 0) {};
\node(5)[solid node, label=below:{$1$}] at (10, 0){};

\node(6)[] at(1, -0.03) {\huge$\cdots$};

  \path[]

(1) edge   (2)
(2) edge   (3)
(3) edge   (4)
(4) edge   (5);
\end{tikzpicture}
\end{center}

\smallskip

We put $a((\frac1{k+1},\frac1{k}))=\frac1{k+1}$ and $b((\frac{1}{k+1},\frac1{k}))=\frac1{k}$.  Consider the operator
\[
    B x = (x'_e)_{e\in E},
\]
where
\[
x\in D(B)=\{ x\in X\:;\:  (x'_e)_{e\in E}\in X,\, x_{(\frac1{k+1},\frac1{k})}\bigl(\textstyle{\frac1{k+1}}\bigr)=x_{({\frac1{k+2}},\frac1{k+1})}\bigl({\frac1{k+2}}\bigr),\, x_{(\frac12,1)}(1)=0\}.
\]
We use Theorem \ref{super-thm2} applied to $A=B$ and $P_n = \chi_{E_n}$ with $E_n =\{ (\frac1{k+2},\frac1{k+1})\:;\: k\in\{0,\dots,n\}\}$. Note that Assumptions in Theorem are easily checked. Moreover, with the boundary system for uniformly positive edge lengths it is easy to see that $A_n$  is maximal dissipative. Hence, $B$ generates a contraction semi-group as expected.
\end{ex}

We conclude this article by pointing out that the previous example suggests how more complicated examples with arbitrarily small edges can be constructed: Given a graph $\Gamma_1$ where maximal dissipative extensions are already characterized by boundary conditions, and an arbitrary countably infinite graph $\Gamma_2$, the method of Example \ref{LINE} allows to treat the new graph that arises by assigning to each vertex of $\Gamma_2$ a copy of $\Gamma_1$ but with all edges of $\Gamma_1$ being scaled down such that the overall edge lengths accumulate zero. We illustrate the latter idea with the following two pictures.

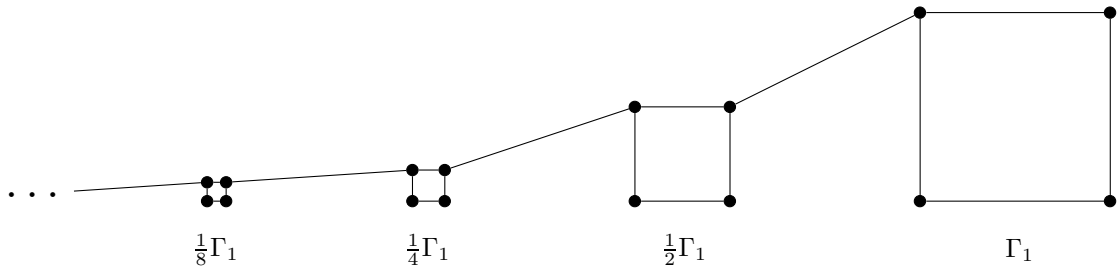
\begin{center}
\begin{tikzpicture}[scale=1.25,font=\footnotesize]

\tikzset{
% Two node styles for game trees: solid and hollow
solid node/.style={circle,draw,inner sep=1.5,fill=black},
hollow node/.style={circle,draw,inner sep=1.5}
}

% The line

\node(0)[solid node, label=below:{}] at (8, 0){};
\node(1)[solid node, label=below:{}] at (10, 0){};
\node(2)[solid node, label=below:{}] at (10, 2){};
\node(3)[solid node, label=below:{}] at (8, 2){};

\node[label=below:{\hspace{-10pt}\normalsize$\Gamma_1$}] at (9, -0.2){};

\node(4)[solid node, label=below:{}] at (5, 0){};
\node(5)[solid node, label=below:{}] at (6, 0){};
\node(6)[solid node, label=below:{}] at (6, 1){};
\node(7)[solid node, label=below:{}] at (5, 1){};

\node[label=below:{\hspace{-12pt}\normalsize$\frac{1}{2}\Gamma_1$}] at (5.5, -0.15){};

\node(8)[solid node, label=below:{}] at (3, 0){};
\node(9)[solid node, label=below:{}] at (2.66, 0){};
\node(10)[solid node, label=below:{}] at (3, 0.33){};
\node(11)[solid node, label=below:{}] at (2.66, 0.33){};

\node[label=below:{\hspace{-12pt}\normalsize$\frac{1}{4}\Gamma_1$}] at (2.8, -0.15){};

\node(12)[solid node, label=below:{}] at (0.7, 0){};
\node(13)[solid node, label=below:{}] at (0.5, 0){};
\node(14)[solid node, label=below:{}] at (0.7, 0.2){};
\node(15)[solid node, label=below:{}] at (0.5, 0.2){};

\node[label=below:{\hspace{-14pt}\normalsize$\frac{1}{8}\Gamma_1$}] at (0.6, -0.15){};

\node(16)[ label=below:{}] at (-1, 0.1){};

\node(100)[] at(-1.3, 0.06) {\huge$\cdots$};

  \path[]
(0) edge   (1)
(1) edge   (2)
(2) edge   (3)
(3) edge   (0);

  \path[]
(3) edge   (6);

  \path[]
(4) edge   (5)
(5) edge   (6)
(6) edge   (7)
(7) edge   (4);

  \path[]
(7) edge   (10);

  \path[]
(8) edge   (9)
(9) edge   (11)
(10) edge   (11)
(10) edge   (8);

  \path[]
(11) edge   (14);

  \path[]
(14) edge   (15)
(15) edge   (13)
(13) edge   (12)
(12) edge   (14);

  \path[]
(15) edge   (16);

\end{tikzpicture}
\end{center}

\bigskip

In the first picture $\Gamma_1$ is the 2-dimensional cube and in the second one the line with three vertices. The graph $\Gamma_2$ is in both cases a line graph which is infinite in one direction with uniform edge lengths.

\bigskip\bigskip\bigskip

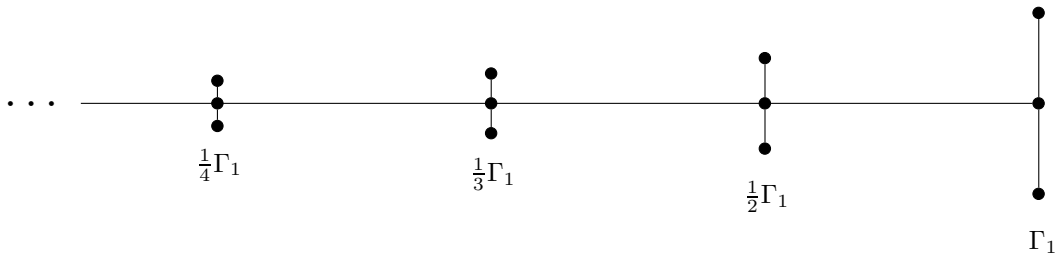
\begin{center}
\begin{tikzpicture}[scale=1.2,font=\footnotesize]

\tikzset{
% Two node styles for game trees: solid and hollow
solid node/.style={circle,draw,inner sep=1.5,fill=black},
hollow node/.style={circle,draw,inner sep=1.5}
}

% The line

\node(0)[solid node, label=below:{}] at (10, 1){};
\node(1)[solid node, label=below:{}] at (10, 0){};
\node(2)[solid node, label=below:{}] at (10, -1){};

\node[label=below:{\hspace{-10pt}\normalsize$\Gamma_1$}] at (10, -1.2){};

  \path[]
(0) edge   (1)
(1) edge   (2);

\node(3)[solid node, label=below:{}] at (7, 0.5){};
\node(4)[solid node, label=below:{}] at (7, 0){};
\node(5)[solid node, label=below:{}] at (7, -0.5){};

\node[label=below:{\hspace{-12pt}\normalsize$\frac{1}{2}\Gamma_1$}] at (7, -0.65){};

  \path[]
(3) edge   (4)
(4) edge   (5);

\node(6)[solid node, label=below:{}] at (4, 0.33){};
\node(7)[solid node, label=below:{}] at (4, 0){};
\node(8)[solid node, label=below:{}] at (4, -0.33){};

\node[label=below:{\hspace{-12pt}\normalsize$\frac{1}{3}\Gamma_1$}] at (4, -0.4){};
  \path[]
(6) edge   (7)
(7) edge   (8);

\node(9)[solid node, label=below:{}] at (1, 0.25){};
\node(10)[solid node, label=below:{}] at (1, 0){};
\node(11)[solid node, label=below:{}] at (1, -0.25){};

\node[label=below:{\hspace{-12pt}\normalsize$\frac{1}{4}\Gamma_1$}] at (1, -0.3){};

  \path[]
(9) edge   (10)
(10) edge   (11);

\node(12)[ label=below:{}] at (-0.6, 0){};

\path[]
(1) edge   (12);

\node(100)[] at(-1, -0.02) {\huge$\cdots$};

\end{tikzpicture}
\end{center}

\bigskip

\footnotesize

\red{{\sc Acknowledgements. }The authors would like to thank the anonymous referee for their careful work and for several remarks that helped to improve the article.}

\normalsize

\normalsize

\bibliographystyle{amsplain}

\end{document}